\documentclass[10pt, xcolor = {usenames,dvipsnames}]{article}

\usepackage[english]{babel}
\usepackage[utf8]{inputenc} 
\usepackage{amsfonts}
\usepackage{dsfont}
\usepackage{setspace}

\usepackage[square,sort,comma,numbers]{natbib}
\bibpunct{[}{]}{;}{a}{,}{,}
\usepackage[hidelinks]{hyperref}

\usepackage{tikz}
\usetikzlibrary{calc}

\usepackage{color,soul}

\usepackage{mathtools}
\usepackage{amsthm}
\usepackage{amssymb}
\usepackage{amsmath}
\usepackage{needspace}
\usepackage{etoolbox}
\usepackage{lipsum}
\usepackage{authblk}
\usepackage{mathrsfs}
\usepackage{bbold}
\usepackage{graphicx}
\usepackage{subcaption}
\usepackage{skeldoc}            %

\newcounter{theo}
\newtheorem{lmm}[theo]{Lemma}
\newtheorem{thr}{Theorem}

\newtheorem{crl}[theo]{Corollary}
\newtheorem{prp}[theo]{Proposition}

\newtheorem{defn}[theo]{Definition}
\newtheorem{rmrk}[theo]{Remark}

\newcommand{\R}{\mathbb{R}}
\newcommand{\Compl}{\mathbb{C}}
\newcommand{\N}{\mathbb{N}}
\newcommand{\Z}{\mathbb{Z}}
\newcommand{\T}{\mathbb{T}}
\newcommand{\D}{\mathbb{D}}
\newcommand{\Prob}{\mathbb{P}}
\newcommand{\E}{\mathbb{E}}
\newcommand{\eps}{\varepsilon}

\newcommand\CnjCl[1]{[#1]}
\newcommand{\one}{\mathds{1}}

\newcommand{\Diff}{\mathrm{Diff}}

\newcommand{\diff}{\mathfrak{diff}}

\newcommand{\SL}{\mathrm{PSL}}
\newcommand{\nSL}{\mathrm{SL}}
\newcommand{\tr}{\mathrm{tr}}

\newcommand{\FMeas}[1]{\mathscr{M}_{#1}}
\newcommand{\FMeasSL}[1]{\widetilde{\mathscr{M}}_{#1}}
\newcommand{\Haar}{\nu_H}
\newcommand{\TechMeas}{\mathcal{P}}
\newcommand{\ang}{a}
\newcommand{\dense}[1]{{D}^{#1}}
\newcommand{\Schw}{\mathcal{S}}
\newcommand{\Wiener}{\mathcal{W}}
\newcommand{\A}{\mathsf{P}}

\renewcommand*{\d}{\mathop{}\!\mathrm{d}}
\newcommand{\WS}[3]{\mathcal{B}_{#1}^{\, #2, #3}} %
\newcommand{\Cfree}{C_{0,\text{free}}}

\newcommand{\measOrb}[2]{\mathscr{L}_{#1}^{#2}}
\newcommand{\measPin}[3]{\mathscr{L}_{#1}^{#2, #3}}

\newcommand{\PartF}[2]{\mcl{Z}_{#1}^{\, #2}}
\newcommand{\PartFSL}{\mcl{Z}}

\makeatletter
\def\obs#1#2#3{\def\temp@expr{\big(#1; #2, #3\big)}\mathcal{O}\obs@}
\def\obs@{%
	\@ifnextchar{_}{\obs@sub}{
	\@ifnextchar{^}{\obs@sup}{\temp@expr}}}
\def\obs@sub#1#2{_{#2}\obs@}
\def\obs@sup#1#2{^{\, #2}\obs@}
\makeatother

\newcommand{\RR}{\mathbb{R}}

\newcommand{\NN}{\mathbb{N}}

\newcommand{\dd}{\mathrm{d}}
\newcommand{\de}{\partial}

\newcommand{\mcl}[1]{\mathcal{#1}}
\newcommand{\mfk}[1]{\mathfrak{#1}}
\newcommand{\mrm}[1]{\mathrm{#1}}

\DeclarePairedDelimiterX{\inp}[2]{\langle}{\rangle}{#1, #2}

\reversemarginpar
\usepackage[final
]{showlabels}

\usepackage{ragged2e}

\title{Probabilistic Definition of the Schwarzian Field Theory}
\author{Roland Bauerschmidt\footnote{Courant Institute of Mathematical Sciences, New York University. E-mail: \url{bauerschmidt@cims.nyu.edu}.}\qquad
  Ilya Losev\footnote{DPMMS, University of Cambridge. E-mail: \url{il320@cam.ac.uk}.}\qquad
  Peter Wildemann\footnote{Université de Genève. E-mail: \url{peter.wildemann@unige.ch}.}}
\date{June 17, 2025}

\begin{document}

\maketitle

\abstract{We provide mathematical foundations for the Schwarzian Field Theory as a finite Borel measure
  on $\mathrm{Diff}^1(\mathbb{T})/\mathrm{PSL}(2,\mathbb{R})$, a quotient of the space of circle reparametrisations.
  The measure is defined by a natural change of variables formula,
  which we show uniquely characterises it. We further compute its partition function (total mass) from this change of variable formula.
  The existence of the measure then
  follows from an explicit construction involving a nonlinear transformation of a Brownian Bridge, proposed by Belokurov--Shavgulidze.
  In two companion papers by Losev, the predicted exact cross-ratio correlation functions for non-crossing Wilson lines and the large deviations are derived from this measure.
}

\section{Introduction and main results}\label{sect_main_part}

\subsection{Introduction}
\label{sec:intro}

The Schwarzian Field Theory arose in physics in the study of the Sachdev--Ye--Kitaev (SYK) random matrix model,
see \citep{MaldacenaStanford} and \citep{KitaevJosephine} for introductions.
It is also interesting as an example of a nonlinear but, at least formally, exactly  solvable Euclidean field theory in $0+1$ dimensions,
with nonlinearly realised symmetry, connections to Liouville Field Theory,
infinite dimensional symplectic geometry, two-dimensional topological Yang-Mills theory, and other topics; references are given below.
Perhaps most intriguingly, in physics, it has been proposed as the holographic dual to Jackiw--Teitelboim (JT) gravity on the Poincar\'e disk, see for example \citep{SaadShenkerStanford2019,MR4058854,JT_Wilson_line}.

The goal of this paper is to provide mathematical foundations for the Schwarzian Field Theory
as a finite Borel measure on $\Diff^1(\T)/\SL(2,\R)$, the quotient of the space of circle reparametrisations.
We address uniqueness, existence, and computation of its total mass (partition function).
This measure should formally be given by the ill-defined density (see \cite[(1.1)]{StanfordWittenFermionicLocalization})
\begin{equation}\label{eq:1}
\d\FMeas{\sigma^2}\big(\phi\big) = 
\exp\left\{+\frac{1}{\sigma^2}\int_{\T} 
\left[\Schw_\phi(\tau)+2\pi^2\phi'^{\, 2}(\tau)\right] \d\tau \right\}
\frac{\prod_{\tau \in \T}\frac{\d\phi(\tau)}{\phi'(\tau)}}{\SL(2, \R)} ,
\end{equation} 
where $\Schw_\phi(\tau)$ is the Schwarzian derivative of $\phi$ defined by
\begin{equation} \label{eq:Schw-def}
  \Schw_\phi(\tau) = \Schw(\phi, \tau) = \left(\frac{\phi''(\tau)}{\phi'(\tau)}\right)'- \frac{1}{2} \left(\frac{\phi''(\tau)}{\phi'(\tau)}\right)^2,
\end{equation}
and the measure $\FMeas{\sigma^2}$ in \eqref{eq:1} should be supported on the topological space $\Diff^1(\T)/\SL(2, \R)$, where
$\T = [0,1]/\{0 \sim 1\}$ is the unit circle parametrised by the angle $[0,1)$, and
$\Diff^1(\T)$ is the space of $C^1$ orientation-preserving diffeomorphisms of $\T$, see Section~\ref{sec:prelim}.
The $\SL(2,\R)$-action on $\Diff^1(\T)$ implicit in the quotient in \eqref{eq:1} is described in the next paragraph.
Heuristically, the formal density \eqref{eq:1} only depends on the orbit of this action and the quotient by $\SL(2,\R)$ therefore makes sense.

The $\SL(2,\R)$ action on $\Diff^1(\T)$ arises from post-compositions by conformal diffeomorphisms of the unit disk restricted to the boundary, which is identified with $\T$.
Explicitly,
it is instructive to map the circle to the real line and consider
$f(\tau) = \tan(\pi\phi(\tau)-\tfrac{\pi}{2})$ instead of $\phi(\tau)$, where $\tau \in \T$ remains on the circle.
In terms of this variable, the exponential in the formal density \eqref{eq:1} of the measure can be written as
\begin{equation} \label{eq:1-Schw}
  \exp\left\{+ \frac{1}{\sigma^2}\int_{\T} \Schw(\tan (\pi\phi-\tfrac{\pi}{2}),\tau) \, \d\tau \right\},
\end{equation}
again see Section~\ref{sec:prelim}, explaining the name Schwarzian Field Theory.
The bijection $\phi \in (0,1)\mapsto f = \tan(\pi\phi-\tfrac{\pi}{2}) \in \R$
is the restriction to the boundary
of the standard conformal map $z\mapsto i \frac{1+z}{1-z}$ from the unit disk in $\mathbb{C}$
to the upper half plane, on which $\SL(2,\R)$ acts
by the fractional linear transformation
\begin{equation}
  \label{eq:66}
  f\mapsto M \circ f = \frac{af+b}{cf+d}, \qquad M= \pm \left(\begin{smallmatrix} a & b\\ c &d \end{smallmatrix}\right) \in \SL(2,\R).
\end{equation}
Even though this $\SL(2,\R)$-action is by left composition, we will call it
the right-action on $\Diff^1(\T)$, following the discussion in \citep[Section~2.1]{StanfordWittenFermionicLocalization}, where this action is interpreted as an action on the inverse of $\phi$.
The Schwarzian is invariant under this $\SL(2,\R)$ action, i.e.,
\begin{equation}
  \Schw(M\circ f,\tau) = \Schw(f,\tau) \qquad \text{for any $M \in \SL(2,\R)$,}
\end{equation}
and it vanishes if and only if $f$ itself is a fractional linear transformation.
It can further be argued that $\prod_{\tau\in \T} \frac{\d\phi(\tau)}{\phi'(\tau)}$
in the formal expression \eqref{eq:1} should be the non-existent Haar measure on the group $\Diff^1(\T)$,
and in particular be invariant under the $\SL(2,\R)$ action, see \citep[Section~2.2]{StanfordWittenFermionicLocalization}.
We abuse the notation throughout the paper and use the same symbol for $\phi\in \Diff^1(\T)$ and its conjugacy class $\phi \in \Diff^1(\T)/\SL(2,\R)$ when only the latter is relevant.

\bigskip

In \citep{StanfordWittenFermionicLocalization} the partition function of the Schwarzian Field Theory  was computed by a formal application
of the Duistermaat--Heckman theorem on the infinite dimensional space $\Diff^1(\T)/\SL(2,\R)$,
and in \citep{ConformalBootstrap} the natural cross-ratio correlation functions of the Schwarzian Field Theory
were obtained via an application of the conformal bootstrap and the DOZZ formula to a degenerate limit of the two-dimensional Liouville Field Theory.

More probabilistic perspectives on the Schwarzian Field Theory 
closely related to our results and discussed further below
were proposed in \citep{BelokurovShavgulidzeExactSolutionSchwarz},
and we also refer to \citep{AlekseevBosonization} and references for another related but more geometric point of view
in the context of Virasoro coadjoint orbits which the Schwarzian Field Theory is an instance of, see below.

\subsection{Main results for Schwarzian Field Theory}
\label{sec:intro-schwarzian}

As an important ingredient of the characterisation and construction of the Schwarzian measure \eqref{eq:1},
we first consider an unquotiented version of the measure,
which should formally be given by
\begin{equation}\label{eq:1-unquotiented}
  \d\FMeasSL{\sigma^2}\big(\phi\big) = 
  \exp\left\{+\frac{1}{\sigma^2}\int_{\T} 
    \Schw(\tan (\pi\phi-\tfrac{\pi}{2}),\tau) \,
    \d\tau \right\}
  \prod_{\tau \in \T}\frac{\d\phi(\tau)}{\phi'(\tau)} ,
\end{equation}
without quotienting by $\SL(2,\R)$.
Such a measure should have transformation behaviour
consistent with the Schwarzian chain rule:
\begin{equation} \label{e:Schw-chain}
  \Schw(g\circ f, \tau) = \Schw(f,\tau) + \Schw(g,f(\tau)) f'(\tau)^2 ,
\end{equation}
valid for $f,g\in C^3$. %
Checking that $\Schw(\tan(\pi\tau-\tfrac{\pi}{2}), \tau)= 2\pi^2$, it follows that
\begin{equation} \label{e:Schw-tan-chain}
  \Schw(\tan(\pi(\psi\circ\phi)-\tfrac{\pi}{2}), \tau)
  = \Schw(\tan(\pi\phi-\tfrac{\pi}{2}),\tau) + \Big(\Schw(\tan(\pi\psi-\tfrac{\pi}{2}), \phi(\tau)) -2\pi^2\Big)\phi'(\tau)^2
  ,
\end{equation}
see Section~\ref{sec:prelim}.
Thus changing variables  in \eqref{eq:1-unquotiented} from $\phi$ to $\psi\circ\phi$  for a fixed $\psi\in \Diff^3(\T)$,
we expect that the Radon--Nikodym derivative of the measures should be given by the exponential of the second term on the right-hand side of~\eqref{e:Schw-tan-chain}.
We take the transformation behaviour as the definition of the unquotiented Schwarzian measure and show that it indeed characterises the measure uniquely.

\begin{thr}\label{thrMeasureSLInv}
  There is a unique (up to a multiplicative constant) $\SL(2,\R)$-invariant Borel measure $\FMeasSL{\sigma^2}$ supported on $\Diff^1(\T)$
  that satisfies the expected change of variables formula
  \begin{equation}  \label{eqDiffeoMeasureChange}
    \frac{\d \psi^{\ast}\!\FMeasSL{\sigma^2}(\phi)}{\d\FMeasSL{\sigma^2}(\phi)} =
    \frac{\d \FMeasSL{\sigma^2}(\psi\circ \phi)}{\d\FMeasSL{\sigma^2}(\phi)} = 
    \exp\left\{
      \frac{1}{\sigma^2}\int_{\T}\Big[ \Schw(\tan(\pi\psi-\tfrac{\pi}{2}),\phi(\tau))-2\pi^2 \Big]\, \phi'(\tau)^2 \d \tau \right\}
    ,
  \end{equation}
    for  any $\psi\in \Diff^3(\T)$, and has 
  a quotient $\FMeas{\sigma^2} = \FMeasSL{\sigma^{2}}/\SL(2,\R)$ that is a finite Borel measure
  on $\Diff^1(\T)/\SL(2,\R)$.
  This measure has total mass
  \begin{equation} \label{e:Zsigma2}
    \PartFSL(\sigma^2)
    =
    \left(\frac{2\pi}{\sigma^2}\right)^{3/2} \exp\left(\frac{2\pi^2}{\sigma^2}\right).
  \end{equation}
  when we impose a natural normalisation, consistent with the $\zeta$-regularisation for the partition function of a Brownian Bridge, see Section~\ref{sec:BB}.
\end{thr}

More precisely, in the statement of the theorem, %
diffeomorphisms $\psi \in \Diff^3(\T)$ act on $\phi \in \Diff^1(\T)$ by post-composition, i.e.\ $\phi \mapsto \psi \circ \phi \in \Diff^1(\T)$,
and under this action the measure transforms according to the pullback $\FMeasSL{\sigma^2} \mapsto \psi^{\ast}\!\FMeasSL{\sigma^2}$:
\begin{equation}
  \psi^{\ast}\!\FMeasSL{\sigma^2}(A)
  = \psi^{-1}_{\ast}\!\FMeasSL{\sigma^2}(A)
  = \FMeasSL{\sigma^2}(\psi \circ A),
\end{equation} 
where $\psi\circ A \coloneqq \left\{\psi \circ \phi\, \big| \, \phi \in A\right\}$.
For $\psi$ corresponding to the action of $\SL(2,\R)$ on the circle $\T$ (see Section~\ref{sec:prelim})
one has $\Schw(\tan(\pi\psi-\tfrac{\pi}{2}),\phi)
=2\pi^2$, i.e.\ the second term in \eqref{e:Schw-tan-chain} vanishes and the measure is invariant.

The change of variables formula \eqref{eqDiffeoMeasureChange} should be regarded
as the analogue of the Cameron--Martin formula for Brownian motion under the change of variable from $B$ to $B+h$.

The unique characterisation of the measure relies on the computation of the partition function, i.e.\ total mass of the measure.
Perhaps surprisingly, we obtain the  latter using only the transformation formula \eqref{eqDiffeoMeasureChange},
or more precisely its deformation to more general Virasoro coadjoint orbit measures,
but does  not use the explicit realisation of the measure,
see Section~\ref{sec:orbital}--\ref{sec:schwarzian}.

The existence and explicit realisation of the measure \eqref{eq:1-unquotiented}, detailed in  Section~\ref{sec:construction}, follows the proposal of
\citep{BelokurovShavgulidzeExactSolutionSchwarz, BelokurovShavgulidzeCorrelationFunctionsSchwarz}.
Namely, one can construct a measure as an appropriate non-linear transformation of a (reweighted) Brownian bridge $\xi$ and a uniform shift $\Theta\in\T$, %
and then verify that the obtained measure satisfies %
the change of variables formula \eqref{eqDiffeoMeasureChange} as well as the integrability condition required for Theorem~\ref{thrMeasureSLInv}.
Essentially, using the parametrisation (which we refer to as the \emph{Malliavin--Shavgulidze map})
\begin{equation} \label{e:phi-exp}
  \phi(\tau) = \Theta + \frac{\int_0^\tau e^{\xi(t)}  \d t}{\int_0^1 e^{\xi(t)}\d t} =  \Theta + \A_\xi(\tau),
\end{equation}
with $\xi\colon [0,1] \to \R$ and $\Theta \in \T$ a constant, one has
\begin{equation} \label{e:Schw-BB}
  - \Schw(\tan(\pi\phi-\tfrac{\pi}{2}),\tau) = \frac12 \xi'(\tau)^2 -\xi''(\tau) - 2\pi^2 \left(\frac{e^{\xi(\tau)}}{\int_0^1 e^{\xi(t)}\d t}\right)^2,
\end{equation}
and under this parametrisation, the ``reference measure'' is heuristically given by
\begin{equation}
  \prod_{\tau} \frac{\d\phi(\tau)}{\phi'(\tau)}
  = \d\Theta \prod_\tau \d\xi(\tau).
\end{equation}
Thus it is natural to interpret the measure with action given by \eqref{e:Schw-BB} in terms of
a reweighted Brownian Bridge, whose formal action is $\frac12 \int \xi'(\tau)^2 \, \d\tau$,
and a constant $\Theta$ (the zero mode), distributed according to the Lebesgue measure.
In Section~\ref{sec:construction}, we give a precise version of this construction and show that it indeed leads
to an infinite Borel measure $\FMeasSL{\sigma^2}$ on $\Diff^1(\T)$ which satisfies the conditions of
Theorem~\ref{thrMeasureSLInv}.

A convincing outline of the probabilistic construction of the Schwarzian measure was already given by \citep{BelokurovShavgulidzeExactSolutionSchwarz}.
Because their argument is not fully mathematically rigorous,
we complete the details of their construction and provide a result that can be applied to Theorem~\ref{thrMeasureSLInv} and that also
provides the basis for subsequent probabilistic applications to correlation functions \citep{LosevCorr} and large deviations \citep{LosevLDP}
(see also Section~\ref{sec:literature} for closer comparison).

\begin{rmrk} \label{rk:spectraldensity}
The partition function can also be rewritten as
\begin{equation}\label{eq:8}
  \begin{aligned}
    \PartFSL(\sigma^2)
    \coloneqq \FMeas{\sigma^2}\Big(\Diff^1(\T)/\SL(2, \R)\Big)
    &= 
      \left(\frac{2\pi}{\sigma^2}\right)^{3/2} \exp\left(\frac{2\pi^2}{\sigma^2}\right)\\
    &=
      \int_0^{\infty} \exp\left(-\frac{\sigma^2 k^2}{2}\right)\sinh(2\pi k) \, 2k \d k
    \\
    &=\int_0^{\infty} e^{-\sigma^2E} \sinh(2\pi \sqrt{2E}) \, 2\d E.
  \end{aligned}
\end{equation}
The second line in \eqref{eq:8} follows by integration by parts:
\begin{equation}
  \label{eq:9}
  \begin{aligned}
    \int_0^{\infty} \exp\left(-\frac{\sigma^2 k^2}{2}\right)\sinh(2\pi k)\, 2k \d k
    &= 
      \frac{4\pi}{\sigma^2}
      \int_0^{\infty} \exp\left(-\frac{\sigma^2 k^2}{2}\right)\cosh(2\pi k) \d k \\
    &= 
      \frac{2\pi}{\sigma^2}
      \int_{-\infty}^{\infty} \exp\left(-\frac{\sigma^2 k^2}{2}+2\pi k\right) \d k \\
    &=
      \left(\frac{2\pi}{\sigma}\right)^{3/2} \exp\left(\frac{2\pi^2}{\sigma^2}\right).
    \end{aligned}
\end{equation}
The third line then follows by changing variables to $E=\frac12 k^2$.

The right-hand side of \eqref{eq:8} has the form of a Laplace transform of a spectral density $\nu(E) = 2\sinh(2\pi\sqrt{2E})$.
It is expected that it approximates the (low temperature) thermal partition function $\E[\tr(e^{-\beta H})]$ of the SYK model
(with $\sigma^2$ corresponding to inverse temperature $\beta$).
For further discussion, see for example \citep[Section 2.4]{StanfordWittenFermionicLocalization}.
\end{rmrk}

\subsection{From Schwarzian theory to Virasoro orbital measures}

The Schwarzian Field Theory \eqref{eq:1} can be seen as an \emph{exceptional} representative ($\alpha=\pi$) of a wider class of \emph{orbital measures} with parameter $\alpha$.
These are motivated by the coadjoint action of $\Diff(\T)$ on the dual of the Virasoro algebra, $\mfk{vir}^\ast \cong (\diff(\T) \oplus \R)^\ast$.
The orbits of this action are referred to as \emph{Virasoro coadjoint orbits} and carry a natural symplectic structure.
This geometric viewpoint was emphasised in \citep{StanfordWittenFermionicLocalization} (see also \citep{WittenCoadjoint}),
and exploited there to predict the formula for the partition function via a formal application of the Duistermaat--Heckman localisation formula.

For our purposes, the $\alpha$-orbits with $\alpha\neq \pi$ provide a natural deformation of the Schwarzian measure that is complex analytic in $\alpha^2$.
As such, they are important in the study of the Schwarzian measure as a limit $\alpha \nearrow \pi$ in Section~\ref{sec:schwarzian}.
The \emph{unquotiented (Virasoro) $\alpha$-orbital measure} should be a measure on $\Diff^{1}(\T)$ described in terms of the formal density
\begin{equation}
  \label{eq:55}
  \d\measOrb{\sigma^2}{\alpha}(\phi)
  = 
  \exp\left\{+\frac{1}{\sigma^2}\int_{\T} 
    \left[\Schw_\phi(\tau)+2\alpha^2\phi'^{\, 2}(\tau)\right] \d\tau \right\}
  \prod_{\tau\in\T}\frac{\d\phi(\tau)}{\phi'(\tau)}
  .
\end{equation}
We are interested in $\alpha \in \RR_{\geq 0} \cup i\RR_{\geq 0}$ and often regard them as a family in the real parameter $\alpha^{2} \in \RR$.
The unquotiented Schwarzian measure \eqref{eq:1-unquotiented} corresponds to $\FMeasSL{\sigma^2} =\measOrb{\sigma^2}{\pi}$.

As for the Schwarzian case discussed already around Theorem~\ref{thrMeasureSLInv},
the measure \eqref{eq:55} should transform under post-composition
in agreement with the Schwarzian chain rule \eqref{e:Schw-chain},
and we use the expected transformation behaviour as the defining property of orbital measures
and show that this definition indeed characterises the measure uniquely.

\begin{figure}
  \centering
  \includegraphics[width=0.89\linewidth]{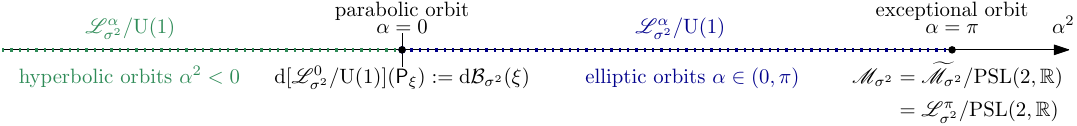}
  \captionsetup{width=0.89\linewidth,font=scriptsize}
  \caption{
  An illustration of the classification of the orbital measures.
    The coadjoint orbits of the Virasoro group are identified with quotients of $\Diff^{1}(\T)$ and the orbital measures are supported on those, after taking the quotient with respect to an appropriate invariance group.
    The orbits $\alpha = 0$ and $\alpha = \pi$ take a distinguished role in our analysis.
    For $\alpha = 0$ the orbital measure can be seen as a push-forward of a Brownian Bridge measure along the Malliavin--Shavgulidze map $\xi \mapsto \A_{\xi}$, see \eqref{eq:62}.
    For $\alpha = \pi$, the invariance group is the non-compact group $\mrm{PSL}(2,\R)$ and the quotiented measure is identified with the Schwarzian field theory.
    The names originate in the classification of $\mrm{PSL}(2,\RR)$-conjugacy classes.
  }
  \label{fig:orbits_overview}
\end{figure}

\begin{defn} \label{def:orbital-measure}
  We say that a Radon measure  $\measOrb{\sigma^2}{\alpha}$ on $\Diff^1(\T)$ is an unquotiented Virasoro $\alpha$-orbital measure if the following change of variables formula holds: for any $\psi \in \Diff^3(\T)$,
  \begin{equation}  \label{eqDiffeoMeasureChangeUniquness}
    \frac{\d \psi^{\ast}\!\measOrb{\sigma^2}{\alpha}(\phi)}{\d\measOrb{\sigma^2}{\alpha}(\phi)} = 
    \exp\left\{
      \frac{1}{\sigma^2}\int_{\T}\Big[ \Schw_\psi(\phi(\tau))+2\alpha^2 \big(\psi'(\tau)^2-1\big) \Big]\, \phi'(\tau)^2 \d \tau \right\}
    .
  \end{equation}
\end{defn}
 
Equivalently, the change of variables formula  \eqref{eqDiffeoMeasureChangeUniquness} can be written as
\begin{equation}  \label{eqDiffeoMeasureChangeUniquness-tan}
  \frac{\d \psi^{\ast}\!\measOrb{\sigma^2}{\alpha}(\phi)}{\d\measOrb{\sigma^2}{\alpha}(\phi)} = 
  \exp\left\{
    \frac{1}{\sigma^2}\int_{\T}\Big[ \Schw(\tfrac{1}{\alpha}\tan(\alpha\psi-\tfrac{\alpha}{2}),\phi(\tau))-2\alpha^2 \Big]\, \phi'(\tau)^2 \d \tau \right\},
\end{equation}
where we implicitly understand that $\frac1{\alpha}\tan(\alpha\,\cdot) = \mrm{id}$ if $\alpha=0$,
see \eqref{e:Schw-tan-chain-bis} below.

\begin{rmrk} \label{rmrk:U1}
For any $\alpha$, the change of variables formula \eqref{eqDiffeoMeasureChangeUniquness-tan} implies that
the unquotiented Virasoro orbital measures are invariant
under the $\mrm{U}(1)$ action $\phi \mapsto \phi+\theta$ where  $\theta\in \T$.
For the exceptional value $\alpha=\pi$, the measure is invariant
under the action of the larger group $\SL(2,\R)$,
and in fact for $\alpha=k\pi$, $k\in\NN$, it is invariant under $G_\alpha = \SL^{(k)}(2,\R)$, i.e.\ the $k$-fold covering group of $\mrm{PSL}(2,\RR)$.

Consequently, one may consider the quotient measure $\measOrb{\sigma^2}{\alpha}/G_{\alpha}$,
and in fact this is the natural object from the point of view of Virasoro coadjoint orbits.
See Section~\ref{sec:schwarzian} in the most delicate case of the Schwarzian orbit.
We will only consider the parameter range $\alpha^{2} \leq \pi^{2}$, in which the action in \eqref{eq:55} is bounded from below and the (quotiented) orbital measure is expected to have finite mass.
\end{rmrk}

Considering \eqref{eq:55}, one expects unquotiented $\alpha$-orbital measures to form a family of mutually absolutely continuous measures (note that orbital measures are assumed to be supported on $C^{1}$-reparametrisations $\Diff^{1}(\T)$, so the $\alpha^{2}$-dependent term is a.e.\ finite).
In fact the following holds:
\begin{lmm}\label{lem:orbital-measure-family}
  If $\measOrb{\sigma^2}{\alpha_0}$ is an unquotiented Virasoro $\alpha_0$-orbital measure, then $\measOrb{\sigma^2}{\alpha}$, defined by
  \begin{equation}\label{eq:57}
    \d\measOrb{\sigma^2}{\alpha}(\phi) = \exp\left\{\frac{2(\alpha^2-\alpha_0^2)}{\sigma^2}\int_{\T}\phi'^{\, 2}(\tau)\d\tau\right\} \d\measOrb{\sigma^2}{\alpha_0}(\phi)
  \end{equation}
  is an unquotiented $\alpha$-orbital measure.
\end{lmm}

\begin{proof}
  The proof is straightforward from the definition.
  Indeed, writing $\d\psi^{\ast}\measOrb{\sigma^2}{\alpha}(\phi) = \d\measOrb{\sigma^2}{\alpha}(\psi\circ \phi)$,
  and applying the Radon-Nikodym chain rule %
  one has:
  \begin{equation}
    \label{eq:56}
    \frac{\d\measOrb{\sigma^2}{\alpha}(\psi\circ\phi)}{\d\measOrb{\sigma^2}{\alpha}(\phi)}
    =
    \frac{\d\measOrb{\sigma^2}{\alpha}(\psi\circ\phi)}{\d\measOrb{\sigma^2}{\alpha_0}(\psi \circ \phi)}
    \cdot
    \frac{\d\measOrb{\sigma^2}{\alpha_0}(\psi\circ\phi)}{\d\measOrb{\sigma^2}{\alpha_0}(\phi)}
    \cdot
    \frac{\d\measOrb{\sigma^2}{\alpha_0}(\phi)}{\d\measOrb{\sigma^2}{\alpha}(\phi)}.
  \end{equation}
  By the definition in \eqref{eq:57}, the first and last factor combine to $\exp(\frac1{\sigma^{2}}{2(\alpha^2-\alpha_0^2)}\int [(\psi\circ\phi)'^{2} -\phi'^{2}])$.
  The middle factor is given by \eqref{eqDiffeoMeasureChangeUniquness}.
  Combining these factors it follows easily that $\measOrb{\sigma^2}{\alpha}$ is an unquotiented $\alpha$-orbital measure.
\end{proof}

The following theorem shows that the change of variable formula \eqref{eqDiffeoMeasureChangeUniquness}
uniquely characterises orbital measures up to a multiplicative constant and
provided a natural regularity condition holds. The regularity condition is motivated by the previous lemma and stated first.

\begin{defn} \label{def:regular}
  An unquotiented $\alpha_{0}$-orbital measure $\measOrb{\sigma^2}{\alpha_{0}}$ is called \emph{regular} if there exists an $\alpha^{2}\in\R$
  such that $\measOrb{\sigma^2}{\alpha}$ as given by Lemma~\ref{lem:orbital-measure-family} has finite total mass.
\end{defn}

\begin{thr}\label{thm:uniqueness}
  Let $\measOrb{\sigma^2}{\alpha}$ be a regular unquotiented Virasoro $\alpha$-orbital measure.
  Then $\measOrb{\sigma^2}{\alpha}$ is unique up to a multiplicative constant.
\end{thr}

Since the change of variables formula is invariant under translations,
by which we understand the action $\phi(\cdot) \mapsto \phi({\cdot}\, + \tau_{0}
)$ for any fixed $\tau_{0} \in \T$,
the theorem implies in particular that regular unquotiented Virasoro $\alpha$-orbital measure are translation-invariant.
Note that the definition of $\alpha$-orbital measures refers to the transformation of the measure under post-composition, whereas the translation invariance is with respect to pre-composition.
An interesting byproduct as well as ingredient of the uniqueness theorem's proof is an explicit expression for the $\alpha$-dependence of the partition functions for $\alpha^{2} < \pi^{2}$:

\begin{thr}\label{thm:partition-function}
  Let $\left\{\measOrb{\sigma^2}{\alpha}\right\}_{\alpha^{2}\in\RR}$ be a family of regular unquotiented orbital measures as in Lemma~\ref{lem:orbital-measure-family}.
  Write $\PartF{\sigma^{2}}{\alpha} = \measOrb{\sigma^2}{\alpha}(\Diff^{1}(\T))$ for the total mass.
  Then we have for any $\alpha^{2} < \pi^{2}$ that
  \begin{equation}
    \label{eq:4}
    \PartF{\sigma^{2}}{\alpha} = \frac{\alpha}{\sin(\alpha)}e^{2\alpha^{2}/\sigma^{2}} \PartF{\sigma^{2}}{0}.
  \end{equation}
\end{thr}

Note that the total mass \eqref{eq:4} diverges as $\alpha \nearrow \pi$, which is expected due to the emergence of the measure's $\mrm{PSL}(2,\RR)$-symmetry at the exceptional point.
Normalising the circle $\T$ to have unit length, the total mass of the quotiented measure $\measOrb{\sigma^{2}}{\alpha}/\mrm{U}(1)$ is the same as in \eqref{eq:4}.

In Section~\ref{sec:construction} we show the existence
of regular unquotiented orbital measures via an explicit construction in terms of Brownian Bridges.
By Lemma~\ref{lem:orbital-measure-family} it suffices to consider $\alpha=0$.

\begin{thr} \label{thm:orbital-exist-bis}
  Regular unquotiented Virasoro $\alpha$-orbital measures exist.
  In fact, the $0$-orbital measure can be realised as a function of a Brownian Bridge and a Lebesgue measure on $\T$ for the zero mode:
  \begin{equation}%
    \d\mu_{\sigma^2}(\phi) \coloneqq 
    \d \WS{\sigma^2}{0}{1}(\xi)
    \otimes \d\Theta, \qquad \text{with } \phi(t) = \Theta + \A_{\xi}(t)\enspace (\mathrm{mod}\, 1), \text{ for } \Theta\in [0, 1),
  \end{equation}
  where $\d\Theta$ is the Lebesgue measure on $[0,1)$ and with the Malliavin--Shavgulidze map
  \begin{equation} \label{eq:62}
    \A(\xi)(t) \coloneqq \A_{\xi}(t) 
    \coloneqq \frac{\int_{0}^t e^{\xi(\tau)}\d\tau}{\int_{0}^1 e^{\xi(\tau)}\d\tau}.
  \end{equation}
\end{thr}

Since it is obviously finite (thus regular), we can thus identify $\mu_{\sigma^2}$ with
the unique regular $0$-orbital measure $\measOrb{\sigma^2}{0}$ according to Theorem~\ref{thm:uniqueness}, up to normalisation.

Since the $0$-orbital measure is a function of a Brownian Bridge, its partition function
should be the same as that of the Brownian Bridge, and it is natural to normalise it to be
\begin{equation}
  \PartF{\sigma^2}{0} = \frac{1}{\sqrt{2\pi \sigma^2}}.
\end{equation}
The right-hand side coincides (up to a multiplicative constant independent of $\sigma^2$) with the natural interpretation of
the partition function of the Brownian Bridge, for example, defined in terms of the $\zeta$-regularised determinant of the Dirichlet Laplacian
on $[0,1]$, see Section~\ref{sec:BB}.

A related approach to the calculation of the partition function is via ``bosonisation'' of hyperbolic ($\alpha^{2} < 0$) orbital measures,
see \citep{AlekseevBosonization}. %

\subsection{Variation of the metric and Schwarzian correlation functions}

Theorem~\ref{thrMeasureSLInv} can be generalised in terms of a non-constant metric on the circle, $\rho^2\colon \T \to \R_+$.
Similar to \citep[Appendix~C]{StanfordWittenFermionicLocalization},
the Schwarzian Field Theory with background metric $\rho^2$ on $\T$ should formally be given by
\begin{equation} \label{eq:1-h}
  \d\FMeas{\rho}(\phi)
  =
  \exp\left\{
    \int_{\T} \Schw(\tan (\pi \phi -\tfrac{\pi}{2}),\tau) \, \frac{\d\tau}{\rho(\tau)}
  \right\}
  \frac{\prod_{\tau \in \T}\frac{\d\phi(\tau)}{\phi'(\tau)}}{\SL(2, \R)} ,
\end{equation}
where $\rho = \sqrt{\rho^2}$ is the positive square root of the metric $\rho^2$.
Thus the constant choice $\rho(\tau)=\sigma^2$ for all $\tau \in\T$ corresponds to \eqref{eq:1}.
The proof of the following theorem  is outline in Section~\ref{sec:metric}.

\begin{thr} \label{thr:2}
  For $\rho \colon \T \to \R_+$ in $C^1(\T)$,
  there is a unique (up to a multiplicative constant) $\SL(2,\R)$-invariant Borel measure $\FMeasSL{\rho}$ on $\Diff^1(\T)$
  that satisfies the expected change of variable formula
  \begin{equation}    \label{eqDiffeoMeasureChange-metric}
    \frac{\d \psi^{\ast}\!\FMeasSL{\rho}(\phi)}{\d\FMeasSL{\rho}(\phi)} =
    \frac{\d \FMeasSL{\rho}(\psi \circ \phi)}{\d\FMeasSL{\rho}(\phi)} = 
    \exp\left\{
      \int_{\T} \Big[\Schw(\tan(\pi\psi-\tfrac{\pi}{2}),\phi(\tau))-2\pi^2\Big]\, \phi'(\tau)^2 \frac{\d \tau}{\rho(\tau)} \right\},
  \end{equation}
  for any $\psi\in \Diff^3(\T)$,  and has a quotient $\FMeasSL{\rho}/\SL(2,\R)$ that is a finite Borel measure.
  This measure has total mass
  \begin{equation} \label{e:Z-general}
    \PartFSL(\rho) = \exp\left\{\frac{1}{2}\int \frac{\rho'(\tau)^2}{\rho(\tau)^3}\, \d\tau\right\} \PartFSL(\sigma^2_{\rho}), \qquad \text{where $\sigma^2_{\rho} =\int \rho \d\tau$,}
  \end{equation}
  where $\PartFSL(\sigma^2)$ denotes the partition function \eqref{e:Zsigma2},
  and we again imposed a normalisation consistent with the $\zeta$-regularisation of
  the partition function of a Brownian Bridge (cf.\ Section~\ref{sec:metric}). %
\end{thr}

In \citep[Appendix~C]{StanfordWittenFermionicLocalization}, the `correlation functions' of the Schwarzian Field Theory
are formally defined by differentiation of the partition function with respect to the metric:
\begin{equation}
  \label{eq:18}
  \begin{aligned}
    \langle\Schw(\tau_1)\cdots\Schw(\tau_n)\rangle_{\sigma^2} &\coloneqq \frac{1}{\PartFSL(\sigma^2)} \frac{\partial^n}{\partial \epsilon_1 \cdots \partial \epsilon_n} \PartFSL(\rho(\epsilon_1,\dots,\epsilon_n)),
    \\
    \frac{1}{\rho}(\epsilon_1,\dots,\epsilon_n) &= \frac{1}{\sigma^2} + \epsilon_1 \delta_{\tau_1} + \cdots + \epsilon_n \delta_{\tau_n}.
  \end{aligned}
\end{equation}
These `correlation functions' can be computed from the formula \eqref{e:Z-general}.
For example,
\begin{equation}\label{eq:22}
  \langle \Schw(0)\rangle_{\sigma^2}
  = 2\pi^2 + \tfrac32\sigma^{2}
\end{equation}
and
\begin{equation}\label{eq:23}
  \langle \Schw(0)\Schw(\tau)\rangle_{\sigma^2}
  = [4\pi^{4} + 10\pi^{2}\sigma^{2} + \tfrac{15}{4}\sigma^{4}] - 2\sigma^{2}[2\pi^{2} +\tfrac32 \sigma^{2}]\,\delta(\tau) - \sigma^{2}\delta''(\tau),
\end{equation}
see Section~\ref{sec:appendix-formal-corr}. In the relation to Liouville Field Theory, these correlators correspond to stress-energy tensor correlation
functions, see \citep[Appendix~A]{ConformalBootstrap}.

Since the normalisation of the partition function involves $\zeta$-regularisation %
of the Brownian Bridge (see Section~\ref{sec:metric}),
the correlation functions
defined by \eqref{eq:18} are not obviously expectations of random variables.
In fact, the Schwarzian derivative $\Schw_f(\tau)$ of a function $f\colon \T \to \R$ is only defined if $f \in C^3$ while
the support of the Schwarzian Field Theory measure $\d\FMeas{\sigma^2}$ only has regularity $C^{3/2-}$.
A finite-difference-type regularisation of the Schwarzian derivative is given by 
\emph{cross-ratios}
\begin{equation} \label{e:obs}
  \obs{\phi}{s}{t}=
  \frac{\pi \sqrt{\phi'(t)\phi'(s)}}{\sin(\pi[\phi(t)-\phi(s)])} \qquad\text{for}\;\, s\neq t.
\end{equation}
These still respect the $\SL(2,\R)$-invariance and are well-defined for $C^{1}$-functions.
For $C^3$-functions, in the limit of infinitesimally close end-points, the cross-ratios approximate the Schwarzian derivative,
see \citep{LosevCorr}.

In \citep{LosevCorr}, the (probabilistically well-defined) correlation functions of cross-ratios are explicitly computed,
confirming the predictions of \citep{ConformalBootstrap} obtained using the conformal bootstrap.
It is further shown that in the limit $t-s \to 0$
these coincide with the above Schwarzian correlation functions obtained by differentiating the partition function.
This confirms that correlation functions obtained from $\zeta$-regularised partition functions agree with honest probabilistically
defined ones, but it would be interesting to have a more general and conceptual understanding of this.

In \citep{JT_Wilson_line,MR4058854}, these cross-ratio observables are related to Wilson lines in the gauge theory formulation of JT gravity.

\subsection{Related probabilistic literature}
\label{sec:literature}

We will not survey the vast literature in physics related to the Schwarzian Field Theory and the SYK model,
but refer to \citep{MaldacenaStanford,KitaevJosephine} for a starting point on the SYK model
and \citep{SaadShenkerStanford2019,JT_Wilson_line,MR4058854} for a starting point on its relation to JT gravity.
Further physical perspectives on the construction carried out in this paper can be found in \mbox{\citep{BelokurovShavgulidzeExactSolutionSchwarz, BelokurovShavgulidzeCorrelationFunctionsSchwarz, BelokurovShavgulidze3, BelokurovShavgulidze4}}.
In the following, we do mention some other related probabilistic references.

The Schwarzian Field Theory is formally related to a degenerate limit of Liouville Field Theory \citep{SYKasLQM,ConformalBootstrap},
and the conformal bootstrap and the DOZZ formula applied in this context has been used to predict the correlation functions of the Schwarzian Field Theory \citep{ConformalBootstrap}.
While much progress has been made on the mathematical justification of Liouville Field Theory \citep{MR4060417,2005.11530},
see \citep{2403.12780} for a review, and it would be very interesting to explore this connection,
this paper, \citep{LosevCorr}, and \citep{LosevLDP} only use standard stochastic analysis.

Random homeomorphisms of $\T$ have also been studied in the context of random conformal welding \citep{MR2892610,MR3551203},
where given a random homeomorphism of the circle one constructs an associated random Jordan curve in the plane.
The Schwarzian Field Theory provides a different natural random diffeomorphism of the circle,
and it would be interesting to explore the associated random conformal welding.
We also remark that the space $\Diff^1(\T)/\SL(2,\R)$ has also appeared in the study of large deviations of
SLE \citep{MR4417203}.

Another motivation is the characterisation and uniqueness of Euclidean field theories, for which the action does not naively make sense,
and a desirable direction is to characterise such field theory measures in terms of change of measure or Schwinger--Dyson type formulas.
See \citep{GubinelliHofmanova} for some developments of this question in context of the $\varphi^4_3$ theory
and \citep{BaverezJego} for recent progress on the CFT related to the SLE loop measure (see also \citep{GordinaQianWang}).
Our uniqueness result for the Schwarzian Field Theory provides an example of such a problem with geometric symmetry.

Some motivation for the construction of quasi-invariant measures on diffeomorphism groups (and loop groups),
of which the Schwarzian Theory is an example,
has been the construction of unitary representations of those. %
The earliest references concerning quasi-invariant measures on $\Diff^1(\T)$ appear to go back to Shavgulidze and collaborators such as
\citep{shavgulidze_example_1978}, and we refer to \citep{BogachevMalliavin} for further discussion and  references.
The Malliavins also considered such measures \citep{malliavin_infinitesimally_1991},
and 
diffusion on such spaces (and quotients) and associated Wiener measures were studied in follow-up works such as
\citep{malliavin_heat_2005,airault_quasi-invariance_2006}.

Finally, we mention that the theory of path integrals for coadjoint orbits of loop group extension of compact Lie groups
is somewhat well-developed.
On a formal level, these orbits again carry a natural symplectic structure and the path integral associated to the Hamiltonian generating the $\mathrm{U}(1)$ action is of Duistermaat--Heckman form.
However, this case looks simpler as the relevant Hamiltonian is simply the Dirichlet energy associated with a Lie algebra valued Brownian Bridge.
Bismut has developed an analytical approach to the calculation of the corresponding heat kernels via a rigorous Duistermaat--Heckman-type deformation involving hypoelliptic Laplacians, see, e.g., \citep{bismut_hypoelliptic_2015}.

As already discussed in Section~\ref{sec:intro-schwarzian}, our construction of the Schwarzian Field Theory is closely inspired by
the one proposed in \citep{BelokurovShavgulidzeExactSolutionSchwarz}, which is lacking some mathematical rigor.
For example, the proof the required change of variable formulas
as well as the existence of the  quotient measures on $\Diff^1(\T)/\SL(2,\R)$ are somewhat subtle, and
our version fills the gaps.
Beyond these more technical points, our uniqueness result for the measure is new
and also relies on a new approach to compute partition functions from generalisation of the change of variables formula.

\subsection{Preliminaries and notation}
\label{sec:prelim}

We write $\D_r = \left\{z\in \Compl\colon |z|<r\right\}$ for the open disk of radius $r$ and $\D=\D_1$ for the open unit disk.
The unit circle is denoted by $\T=[0,1]/\{0\sim 1\}$,
and $\Diff^k(\T)$ is the set of orientation-preserving $C^k$-diffeomorphisms of $\T$, i.e.\ satisfying $\phi'(\tau) > 0$. Note that $\Diff^k(\T)$ is not a linear space.
The topology on $\Diff^k(\T)$ is the natural one given by the identification of $\phi$
with $\xi \in C^{k-1}[0,1]$ and $\Theta \in \T$ via the Malliavin--Shavgulidze map \eqref{e:phi-exp},
which makes $\Diff^{k}(\T)$ a Polish (separable completely metrisable) space as well as a topological group.
The same topology is induced by viewing $\Diff^k(\T)$ as a subspace of $C^k(\T)$.

It will also be useful to consider reparametrisations of $[0,1]$, or more general intervals $[0,T]$, whose derivatives are not periodic.
We write $\Diff^k [0, T]$ for the set of orientation-preserving $C^k$-diffeomorphisms of $[0, T]$, i.e.\ satisfying $\phi'(t)>0$, $\phi(0) = 0$, and $\phi(T) = T$.
In particular, the derivatives do not have to match at the endpoints.
We further set $\Cfree[0,T] = \left\{f\in C[0,T]\, | \, f(0)=0\right\}$, and $C_0[0,T] = \left\{f\in C[0,T]\, | \, f(0)=f(T) = 0\right\}$.

The projective special linear group is $\SL(2,\R)=\nSL(2,\R)/\{\pm 1\}$ where $\nSL(2,\R)$ consists of all
matrices $M = \bigl( \begin{smallmatrix}a & b\\ c & d\end{smallmatrix}\bigr)$ with real entries and determinant $1$.
The action of $\pm M\in \SL(2,\R)$ on $\phi \in \T$ is
\begin{equation}
  f\mapsto M \circ f = \frac{a f + b}  {cf+d}, \qquad\text{where } f = \tan(\pi\phi - \tfrac{\pi}{2}).
\end{equation}
We may identify $\SL(2,\R)$ with its orbit at $\mathrm{id}_{\T} \in \Diff^{1}(\T)$, equivalently parametrised by\footnote{This
  parametrisation is the restriction to the boundary of the action of $\SL(2,\R)$ as a conformal map of the unit disk onto itself,
  see for example \cite[Section~6.2]{MR1738432}.}
\begin{equation}\label{eqSLTransformFormula}
\phi_{z, \ang}(t) =  \ang -\frac{i}{2\pi} \ln \frac{e^{i2\pi t}-z}{1-\bar{z}e^{i2\pi t}} \quad (\text{mod } 1), \qquad \text{for } z\in \mathbb{D} \coloneqq \{z\in\mathbb{C}: |z|<1\}, \, \ang\in \T.
\end{equation}
Up to normalisation, the Haar measure on $\SL(2,\R)$ then takes the form\footnote{See for
  example \cite[Lemma~9.16]{MR2723325} which states that the Haar measure on $\SL(2,\R)$ is given as
the uniform measure on circle (corresponding to $\d a$) and the hyperbolic measure on the upper half plane $\mathbb{H}$.
In \eqref{eqSLHaarMeasureFormula} we have parametrised the hyperbolic measure by the Poincar\'e disk $\D$ instead of $\mathbb{H}$.}
\begin{equation}\label{eqSLHaarMeasureFormula}
\d\Haar (\phi_{z,a}) = \frac{4 \rho \d\rho \, \d\theta \, \d\ang}{(1-\rho^2)^2}, \qquad \text{where } z = \rho\, e^{i 2\pi\theta},
\end{equation}
and we always assume this normalisation for the Haar measure.
One may check that the subspace topology on the $\{\phi_{z,a}\}_{z\in \mathbb{D}, a\in \T}$ inherited from $\Diff^{1}(\T)$ agrees with the topology on $\SL(2,\R)$.
Hence, the $\SL(2,\R)$-orbit at $\mathrm{id}_{\T}$ is a faithful embedding of $\SL(2,\R)$ as a subgroup of $\Diff^1(\T)$.
As a consequence, the action of $\SL(2,\R)$ on $\Diff^{1}(\T)$ is a \emph{proper} group action.

Finally, we recall the definition of the Schwarzian derivative \eqref{eq:Schw-def} and the chain rule \eqref{e:Schw-chain}.
The chain rule implies that the Schwarzian action can be written as
\begin{equation} \label{e:Schw-phi}
  \Schw(\tan(\alpha\phi-\tfrac{\alpha}{2}), \tau)= \Schw(\phi,\tau)+ 2\alpha^2 \phi'(\tau)^2
\end{equation}
where we used that $\Schw(\tan(\alpha\phi-\tfrac{\alpha}{2}),\phi) = 2\alpha^2$. In particular,
\begin{equation}
\begin{aligned} \label{e:Schw-tan-chain-bis}
  \Schw(\tan(\alpha(\psi\circ\phi)-\tfrac{\alpha}{2}), \tau)
  &= \Schw(\psi\circ \phi,\tau)+ 2\alpha^2 (\psi\circ \phi)'(\tau)^2
    \\
  &= \Schw(\phi,\tau) + \Big[ \Schw(\psi, \phi(\tau))+ 2\alpha^2 (\psi'(\phi(\tau))^2 \Big] \phi'(\tau)^2
    \\
  &= \Schw(\tan(\alpha\phi-\tfrac{\alpha}{2}),\tau) + \Big(\Schw(\tan(\alpha\psi-\tfrac{\alpha}{2}), \phi(\tau)) -2\alpha^2\Big)\phi'(\tau)^2
    ,
\end{aligned}
\end{equation}
where we used \eqref{e:Schw-phi} on the first and third line and the chain rule \eqref{e:Schw-chain} on the second line.

\section{Uniqueness and partition function of orbital measures}
\label{sec:orbital}

In this section we first introduce a \emph{pinned} variant of orbital measures (essentially fixing $\phi(t_{0}) = 0$ for some $t_{0}\in\T$).
We derive a generalised change of variables formula (Proposition~\ref{prpWSMeasureChangePinned}) for these pinned measures.
A direct consequence of the latter is the ``boundary defect trick'' (Proposition~\ref{prop:11}), which is a very useful property for concrete calculation with the (pinned) orbital measures (e.g.\ in \citep{LosevCorr} is is instrumental for the derivation of cross-ratio correlation functions).
In particular, we will use it to prove Theorem~\ref{thm:partition-function}, i.e.\ to calculate the $\alpha$-dependence of the partition functions for $\alpha$-orbital measures.
Ideas from the proof are also applied to show uniqueness of orbital measures (Theorem~\ref{thm:uniqueness}).

\subsection{Pinned orbital measures}
\label{sec:pinned}

From Remark~\ref{rmrk:U1} we recall that unquotiented Virasoro $\alpha$-orbital measures are invariant
under the $\mrm{U}(1)$ action $\phi \mapsto \phi+\theta$ where the angle $\theta\in \T$ is identified with an element of $\mrm{U}(1)$.
It will be convenient to consider \emph{pinned} measures
in which a representative of the $\mrm{U}(1)$ orbits is fixed
(which may be interpreted as ``gauge-fixing'' with respect to the global $\mrm{U}(1)$ symmetry).
For an unquotiented $\alpha$-orbital measure $\measOrb{\sigma^2}{\alpha}$ and $t_{0} \in \T$ we define the measure $\measPin{\sigma^2}{\alpha}{t_{0}}$ via
\begin{equation}
  \label{eq:59}
  \int F(\phi) \d\measPin{\sigma^2}{\alpha}{t_0}(\phi)\,
  \coloneqq
  \int F(\phi - \phi(t_{0})) \d\measOrb{\sigma^2}{\alpha}(\phi).
\end{equation}

\noindent
From this definition it is clear that the distribution of $\phi'(\cdot)$ is the same under $\measPin{\sigma^2}{\alpha}{t_{0}}$ for any $t_{0}$
and under the measure $\measOrb{\sigma^2}{\alpha}$ without pinning.
Note that we do not assume that $\measOrb{\sigma^2}{\alpha}$ is translation-invariant
(i.e.\ under $\phi(t)\mapsto \phi(t+\theta)$); it will later be a consequence of the uniqueness theorem.

\begin{lmm}\label{lmmDiffMeasureDisintegration}
\begin{enumerate}
\item
For any $t_0\in \T$ the measure $\measPin{\sigma^2}{\alpha}{t_0}$ is a Radon measure.

\item 
  From the pinned measure $\measPin{\sigma^2}{\alpha}{t_0}$ defined in \eqref{eq:59} we can recover the original unquotiented measure by adding a uniform random shift:
  \begin{equation}
    \label{eq:61}
    \int \int_{\T} F(\phi + \theta) \d\theta\d\measPin{\sigma^2}{\alpha}{t_0}(\phi)
    = \int F(\phi) \d\measOrb{\sigma^2}{\alpha}(\phi),
  \end{equation}
\item
The family of measures $(\measPin{\sigma^2}{\alpha}{t_0})_{t_{0}\in\T}$ is uniquely characterised by the disintegration identity
  \begin{equation}
    \label{eq:65}
    \d\measOrb{\sigma^2}{\alpha}(\phi) = \int_{\T}\dd{t_{0}}\,  \phi'(t_{0}) \d\measPin{\sigma^2}{\alpha}{t_0}(\phi),
  \end{equation}
  with the requirement that $\measPin{\sigma^2}{\alpha}{t_0}$ is supported in $\{\phi\in\Diff^{1}(\T)\colon \phi(t_{0}) = 0\}$
  and that $t_0 \mapsto \measPin{\sigma^2}{\alpha}{t_0}$ is continuous (in the sense of Appendix~\ref{sec:disintegration}, see Definition~\ref{defMeasureCont}).
\end{enumerate}
\end{lmm}

\begin{proof}
Using the the definition of the pinned measure \eqref{eq:59}, invariance of $\T$ under $\T+a$ for each $a\in \T$,
the invariance of $\measOrb{\sigma^2}{\alpha}$ under $\phi\mapsto \phi+\theta$ for each $\theta\in \T$ gives \eqref{eq:61}:
\begin{equation}
  \begin{aligned}
   \int \int_{\T} F(\phi + \theta) \d\theta\d\measPin{\sigma^2}{\alpha}{t_0}(\phi)
  &=
    \int \int_{\T}  F(\phi -\phi(t_0) + \theta) \d\theta\d\measOrb{\sigma^2}{\alpha}(\phi)
    \\
  &=
    \int \int_{\T}  F(\phi + \theta) \d\theta  \d\measOrb{\sigma^2}{\alpha}(\phi)
    \\
  &=
    \int_{\T} \d\theta \int F(\phi) \d\measOrb{\sigma^2}{\alpha}(\phi)
  =
    \int F(\phi) \d\measOrb{\sigma^2}{\alpha}(\phi)
    .
  \end{aligned}
\end{equation}

We now prove that $\measPin{\sigma^2}{\alpha}{t_0}$ is locally finite (which implies that it is a Radon measure, since $\Diff^1(\T)$ is a Polish space).
Fix any $\phi_0\in \Diff^1(\T)$.
Since the unpinned measure $\measOrb{\sigma^2}{\alpha}$ is locally finite (as it is Radon by definition), there exists a continuous bounded function $F_{\phi_0}\colon \Diff^1(\T)\to [0, \infty)$ such that $F_{\phi_0}(\phi_0) > 0$ and $\int F_{\phi_0}(\phi)\d \measOrb{\sigma^2}{\alpha}<\infty$.
Consider $G_{\phi_0}(\phi) =  \int_{\T} F_{\phi_0}(\phi + \theta) \d \theta$, which is also a continuous bounded non-negative function on $\Diff^1(\T)$ (continuity follows by dominated convergence).
We also have that $G_{\phi_0}(\phi_0)>0$, since $F_{\phi_0}(\phi_0)>0$.
Moreover, for any $\theta\in \T$ we have $G_{\phi_0}(\phi+\theta) = G_{\phi_0}(\phi)$.
Therefore, from \eqref{eq:61} we get that 
\begin{equation}\label{eqLemmaDisintegrRadonBound}
\int G_{\phi_0}(\phi) \d\measPin{\sigma^2}{\alpha}{t_0}(\phi) =  \int F_{\phi_0}(\phi)  \d\measOrb{\sigma^2}{\alpha}(\phi)<\infty,
\end{equation}
which implies that $\measPin{\sigma^2}{\alpha}{t_0}(\phi)$ is finite in the neighbourhood of $\phi_0$.

For the disintegration identity \eqref{eq:65}, one has
\begin{equation}
  \begin{aligned}
    \int F(\phi) \int_{\T}\dd{t_{0}}\,  \phi'(t_{0}) \d\measPin{\sigma^2}{\alpha}{t_0}(\phi)
    &=
      \int \left(\int_{\T} \dd{t_0} \, F(\phi-\phi(t_0))  \phi'(t_{0}) \right) \d\measOrb{\sigma^2}{\alpha}(\phi)
    \\
    &=
      \int \left(\int_{\T} \dd{\theta} \, F(\phi+\theta) \right) \d\measOrb{\sigma^2}{\alpha}(\phi)
    \\
    &=
      \int F(\phi) \d\measOrb{\sigma^2}{\alpha}(\phi).
  \end{aligned}
\end{equation}
In Appendix~\ref{sec:disintegration}, we include a general result that, for a given measure,
disintegration into a continuous family of measures is unique,
where continuity is understood in the sense of Definition~\ref{defMeasureCont}.
To be precise, for uniqueness, we apply Proposition~\ref{prop:disintegration}, where we take $X=\Diff^1(\T)$, $Y=\T$, the projection $\pi\colon X \to Y,\, \phi \mapsto \phi^{-1}(0) $, and the family of measures $\left\{\phi'(t_0)\measPin{\sigma^2}{\alpha}{t_0}\right\}_{t_0\in \T}$.
Then what is left to show is the continuity of the latter in $t_0$.
Since $\phi'(t_0)\measPin{\sigma^2}{\alpha}{t_0} = \phi'\left(\phi^{-1}(0)\right)\measPin{\sigma^2}{\alpha}{t_0}$ and the function $\phi\mapsto \phi'\left(\phi^{-1}(0)\right)$ is continuous on $\Diff^1(\T)$, it follows from Proposition~\ref{prop_measure_density_contin} that it is sufficient to check the continuity of $\left\{\measPin{\sigma^2}{\alpha}{t_0}\right\}_{t_0\in \T}$.
For this we take the family of open sets $U_{\phi_0} = \left\{\phi:\, G_{\phi_0}(\phi)>\tfrac{1}{2} G_{\phi_0}(\phi_0)\right\}$ parametrized by $\phi_0\in \Diff^1(\T)$, where $G_{\phi_0}$ is defined above, and check that these $U_{\phi_0}$ satisfy the properties of Definition~\ref{defMeasureCont}.

Firstly, it is obvious that $\cup_{\phi_0\in \Diff^1(\T)} U_{\phi_0} = \Diff^1(\T)$ since we always have $\phi_0 \in U_{\phi_0}$.
Secondly, for any $\phi_0$ we have that $\measPin{\sigma^2}{\alpha}{t_0}\left(U_{\phi_0}\right)$ are bounded uniformly in $t_0$ because of \eqref{eqLemmaDisintegrRadonBound}.
Finally, suppose that $F:\Diff^1(\T)\to [0, \infty)$ is a bounded continuous function such that $F(\phi)=0$ for $\phi\notin U_{\phi_0}$.
Normalizing $F$ we may assume  $F\leq 1$.
Observe that since for any $\theta\in \T$: $G_{\phi_0}(\phi+\theta) = G_{\phi_0}(\phi)$, 
we have
\begin{equation} \label{eqLemmaDisintegrUniformBound}
F(\phi-\phi(t_0)) 
< 
\frac{2}{G_{\phi_0}(\phi_0)} \, G_{\phi_0}(\phi-\phi(t_0))
=
\frac{2}{G_{\phi_0}(\phi_0)} G_{\phi_0}(\phi).
\end{equation}
Then, continuity of $t_0\mapsto \int F \d\measPin{\sigma^2}{\alpha}{t_0}$ follows from the definition \eqref{eq:59} and dominated convergence:
\begin{equation}
\lim_{t\to t_0} \int F(\phi - \phi(t)) \d\measOrb{\sigma^2}{\alpha}(\phi)
=
\int F(\phi - \phi(t_0)) \d\measOrb{\sigma^2}{\alpha}(\phi).
\end{equation}
Thus uniqueness of the disintegration identity holds.
\end{proof}

After having defined and characterised the pinned orbital measures,
we now show that they satisfy a modified change of variables formula, analogous to \eqref{eqDiffeoMeasureChangeUniquness}.
The statement will be somewhat more general, in that we consider the transformation behaviour under post-composition by a map $f \in \Diff^{1}(\T)\cap \Diff^{3}[0,1] = \{f\in \Diff^{3}[0,1]\colon f'(0)=f'(1)\}$, instead of just $\Diff^{3}(\T)$.
The former allows for discontinuities in the second and third derivatives at $0\sim 1$, see Section~\ref{sec:prelim}.
This is important for the ``boundary defect trick'' in Proposition~\ref{prop:11}.

\begin{prp}\label{prpWSMeasureChangePinned} 
Let $\measPin{\sigma^2}{\alpha}{t_0}$ be an  $\alpha$-orbital measure pinned at $t_0 \in \T$.
Let $f\in \Diff^{1}(\T)\cap \Diff^{3}[0,1]$. 
Then $f^{\ast} \measPin{\sigma^2}{\alpha}{t_0}$ is absolutely continuous with respect to $\measPin{\sigma^2}{\alpha}{t_0}$ and
\begin{multline}\label{eq:64}
  \frac{\d f^{\ast} \measPin{\sigma^2}{\alpha}{t_0}(\phi)}{\d  \measPin{\sigma^2}{\alpha}{t_0}(\phi)} =  
  \frac{1}{\sqrt{f'(0)f'(1)}}
  \\
  \times
  \exp\left\{ 
    \frac{1}{\sigma^2}\left[\frac{f''(0)}{f'(0)}-\frac{f''(1)}{f'(1)}\right]\phi'(t_0)
    +
    \frac{1}{\sigma^2}\int_{\T\backslash\{t_0\}} \left[ \Schw_f \big(\phi(\tau)\big) + 2\alpha^2 \big(f'(\phi(\tau))^2-1\big)\right]\phi'^{\,2}(\tau) \d \tau
  \right\}.
\end{multline}
Note that in the above $f'(0)=f'(1)$ holds by assumption, however we write it in this form to match the more general statement Proposition~\ref{prpBBMMeasureChange}.
\end{prp}

The expression in the exponential of \eqref{eq:64} can be rewritten as
\begin{equation}
  \frac{1}{\sigma^2}\int_{\T} \left[\Schw_f \big(\phi(\tau)\big)+2\alpha^2(f'(\phi(\tau))^2-1)\right]\phi'^{2}(\tau)\d \tau
\end{equation}
if we interpret $\Schw_f (\cdot)$ evaluated at $0$ as  $\big[\frac{f''(0)}{f'(0)}-\frac{f''(1)}{f'(1)}\big]$ multiplied by the delta measure. 
This interpretation is consistent with the usual formula $\Schw_{f} = \big(\frac{f''}{f'}\big)'-\frac12 \big(\frac{f''}{f'}\big)^2$.

\begin{proof}
First suppose that $f \in \Diff^{3}(\T)$ with $f(0) = 0$.
Then, on the one hand, \eqref{eq:65} implies
\begin{equation}
  \label{eq:63}
  \begin{aligned}
    \d f^{\ast}\!\measOrb{\sigma^2}{\alpha}(\phi)
    =
    \d\measOrb{\sigma^2}{\alpha}(f\circ \phi)
    &=  \int_{\T}\dd{t_{0}} f'(0) \phi'(t_{0}) \d\measPin{\sigma^2}{\alpha}{t_{0}}(f\circ \phi)
    \\
    &=  \int_{\T}\dd{t_{0}} f'(0) \phi'(t_{0}) \d f^{\ast}\!\measPin{\sigma^2}{\alpha}{t_{0}}(\phi)
      ,
  \end{aligned}
\end{equation}
and, on other hand, again using the disintegration identity \eqref{eq:65} for $\measOrb{\sigma^2}{\alpha}$,
  \begin{equation}\label{eq_change_of_var_gauge_fixed_1}
    \begin{aligned}
     \d f^{\ast}\!\measOrb{\sigma^2}{\alpha}(\phi)
     =
    \frac{\dd f^{\ast}\!\measOrb{\sigma^2}{\alpha}(\phi)}{\dd\measOrb{\sigma^2}{\alpha}(\phi)}
      \dd\measOrb{\sigma^2}{\alpha}(\phi) 
      &= \int_{\T}\dd{t_{0}} \phi'(t_{0}) \frac{\dd f^{\ast}\!\measOrb{\sigma^2}{\alpha}(\phi)}{\dd\measOrb{\sigma^2}{\alpha}(\phi)} \d\measPin{\sigma^2}{\alpha}{t_{0}}(\phi).
    \end{aligned}
\end{equation}

In both cases, the integrand in the disintegration on the right-hand side is continuous in $t_0$ in the sense of Definition~\ref{defMeasureCont}.
Indeed,
since the family of measures $\left\{\measPin{\sigma^2}{\alpha}{t_{0}}\right\}_{t_0\in \T}$ is continuous in $t_0$, it is easy to deduce that the family $\left\{f^{\ast}\!\measPin{\sigma^2}{\alpha}{t_{0}}\right\}_{t_0\in \T}$ is also continuous in $t_0$.
Thus, by Proposition~\ref{prop_measure_density_contin}, the family $\left\{f'(0) \phi'(t_{0})\d f^{\ast}\!\measPin{\sigma^2}{\alpha}{t_{0}}(\phi)\right\}_{t_0\in \T}$ appearing in \eqref{eq:63} is also continuous in $t_0$.
  For \eqref{eq_change_of_var_gauge_fixed_1}, we start from the fact that
  the integrand on the right-hand side is continuous in $\phi$
  by \eqref{eqDiffeoMeasureChangeUniquness}.
  Therefore, by Proposition~\ref{prop_measure_density_contin}, the family of measures 
appearing on the right-hand side of \eqref{eq_change_of_var_gauge_fixed_1} is continuous in $t_0$.
  
Thus, from \eqref{eq:63}, \eqref{eq_change_of_var_gauge_fixed_1}, uniqueness of disintegration (see Proposition~\ref{prop:disintegration}),
and since $f'(0)=f'(1)$, 
\begin{equation}\label{eq_delta_approx_0}
  \begin{aligned}
    \frac{\d f^{\ast}\!\measPin{\sigma^2}{\alpha}{t_{0}}(\phi)}{\d\measPin{\sigma^2}{\alpha}{t_{0}}(\phi)}
    &= \frac{1}{f'(0)} \frac{\dd f^{\ast}\!\measOrb{\sigma^2}{\alpha}(\phi)}{\dd\measOrb{\sigma^2}{\alpha}(\phi)}
    \\
    &
      =\frac{1}{\sqrt{f'(0)f'(1)}}
      \exp\left\{ 
      \frac{1}{\sigma^2}\int_{\T} \left[ \Schw_f \big(\phi(\tau)\big) + 2\alpha^2 \big(f'(\phi(\tau))^2-1\big)\right]\phi'^{\,2}(\tau) \d \tau
      \right\}.
  \end{aligned}
\end{equation}
which implies \eqref{eq:64} for $f\in\Diff^{3}(\T)$.
  
\medskip

For general $f\in \Diff^{1}(\T) \cap \Diff^{3}[0,1]$ we prove \eqref{eq:64} by passing to the limit of \eqref{eq:64} in $f$. We consider a sequence $f_{n} \in \Diff^{3}(\T)$ with $f_{n}(0) = 0$ such that $f_{n} \to f$ in $\Diff^1(\T)$,
and which satisfies the following two properties.
Firstly, for any $\phi\in\Diff^1(\T)$ with $\phi(t_0) = 0$,
\begin{multline}\label{eq_delta_approx_1}
\lim_{n\to \infty}
\int_{\T} \left[ \Schw_{f_n} \big(\phi(\tau)\big) + 2\alpha^2 \big(f_n'(\phi(\tau))^2-1\big)\right]\phi'^{\,2}(\tau) \d \tau\\
=
\left[\frac{f''(0)}{f'(0)}-\frac{f''(1)}{f'(1)}\right]\phi'(t_0)
    +
\int_{\T\backslash\{t_0\}} \left[ \Schw_f \big(\phi(\tau)\big) + 2\alpha^2 \big(f'(\phi(\tau))^2-1\big)\right]\phi'^{\,2}(\tau) \d \tau.
\end{multline}
Secondly, for any $\phi_0\in \Diff^1(\T)$ there exist an open neighbourhood $U_{\phi_0} \subset \Diff^1(\T)$ containing $\phi_{0}$, and a constant $M>0$, such that for any $n\in \N$ and any $\phi\in U_{\phi_0}$,
\begin{equation}\label{eq_delta_approx_2}
\left|\int_{\T} \left[ \Schw_{f_n} \big(\phi(\tau)\big) + 2\alpha^2 \big(f_n'(\phi(\tau))^2-1\big)\right]\phi'^{\,2}(\tau) \d \tau\right|
< M.
\end{equation}
In what follows we will also assume that $U_{\phi_0}$ is sufficiently small, so that for all $\phi_0\in\Diff^1(\T)$ we have that $\measPin{\sigma^2}{\alpha}{t_0}\left(U_{\phi_0}\right)<\infty$.

Such a sequence $\{f_n\}_n$ can be constructed by standard arguments. 
For example, we can take $f_n(\tau) = (\delta_{n} \ast f)(\tau) - (\delta_{n} \ast f)(0)$ for some approximate identity $\{\delta_{n}\}_{n\to \infty}$.

Since \eqref{eq_delta_approx_1} and \eqref{eq_delta_approx_2} can be deduced for such a sequence $\{f_n\}_n$ by fairly standard elementary methods we provide details in Lemma~\ref{lem:proof_f_n_construct} below this proof. %
We first show how to use the sequence $\{f_n\}_n$ to finish the proof of the proposition.

\medskip

Having the sequence $f_n$ as described above, we fix some $\phi_0\in\Diff^1(\T)$.
By continuity of the group product $\Diff^1(\T)\times \Diff^1(\T) \to \Diff^1(\T)$ given by $(g, \phi) \mapsto g\circ \phi$ we get that there exist $N>0$ and an open $V_{\phi_0}\subset U_{\phi_0}$ such that for any $n>N$ and any $\phi\in V_{\phi_0}$ we have $f_n \circ \phi \in U_{f \circ \phi_0}$. 
Then for any continuous bounded function $F:\Diff^1(\T)\to[0,\infty)$ supported in $V_{\phi_0}$ we have that
\begin{equation}\label{eq_delta_approx_3}
\begin{aligned}
\lim_{n\to \infty}  
\int F(\phi)\d f_n^{\ast} \measPin{\sigma^2}{\alpha}{t_0}(\phi)
&=
\lim_{n\to \infty}  
\int F(f_n^{-1}\circ \phi) \d\measPin{\sigma^2}{\alpha}{t_0}(\phi)\\
&=
\int F(f^{-1}\circ \phi) \d\measPin{\sigma^2}{\alpha}{t_0}(\phi)
=
\int F(\phi)\d f^{\ast}\!\measPin{\sigma^2}{\alpha}{t_0}(\phi),
\end{aligned}
\end{equation}
where we used dominated convergence, the bound
\begin{equation}
F(f_n^{-1}\circ \phi) \leq \sup F \cdot \one_{V_{\phi_0}}(f_n^{-1}\circ \phi) \leq \sup F \cdot \one_{U_{f \circ \phi_0}}(\phi),
\end{equation}
and the fact that $\measPin{\sigma^2}{\alpha}{t_0}\left(U_{f\circ \phi_0}\right)<\infty$ (here, $\one_A$ denotes the indicator function of $A$).
On the other hand, from \eqref{eq_delta_approx_0} applied to $f_n$, \eqref{eq_delta_approx_1}, \eqref{eq_delta_approx_2}, and dominated convergence, we obtain
\begin{equation}\label{eq_delta_approx_4}
\begin{aligned}
&\lim_{n\to \infty}  
\int F(\phi)\d f_n^{\ast}\measPin{\sigma^2}{\alpha}{t_0}(\phi) \\
&=
\lim_{n\to \infty}  
\int F(\phi) 
\frac{1}{f_n'(0)}
  \exp\left\{ 
    \frac{1}{\sigma^2}\int_{\T} \left[ \Schw_{f_n} \big(\phi(\tau)\big) + 2\alpha^2 \big(f_n'(\phi(\tau))^2-1\big)\right]\phi'^{\,2}(\tau) \d \tau
  \right\} 
  \d \measPin{\sigma^2}{\alpha}{t_0}(\phi) \\
&= 
\int F(\phi) 
\frac{1}{f'(0)}
   \exp\left\{ \frac{1}{\sigma^2}
   I_f(\phi)
  \right\}
  \d \measPin{\sigma^2}{\alpha}{t_0}(\phi),
\end{aligned}
\end{equation}
where $I_f(\phi)$ denotes the right-hand side of \eqref{eq_delta_approx_1}.
Thus, combining \eqref{eq_delta_approx_3} and \eqref{eq_delta_approx_4}
\begin{equation}
\int F(\phi)\d f^{\ast}\!\measPin{\sigma^2}{\alpha}{t_0}(\phi)
=
\int F(\phi) 
\frac{1}{f'(0)}
   \exp\left\{\frac{1}{\sigma^2}
   I_f(\phi)
  \right\}
  \d \measPin{\sigma^2}{\alpha}{t_0}(\phi).
\end{equation}
Since this holds for any bounded $F$ supported in $V_{\phi_0}$ and any $\phi_0$, we get the desired result.
\end{proof}

\begin{lmm} \label{lem:proof_f_n_construct}
  There exist $f_n\in \Diff^3(\T)$ such that $f_n\to f$ in $\Diff^1(\T)$ as $n\to \infty$ and which satisfy
  \eqref{eq_delta_approx_1} and \eqref{eq_delta_approx_2}.
\end{lmm}

\begin{proof}
Define $f_n(\tau) = (\delta_{n} \ast f)(\tau) - (\delta_{n} \ast f)(0)$, where $\delta_n\in C^{\infty}(\T)$ is an approximate identity as $n\to \infty$. 
In other words, $\delta_n \geq 0$ and for any $h\in C(\T)$,
\begin{equation}\label{eq_f_n_limit_delta_def}
\lim_{n\to \infty}\left\|\left(\delta_n \ast h\right) - h\right\|_{C(\T)} = 0.
\end{equation}

Using the definition of the Schwarzian derivative $\Schw_{f_n}$,
\begin{multline}\label{eq_f_n_schw_decomposition}
\int_{\T} \left[ \Schw_{f_n} \big(\phi(\tau)\big) + 2\alpha^2 \big(f_n'(\phi(\tau))^2-1\big)\right]\phi'^{\,2}(\tau) \d \tau\\
=
\int_{\T}\frac{f_n'''\big(\phi(\tau)\big)}{f_n'\big(\phi(\tau)\big)} \phi'^{\,2}(\tau)\d\tau
+
\int_{\T}\left[-\frac{3}{2}\left(\frac{f_n''\big(\phi(\tau)\big)}{f_n'\big(\phi(\tau)\big)}\right)^2
+
2\alpha^2 \big(f_n'(\phi(\tau))^2-1\big)\right]\phi'^{\,2}(\tau) \d \tau \\
= 
I^{(1)}_n(\phi) +I^{(2)}_n(\phi)
\end{multline}
It is easy to see that $\sup_{n\in \N, \tau\in (0,1)}\{f_n'(\tau), 1/f_n'(\tau), |f_n''(\tau)|\}< \infty$, that $f_n'\to f'$ uniformly on $\T$ and that $f_n''\to f''$ pointwise on $(0, 1)$. 
Therefore, for any $\phi\in \Diff^1(\T)$,
\begin{equation}\label{eq_f_n_limit_2}
\lim_{n\to \infty}
I^{(2)}_n(\phi)
=
\int_{\T}\left[-\frac{3}{2}\left(\frac{f''\big(\phi(\tau)\big)}{f'\big(\phi(\tau)\big)}\right)^2
+
2\alpha^2 \big(f'(\phi(\tau))^2-1\big)\right]\phi'^{\,2}(\tau) \d \tau.
\end{equation}
Furthermore, if we denote $\psi_{\phi} = \phi^{-1}$ then
\begin{equation}\label{eq_f_n_int_2_rewrite}
I^{(1)}_n(\phi)
=
\int_{\T} \frac{f_n'''(\tau)}{f_n'(\tau)\psi_{\phi}'(\tau)}\d\tau
\end{equation}
Notice that if we denote 
\begin{equation}
g(\tau) = \begin{cases}
    f'''(\tau)       & \quad \text{if } \tau \neq 0\\
    0  & \quad \text{if } \tau = 0,
  \end{cases}
\end{equation}
then, in the sense of distributions, $f''' = \left[f''(0)-f''(1)\right]\delta + g$, where $\delta$ is the Dirac delta function at~$0$.
Thus,
\begin{equation}
f_n'''(\tau) = \left[f''(0)-f''(1)\right]\delta_n(\tau) + \left(\delta_n\ast g \right)(\tau).
\end{equation}
Notice that both $\delta_n$ and $\delta_n\ast g$ are bounded in $L^1(\T)$ uniformly in $n$, hence $f_n'''$ are bounded in $L^1(\T)$ uniformly in $n$ too.
Therefore, for any $\phi\in \Diff^1(\T)$
\begin{equation}\label{eq_f_n_limit_1}
\begin{aligned}
\lim_{n\to\infty}
I^{(1)}_n(\phi)
&=
\lim_{n\to\infty}
\int_{\T} \frac{f_n'''(\tau)}{f'(\tau) \psi_{\phi}'(\tau)}\d\tau 
+
\int_{\T} \frac{f_n'''(\tau)}{\psi_{\phi}'(\tau)} \left(\frac{1}{f_n'(\tau)}-\frac{1}{f'(\tau)}\right)\d\tau \\
&=
\frac{f''(0)-f''(1)}{f'(1)\psi_{\phi}'(0)}
+
\int_{\T\backslash \{t_0\}} \frac{f'''(\tau)}{\psi_{\pi}'(\tau)f'(\tau)}\d\tau\\
&=
\left[\frac{f''(0)}{f'(0)}-\frac{f''(1)}{f'(1)}\right]\phi'(t_0)
+
\int_{\T\backslash \{t_0\}}\frac{f'''\big(\phi(\tau)\big)}{f'\big(\phi(\tau)\big)} \phi'^{\,2}(\tau)\d\tau,
\end{aligned}
\end{equation}
where in the second equality we used that $\frac{1}{f_n'}-\frac{1}{f}$ converges to $0$ uniformly on $\T$ and \eqref{eq_f_n_limit_delta_def}.
Combining \eqref{eq_f_n_schw_decomposition}, \eqref{eq_f_n_limit_2} and \eqref{eq_f_n_limit_1} we obtain \eqref{eq_delta_approx_1}.

It is also easy to see that if $U\subset\Diff^1(\T)$ is open and such that $\sup_{\phi\in U, \tau\in \T} \{|\phi'(\tau)|, 1/|\phi'(\tau)|\} <\infty$, then from $\sup_{n\in \N, \tau\in (0,1)}\{f_n'(\tau), 1/f_n'(\tau), |f_n''(\tau)|\}< \infty$ it follows that
\begin{equation}\label{eq_f_n_bound_2}
\sup_{n\in\N, \, \phi\in U}\left| I^{(2)}_n(\phi)\right| < \infty.
\end{equation}
Moreover, since $f_n'''$ are bounded in $L^1(\T)$ uniformly in $n$, from \eqref{eq_f_n_int_2_rewrite} we also obtain that
\begin{equation}\label{eq_f_n_bound_1}
\sup_{n\in\N,\, \phi\in U}\left| I^{(1)}_n(\phi)\right| < \infty.
\end{equation}
Combining \eqref{eq_f_n_schw_decomposition}, \eqref{eq_f_n_bound_2} and \eqref{eq_f_n_bound_1} we get \eqref{eq_delta_approx_2}.
\end{proof}

\subsection{Partition function from change of variables}\label{sec:partition-function}

In this section, we prove Theorem~\ref{thm:partition-function}.
Given $t_0\in \T$, we write
\begin{equation}
  \label{eq:54}
  [F(\phi)]_{\sigma^{2}}^{\alpha,t_0}
  \coloneqq
  \int F(\phi) \, \dd\measPin{\sigma^2}{\alpha}{t_0}(\phi)
  =
  \int F(\phi) \exp\left\{\frac{2\alpha^{2}}{\sigma^{2}}\int \phi'^{\, 2}(t) \, \d t\right\}\, \dd\measPin{\sigma^2}{0}{t_0}(\phi)
\end{equation}
for the unnormalised expectation value with respect to the $\alpha$-orbital measure with pinning at $t_0$,
i.e.~these measures are supported on functions $\phi$ with $\phi(t_0) = 0$.
For functionals $F$ that depend only on the gradient $\phi'$, the expectations coincide with unpinned measure, which we denote by $[F]_{\sigma^2}^{\alpha}$.

  \begin{prp}[Boundary defect trick] \label{prop:11}
    Let $t_0 \in \T$. For $\alpha^{2} < \pi^{2}$, we have
    \begin{equation} \label{eq:108}
      \big[G(f_\alpha \circ \phi)\big]_{\sigma^2}^{\alpha,t_0}
      =
      \frac{\alpha}{\sin \alpha}
      \left[ G(\phi) \, \exp\left\{\frac{8\sin^{2}(\alpha/2)}{\sigma^{2}}\, \phi'(t_0)\right\}\right]_{\sigma^2}^{0,t_0},
\end{equation}
where
\begin{equation}
  f_\alpha(t)=
  \frac{1}{2}\left(\frac{1}{\tan \frac{\alpha}{2}} \tan\left(\alpha\left(t-\tfrac{1}{2}\right)\right)+1\right),
  \end{equation}
  and $\alpha/\sin\alpha=1$ and $f_\alpha(t)=t$ if $\alpha=0$. (Note that the function $f_\alpha$ is real even if $\alpha \in i\R$.)
  \end{prp}

\begin{proof}
It is easy to check that $f=f_\alpha$ defined in the statement satisfies $f(0)=0$, $f(1)=1$, and
\begin{equation}
\Schw_f(t) = 2\alpha^2,
\qquad
f'(0) = f'(1) = \frac{\alpha}{\sin \alpha},
\qquad
-\frac{f''(0)}{f'(0)} = \frac{f''(1)}{f'(1)} = 2\alpha \tan \frac{\alpha}{2}.
\end{equation}
It follows from the change of variables formula   \eqref{eq:64} for the pinned measure $\measPin{\sigma^2}{0}{t_0}$
that for any non-negative continuous functional $F$ on $\Diff^1(\T)$ we have 
\begin{multline} \label{eqLemmaMeasAlphaCalculationMeasureChange}
\int_{\Diff^1(\T)} F(\phi) \d\measPin{\sigma^2}{0}{t_0}(\phi) \\
= \frac{\sin \alpha}{\alpha}  \int_{\Diff^1(\T)} F(f\circ \phi) \,
\exp\left\{-\frac{4\alpha}{\sigma^2} \tan\frac{\alpha}{2} 
\cdot \phi'(t_0)+\frac{2\alpha^2}{\sigma^2}\int_{\T\backslash \{t_0\}} \phi'^{\, 2}(t) \d t\right\}
\d\measPin{\sigma^2}{0}{t_0}(\phi).
\end{multline}
Now we choose $F$ to be
\begin{equation}
F(\phi) =  G(\phi) \exp\left\{\frac{8 \sin^2 \frac{\alpha}{2}}{\sigma^2}\cdot \phi'(t_0)\right\},
\end{equation}
which guarantees that
\begin{equation}
F(f \circ \phi) = G(f\circ \phi) \exp \left\{\frac{4\alpha}{\sigma^2} \tan\frac{\alpha}{2} \cdot \phi'(t_0)\right\},
\end{equation}
and the claim follows.
\end{proof}

\begin{proof}[Proof of Theorem~\ref{thm:partition-function}]
  By assumption, the measure $\measOrb{\sigma^{2}}{\alpha}$ is regular, hence $\PartF{\sigma^2}{\alpha^{2}}\leq \PartF{\sigma^2}{\alpha_0^{2}} < \infty$ for $\alpha^{2} \leq \alpha^{2}_{0}$ with some $\alpha^{2}_{0} \in \RR$.
    We first consider $\alpha^{2} < \alpha_0^2$. %
    Then by \eqref{eq:108} with $G(\phi)=1$,
    \begin{equation}
      \label{eq:117}
      \begin{aligned}
        \de_{\alpha} \log\Big(\frac{\sin\alpha}{\alpha} [1]^{\alpha}_{\sigma^{2}}\Big)
        &=\de_{\alpha} \log\Big(\frac{\sin\alpha}{\alpha} [1]^{\alpha,t_0}_{\sigma^{2}}\Big)\\
        &=\de_{\alpha}\log \left[\exp\left\{\tfrac{8}{\sigma^2}\, \sin^{2}(\alpha/2) \phi'(t_{0})\right\}\right]^{0,t_0}_{\sigma^2}\\
        &= \frac{4}{\sigma^2}\sin(\alpha) \frac{
          \Big[\phi'(t_{0}) \exp\left\{\frac{8}{\sigma^2}\sin^{2}(\alpha/2) \phi'(t_{0})\right\} \Big]^{0,t_0}_{\sigma^2}}{\Big[ \exp\left\{\frac{8}{\sigma^2}\sin^{2}(\alpha/2) \phi'(t_{0})\right\} \Big]^{0,t_0}_{\sigma^2}}.
      \end{aligned}
    \end{equation}
    Applying \eqref{eq:108} again, with $G(\phi) = \phi'(0)$ in the numerator and $G(\phi)=1$ in the denominator above,
    \begin{equation}
      \begin{aligned}
        \de_{\alpha} \log\Big(\frac{\sin\alpha}{\alpha} [1]^{\alpha}_{\sigma^{2}}\Big)
        &= \frac{4}{\sigma^2}\sin(\alpha) \frac{[f_\alpha'(\phi(t_0)) \phi'(t_0)]^{\alpha,t_0}_{\sigma^2}}{[1]^{\alpha,t_0}_{\sigma^2}}\\
         &= \frac{4}{\sigma^2}\sin(\alpha) f_\alpha'(0) \frac{[\phi'(t_0)]^{\alpha,t_0}_{\sigma^2}}{[1]^{\alpha,t_0}_{\sigma^2}}\\
           &= \frac{4}{\sigma^2}\sin(\alpha) \frac{\alpha}{\sin(\alpha)} \frac{[\phi'(t_0)]^{\alpha}_{\sigma^2}}{[1]^{\alpha}_{\sigma^2}}
        ,
      \end{aligned}
    \end{equation}
    where we used that the law of $\phi'(\cdot)$ is the same under $[\cdot]^{0,t_{0}}_{\sigma^{2}}$ and the unquotiented measure $[\cdot]^{0}_{\sigma^{2}}$.
    Using that the left-hand side is independent of $t_0$ and that $\phi \in \Diff^1(\T)$,
    \begin{equation}
      1 = [\phi(1)-\phi(0)]^\alpha_{\sigma^2} = \int_{\T} \dd t_0\, [\phi'(t_{0})]^{\alpha}_{\sigma^2},
    \end{equation}
    therefore
    \begin{equation}
              \de_{\alpha} \log\Big(\frac{\sin\alpha}{\alpha} [1]^{\alpha}_{\sigma^{2}}\Big) = \frac{4\alpha}{\sigma^2}.
    \end{equation}
    It follows that, for $\alpha^2\leq\alpha_0^2$,
    \begin{equation}
      \label{eq:114}
      \frac{\sin\alpha}{\alpha} \exp\left\{-\frac{2\alpha^{2}}{\sigma^{2}}\right\} [1]_{\sigma^{2}}^\alpha
      =\frac{\sin\alpha_0}{\alpha_0} \exp\left\{-\frac{2\alpha_0^{2}}{\sigma^{2}}\right\} [1]_{\sigma^2}^{\alpha_0}.
    \end{equation}
    To extend the identity to general $\alpha^2\in \R$, we note that if $\alpha_1^2<\alpha_0^2$, then
    \begin{equation}
      \frac{\alpha_1\sin\alpha}{\alpha \sin\alpha_1}
      \cdot
      \frac{\PartF{\sigma^2}{\alpha}}{\PartF{\sigma^2}{\alpha_1}} =
      \frac{\left[\exp\left\{\frac{2(\alpha^{2}-\alpha_1^2)}{\sigma^{2}} \int \phi'^{\, 2}(\tau) \, \dd\tau\right\}\right]_{\sigma^{2}}^{\alpha_1}}{[1]_{\sigma^{2}}^{\alpha_1}}
      =
      \frac{\left[e^{\lambda X}\right]_{\sigma^{2}}^{\alpha_1}}{[1]_{\sigma^{2}}^{\alpha_1}}, \qquad \lambda =\alpha^2-\alpha_1^2
    \end{equation}
    is the moment-generating function of a non-negative random variable $X$. %
    For general $\alpha^2$ the result therefore follows by analytic continuation, see Lemma~\ref{lemmaAnalyticContinuationExponentialMoments} below.
  \end{proof}

\begin{lmm}\label{lemmaAnalyticContinuationExponentialMoments}
Let $\TechMeas$ be a non-negative measure on $\R_+$. 
Assume that there exists $\eps>0$ such that the exponential moment generating function $F(z) = \int \exp \left(z X\right)\d\TechMeas(X)$ exists for all $z\in [0, \eps)$. 
Assume further that for some $R>0$, $F(z)$ can be analytically continued for all $z\in \D_R$. 
Then $\int \exp \left(z X\right)\d\TechMeas(X)$ converges absolutely for all $z\in \D_R$, and is equal to the analytic continuation of $F(z)$ to $\D_R$.
\end{lmm}
\begin{proof}
Since $\TechMeas$ supported on $\R_+$, by Tonelli's Theorem,
\begin{equation}
F(z) = \sum_{n\geq 0} z^n \frac{\int X^n \d\TechMeas(X)}{n!},
\end{equation}
for $z\in [0, \eps)$. Given that $F(z)$ is analytic in $\D_R$, we conclude that the right-hand side converges absolutely for $z\in \D_R$. Thus, by Tonelli's Theorem again, $\E \exp \left(z X\right)$ converges for $z\in [0,R)$ and is equal to $F(z)$. We can continue the equality for the whole disc $\D_R$, since $|z^n X^n|\leq |z|^n X^n$, and all expressions are absolutely convergent in $\D_R$.
\end{proof}

\subsection{Uniqueness of orbital measures}

In this section, we prove Theorem~\ref{thm:uniqueness}.

\begin{proof}[Proof of Theorem~\ref{thm:uniqueness}]
By Lemma~\ref{lem:orbital-measure-family}, it suffices to consider $\alpha=0$.
Let $\measPin{\sigma^2}{0}{t_0}$ denote the pinned measure \eqref{eq:59}. By \eqref{eq:61},
it suffices to show that  $\measPin{\sigma^2}{0}{t_0}$ is uniquely determined by \eqref{eq:64}.
Roughly, the idea of the proof is to characterise the law of $\phi'$ by calculating Laplace transforms
\begin{equation}
  \int  \exp\left\{\frac{1}{\sigma^2}\int_{\T\backslash\{t_0\}}  q\big(\phi(\tau)\big)\phi'^{\,2}(\tau) \d \tau\right\}
  \d \measPin{\sigma^2}{0}{t_0}(\phi)
\end{equation}
for sufficiently general $q\in C^{\infty}[0,1]$.
This will be done by an application of the change of variables formula with appropriately chosen $f_{q} \in \Diff^{3}[0,1]$.
\smallskip

\emph{Construction of $f_{q}$:}
We claim that for any $q \in C^\infty(\T)$ with $q(t) \leq 0$ for all $t\in [0,1]$, there exists $f_q \in \Diff^3[0,1]$ with $f_q'(0)=f_q'(1)$ and $f''(1) < 0<f''(0)$ such that
$\Schw(f_q, s\big) = q(s)$ for all $s\in [0,1]$. %

To construct an appropriate $f_{q}$, we make use of the relationship between the Schwarzian derivative and Hill's equation:
Let $g_i$ be two solutions to $g_i''(t) = -\tfrac12 q(t)g_i(t)$ with $g_1(0)=1,g_1'(0)=0$ and $g_2(0)=0,g_2'(0)=1$.
It follows from $q \leq 0$ that $g_1$ and $g_2$ are both positive on $(0,1]$.

For some $a,b,c,d \in \R$ to be chosen below we let
\begin{equation}
  f_q(t)
  = \frac{a g_2(t) + b g_1(t)}{c g_2(t) + d g_1(t)}
  = \frac{h_2(t)}{h_1(t)}
\end{equation}
with $h_2(t) = a g_2(t) + b g_1(t)$ and $h_1(t) = c g_2(t) + d g_1(t)$.
Since $g_{1}(0), g_{2}(1) > 0$ and $g_{2}(0) = 0$ we can choose $c,d \in \RR$ such that $h_1(0) = h_1(1) > 0$.
Since $q \leq 0$, it then follows that $h_1(t) > 0$ for all $t\in [0,1]$. %
Further we can choose $a, b$ such that $h_2(0) = 0, h_2(1) = h_1(1)$, namely $b=0$ and $a=h_1(1)/g_2(1)$.

Let $D = ad-bc$ and
note that the Wronskian of $(h_{1},h_{2})$ %
is constant in time: $h_1(t)h_2'(t) - h_1'(t)h_2(t) = D [g_1(t)g_2'(t) - g_1'(t)g_2(t)]  = D [g_1(0)g_2'(0) - g_1'(0)g_2(0)] = D$
(which follows from $h_i'' = -\frac12 q h_i$).
Consequently, $f_{q}' = (h_{1}h_{2}' - h_{1}'h_{2})/h_{1}^{2} = D/h_{1}^{2}$ and
\begin{equation}
  \Schw(f_q,t)
  = \log(f_{q}')'' - \tfrac12 \log(f_{q}')'^{\,2}
  = -2 \frac{h_1''(t)h_1(t)} {h_1(t)^2}
  =   q(t).
\end{equation}
Since $h_1(0)=h_1(1)$ we also have $f_q'(0) = f_q'(1)$.
Since we chose $a, b$ such that $h_2(0) = 0, h_2(1)=h_1(1)$. Then $f_q(0)=0$ and $f_q(1)=1$.
Since $f_q'=D/h_1^2$ has constant sign, also $f_q'(t)>0$ for all $t$, i.e., $f \in \Diff^1[0,1]$.
Since $h_1(t)>0$ for all $t$, we also have $h_1''(t) = -\tfrac12 q(t)h_1(t) > 0$ for all $t$.
Using $h_1(0) = h_1(1)$, this implies  $h_1'(0)<0$ and $h_1'(1)>0$.
Thus $f_q''(1)<0< f_q''(0)$.

\smallskip
\emph{Characterising the law:}
Therefore, by the change of variables formula \eqref{eq:64}, using that $\phi(t_0)=0$ and $f_q'(0)=f_q'(1)$,
\begin{equation}\label{eq:64-bis}
\begin{aligned}
  &\exp\left\{\frac{1}{\sigma^2}\int_{\T\backslash\{t_0\}} \left[ q\big(\phi(\tau)\big) \right]\phi'^{\,2}(\tau) \d \tau
  \right\}
    \d  \measPin{\sigma^2}{0}{t_0}(\phi)
    \\
  &= f_q'(0)
  \exp\left\{-\frac{1}{\sigma^2}\left[\frac{f_q''(0)}{f_q'(0)}-\frac{f_q''(1)}{f_q'(1)}\right]\phi'(t_0) \right\}
    \d f^{\ast} \measPin{\sigma^2}{0}{t_0}(\phi)
  \\
    &= f_q'(0)
      \exp\left\{-\frac{1}{\sigma^2}\left[f_q''(0)-f_q''(1)\right]\phi'(t_0) \right\}
  \d\measPin{\sigma^2}{0}{t_0}(\phi)
  .
\end{aligned}
\end{equation}
By \eqref{eq:108} with $G=1$ and $\alpha = \alpha(q) = i \tilde\alpha (q) \in i\R$
such that $8\sin^2(\alpha/2) = -8\sinh^2(\tilde \alpha/2) = [f''_q(1)-f''_q(0)]<0$,
the right-hand side can be written as
\begin{equation}
    f_q'(0) \frac{\sin\alpha(q)}{\alpha(q)}
    \int \exp\left\{\frac{2\alpha(q)^2}{\sigma^2} \int_{\T} \phi'^{\, 2}(\tau)\d\tau\right\} \d\measPin{\sigma^2}{0}{t_{0}}(\phi)
\end{equation}
By \eqref{eq:4}, this equals
\begin{equation}
  f_q'(0) \frac{\sin\alpha(q)}{\alpha(q)}
    \PartF{\sigma^2}{\alpha(q)}
    = f_q'(0) \exp\left\{2\alpha(q)^{2}/\sigma^{2}\right\} \PartF{\sigma^{2}}{0}.
\end{equation}
Taking the inverse $\psi = \phi^{-1}$, we obtain from \eqref{eq:64-bis}  that
\begin{equation}
  \begin{aligned}
  \exp\left\{\frac{1}{\sigma^2} \int_{(0, 1)} \frac{q(s)}{\psi'(s)}\d s \right\} \d\measPin{\sigma^2}{0}{t_{0}}(\phi)
    &= \exp\left\{\frac{1}{\sigma^2}  \int_{\T} q \Big(\phi(\tau)\Big)\Big(\phi'(\tau)\Big)^2  \right\} \d\measPin{\sigma^2}{0}{t_{0}}(\phi)
    \\
    &=
      f_q'(0) \exp\left\{2\alpha(q)^{2}/\sigma^{2}\right\} \PartF{\sigma^{2}}{0}.
  \end{aligned}
\end{equation}
Therefore %
we know the characteristic function of $\frac{1}{\psi'(s)}$, which determines the measure uniquely.
\end{proof}

\section{Schwarzian measure}
\label{sec:schwarzian}

As discussed in Remark~\ref{rmrk:U1},
the unquotiented $\alpha$-orbital measure %
is invariant under a subgroup $G_{\alpha} \subseteq \Diff^{1}(\T)$, and
the corresponding orbital measure should be understood as the quotient measure on the space $\Diff^{1}(\T)/G_{\alpha}$.
In this section, we consider the exceptional case of the Schwarzian Field Theory $\alpha = \pi$
where $G_\alpha = \SL(2,\R)$ is noncompact.
By extending the results of the previous section, we characterize the corresponding measure and compute the total mass of the quotient measure (the partition function).
The compact case $G_\alpha={\rm U}(1)$ would be straightforward.

\subsection{Definition of the Schwarzian measure}

Let  $\FMeasSL{\sigma^2}$ be an unquotiented $\pi$-orbital measure, see Definition~\ref{def:orbital-measure},
not assumed to be regular.
From the $\SL(2,\R)$-invariance of this measure, it follows that $\FMeasSL{\sigma^2}$ can be decomposed into a product of the Haar measure of $\SL(2,\R)$
with its quotient by $\SL(2,\R)$.
Since $\SL(2,  \R)$ is not compact, we need to choose the normalisation of the Haar measure $\Haar$,
and we  work with the normalisation \eqref{eqSLHaarMeasureFormula}.
The precise statement for the existence of a quotient measure is then as follows. 

\begin{prp} \label{propMeasureFactor}
  There exists a unique Borel measure $\FMeas{\sigma^2}$ on $\Diff^1(\T)/\SL(2,\R)$ such that for any continuous $F\colon \Diff^{1}(\T) \to [0,\infty]$,
  \begin{equation}
    \label{e:measurefactor}
    \int\limits_{\mathclap{\Diff^1(\T)}}\d\FMeasSL{\sigma^2}(\phi)\, F(\phi)
    \hspace*{1em}
    =
    \hspace*{2.5em}
    \int\limits_{\mathclap{\Diff^1(\T)/\SL(2,\R)}}\d\FMeas{\sigma^2}(\CnjCl{\phi})
    \hspace*{1em}
    \int\limits_{\mathclap{\SL(2,\R)}} \d\nu_H(\psi)\, F(\psi\circ\phi),
  \end{equation}
  where the right-hand side is well defined since the second integral only depends on $\phi \in \Diff^1(\T)$ through the conjugacy class $\CnjCl{\phi}$ in $\Diff^1(\T)/\SL(2,\R)$.
\end{prp}

\begin{proof}
  Since the space $\Diff^1(\T)$ is not locally compact, we could not locate a reference for the existence of the quotient measures (which would be standard
  in the locally compact situation). Therefore,
  in Appendix~\ref{sec:quotient-measure}, we establish a sufficiently general result (Proposition~\ref{prop:existence-quotient-measure}) about such quotient measures.
  Note that the assumptions of that result are satisfied in our context:
  $\FMeasSL{\sigma^2}$ is a Radon measure on $\Diff^1(\T)$ by assumption, %
  $\Diff^{1}(\T)$ is a complete separable metric space, and $\SL(2,\R)$ acts continuously and properly from the right (note the discussion after \eqref{eqSLHaarMeasureFormula}).
  Moreover, $\SL(2,\R)$ is unimodular and $\FMeasSL{\sigma^2}$ is invariant under its action by post-compositions.
\end{proof}

\begin{prp} \label{prp_schwarzian_uniqueness}
  There is a uniquely determined (up to normalisation)
  finite Borel measure $\FMeas{\sigma^2}$ on $\Diff^1(\T)/\SL(2,\R)$ such that the lift $\FMeasSL{\sigma^2}$ defined by \eqref{e:measurefactor}
  satisfies the change of variables formula \eqref{eqDiffeoMeasureChange}.
  Moreover, the lift $\FMeasSL{\sigma^2}$ is regular.
\end{prp}

Given $\FMeas{\sigma^2}$ and its lift $\FMeasSL{\sigma^2}$ as in the proposition,
$\measOrb{\sigma^2}{\pi} =\FMeasSL{\sigma^2}$ is a regular $\pi$-orbital measure %
and $\measOrb{\sigma^2}{\alpha}$ is the $\alpha$-orbital measure associated with $\measOrb{\sigma^2}{\pi}$  through the change of orbit relation \eqref{eq:57}.

By Section~\ref{sec:construction}, in fact $\measOrb{\sigma^2}{0}$ turns to be a Brownian Bridge times a Lebesgue measure on $\T$,
and it is natural to fix its total mass as $\PartF{\sigma^{2}}{0} = 1/\sqrt{2\pi\sigma^2}$.

\begin{defn} \label{defn_schwarzian}
  The Schwarzian measure is given by $\FMeas{\sigma^2}$, and
  we normalize this measure such
  that  $\measOrb{\sigma^2}{0}$ as defined above  has total mass $\PartF{\sigma^{2}}{0} = 1/\sqrt{2\pi\sigma^2}$.
\end{defn}

Recall that having an $\alpha$-orbital measure $\measOrb{\sigma^2}{\alpha}$ we define the corresponding pinned measure $\measPin{\sigma^2}{\alpha}{0}$ via \eqref{eq:59}.
The following proposition is useful to compute expectations with respect to the quotient measure.

\begin{prp} \label{prpIntegralRegularization}
Let $F\colon\Diff^1(\T)/\SL(2, \R) \to [0, \infty]$ be a continuous function. Then 
\begin{equation}
\int_{\Diff^1(\T)/\SL(2,\R)} F(\CnjCl{\phi}) \d\FMeas{\sigma^2}(\CnjCl{\phi}) 
=
 \lim_{\alpha \to \pi-} \frac{4\pi(\pi-\alpha)}{\sigma^2}\int_{\Diff^1(\T)} F(\CnjCl{\phi}) \d\measPin{\sigma^2}{\alpha}{0}(\phi).
\end{equation}
\end{prp}

Using this proposition together with Theorem~\ref{thm:partition-function} for the partition function of the unquotiented orbital measures with $\alpha<\pi$,
we compute the partition function of $\FMeas{\sigma^2}$ as follows.

\begin{prp} \label{prp_schwarzian_partition}
  The measure $\FMeas{\sigma^2}$  has total mass %
  \begin{equation}
    \FMeas{\sigma^2}\Big(\Diff^1(\T)/\SL(2, \R)\Big) 
    = \left(\frac{2\pi}{\sigma^2}\right)^{3/2} \exp\left(\frac{2\pi^2}{\sigma^2}\right).
  \end{equation}
\end{prp}

\begin{proof} %
We apply Proposition~\ref{prpIntegralRegularization} and then Theorem~\ref{thm:partition-function}: %
  \begin{equation}
    \begin{aligned}
      \FMeas{\sigma^2}\Big(\Diff^1(\T)/\SL(2, \R)\Big) 
      &=
        \lim_{\alpha \to \pi-} \frac{4\pi(\pi-\alpha)}{\sigma^2}     \measOrb{\sigma^2}{\alpha}\Big(\Diff^1(\T)\Big).\\
      &= \lim_{\alpha \to \pi-} \frac{4\pi(\pi-\alpha)}{\sigma^2} \frac{\alpha}{\sin{\alpha}} \frac{e^{2\alpha^{2}/\sigma^{2}}}{\sqrt{2\pi \sigma^{2}}}\\
      &= \left(\frac{2\pi}{\sigma^2}\right)^{3/2} \exp\left(\frac{2\pi^2}{\sigma^2}\right).
    \end{aligned}
  \end{equation}
  This completes the proof.
\end{proof}

\subsection{Proof of Propositions~\ref{prp_schwarzian_uniqueness} and \ref{prpIntegralRegularization}}
\label{sec:prpIntegralRegularization}

The proposition follows from the following lemmas.
Recall the parametrisation \eqref{eqSLTransformFormula}
of $\SL(2, \R)$
and that in this parametrisation, the Haar measure takes the form \eqref{eqSLHaarMeasureFormula}.

\begin{lmm}\label{LemmaCalculationsDVSquareIntegral}
  For $\phi_{z,\ang}$ as in \eqref{eqSLTransformFormula}, with $\rho=|z|<1$,
  \begin{equation}\label{eqLemmaCalculationsDVSquareIntegral}
    \int_{\T}\phi_{z,\ang}'^{\, 2}(s)  \d s = \frac{1+\rho^2}{1-\rho^2}.
\end{equation}
Moreover, as $\rho\nearrow 1$ the functions $\frac{1-\rho^2}{1+\rho^2}\phi_{z,0}'^{\, 2}$ are an approximate identity on $\T$, i.e.\
for any $f \in C(\T)$,
\begin{equation} \label{eqLemmaCalculationsDVSquareIntegral-f}
  \lim_{\rho\nearrow 1} \left[ \frac{1-\rho^2}{1+\rho^2}\int_{\T} \phi_{z,0}'^{\, 2}(s) f(s) \d s \right] = f(\theta), \qquad z = \rho\, e^{i 2\pi \theta},
\end{equation}
uniformly in $\theta \in\T$.
\end{lmm}

\begin{proof}
  Since
  \begin{equation}
    e^{i2\pi \phi_{z,\ang}(t)} = e^{i2\pi \ang}\frac{e^{i2\pi t}-z}{1-e^{i2\pi t}\bar{z}},
  \end{equation}
  it suffices to consider M\"{o}bius transformations of $\D_1$, given by 
  \begin{equation}
    w \to G_z(w) = \frac{w-z}{1-w\bar{z}}.
  \end{equation}
  In order to prove \eqref{eqLemmaCalculationsDVSquareIntegral} we then need to show that
  \begin{equation}\label{eqLemmaCalculationsDVSquareIntegral-pf}
    \frac{1}{(2\pi)^2}\int_{\T} \left|\frac{\d}{\d t}\left(\frac{e^{i2\pi t}-z}{1 - e^{i2\pi t}\,\bar{z}}\right)\right|^2\d t = \frac{1+|z|^2}{1-|z|^2}.
  \end{equation}
  To see this, expand 
  \begin{equation}
    \frac{e^{i2\pi t}-z}{1 - e^{i2\pi t}\,\bar{z}} 
    = 
    -z+\sum_{n=1}^{\infty}e^{i 2\pi n t} \bar{z}^{n-1}(1-z\,\bar{z}).
  \end{equation}
  Therefore,
  \begin{equation}
    \frac{1}{(2\pi)^2}\int_{\T} \left|\frac{\d}{\d t}\left(\frac{e^{i2\pi t}-z}{1 - e^{i2\pi t}\,\bar{z}}\right)\right|^2\d t  
    =
     \sum_{n=1}^{\infty}n^2 |z|^{2n-2}(1-|z|^2)^2.
  \end{equation}
  The right-hand side equals
    \begin{equation} \label{eqLemmaCalculationsDVSquareIntegral-pf2}
      \sum_{n=0}^\infty (n+1)^2 |z|^{2n} - 2\sum_{n=1}^\infty n^2 |z|^{2n} + \sum_{n=2}^\infty (n-1)^2 |z|^{2n}
      = 1 + 4|z|^2 - 2 |z|^2 + \sum_{n=2}^\infty 2|z|^{2n} 
        = \frac{1+|z|^2}{1-|z|^2},
    \end{equation}
    which gives \eqref{eqLemmaCalculationsDVSquareIntegral-pf}. The claim \eqref{eqLemmaCalculationsDVSquareIntegral-f} follows similarly.
    Since $C^{\infty}(\T)$ is dense in $C(\T)$, it suffices to assume that $f \in C^{\infty}(\T)$. Then
    \begin{equation}
      f(t) = \sum_{k \in \Z} e^{i2\pi k t} \hat f_k,
    \end{equation}
    with $(\hat f_k)_{k\in\Z} \in \ell^1$. Therefore 
    \begin{align}
      \frac{1}{(2\pi)^2}\int_{\T} \left|\frac{\d}{\d t}\left(\frac{e^{i2\pi t}-z}{1 - e^{i2\pi t}\,\bar{z}}\right)\right|^2 f(t)\, \d t
      &= \sum_{k} \hat f_k \sum_{n=1}^\infty \mathbf{1}_{n\geq -k}\, n(n+k) z^{k} |z|^{2n-2} (1-|z|^2)^2
        \nonumber\\
      &= \sum_{k} \hat f_k e^{i2\pi k \theta} q_k(\rho)
    \end{align}
    with
    \begin{align}
      q_k(\rho) = \sum_{n=1}^\infty \mathbf{1}_{n\geq -k}\, n(n+k) \rho^{2n-2+k} (1-\rho^2)^2.
    \end{align}
    Analogous to \eqref{eqLemmaCalculationsDVSquareIntegral-pf2},
    \begin{multline}
      q_k(\rho) = \sum_{n=0}^\infty \mathbf{1}_{n \geq -k-1}\, (n+1)(n+1+k) \rho^{2n+\rho} - 2 \sum_{n=1}^\infty \mathbf{1}_{n \geq -k}\, n(n+k) \rho^{2n+k}
       \\
      + \sum_{n=2}^\infty \mathbf{1}_{n \ge -k+1}\, (n-1)(n-1+k) \rho^{2n+k}
      =
      \sum_{n=2}^\infty \mathbf{1}_{n \geq -k}\, 2\rho^{2n+k}  + O(1).
    \end{multline}
    Therefore as $\rho \nearrow 1$,
    \begin{equation}
      (1-\rho^2)q_k(\rho) = O(1), \qquad \lim_{\rho \nearrow 1}(1-\rho^2)q_k(\rho) = 2.
    \end{equation}
    By dominated convergence, \eqref{eqLemmaCalculationsDVSquareIntegral-f} follows.
\end{proof}

\begin{lmm}\label{lemmaCalculationsDV}
For $\phi \in \Diff^1(\T)$, define
\begin{equation} \label{e:Ddef}
\dense{\alpha} (\phi) = 
\frac{4\pi(\pi-\alpha)}{\sigma^2}
\int_{\SL(2, \R)}
\exp\left\{-\frac{2(\pi^2-\alpha^2)}{\sigma^2} 
\int_{\T}(\psi\circ \phi)'^{\, 2}(t)\d t\right\} 
\d\Haar(\psi) .
\end{equation}
Then $D^\alpha$ is $\SL(2,\R)$ invariant and the following holds for $\alpha <\pi$:
\begin{enumerate}
\item \label{lemmaSLAvIneq}
For any $\phi \in \Diff^1(\T)$,
\begin{equation}\label{eq:37}
\dense{\alpha}(\phi)\leq  
\frac{2\pi}{\pi+\alpha}.
\end{equation}
\item For any $\phi \in \Diff^1(\T)$,
\begin{equation}\label{eq:38}
\lim_{\alpha\to\pi-}
\dense{\alpha}(\phi)
= 1.
\end{equation} 
\end{enumerate}
\end{lmm}
\begin{proof}
First, we prove \eqref{eq:37}. 
By \eqref{eqSLHaarMeasureFormula},
\begin{equation}
  \int_{\SL(2,\R)} F(\psi) \, \d \nu_H(\psi) =\int_{\substack{0\leq \rho <1 \\0 \leq \theta<1\\ 0\leq a <1}} F(\phi_{z,\ang}) \, \frac{4 \rho \d\rho \, \d\theta \, \d a }{(1-\rho^2)^2}.
\end{equation}
Recall the parameterisation \eqref{eqSLTransformFormula}.
Clearly one has $(\phi_{z,\ang}\circ \phi)' = (\phi_{z,0}\circ \phi)' = (\phi_{z,0}' \circ \phi) \,\phi'$.
Upon the reparametrisation $s = \phi(t)$, we get
\begin{equation}
\int_{\T}(\phi_{z,\ang}\circ \phi)'^{\, 2}(t)\d t 
= 
\int_{\T} \frac{\phi_{z,0}'^{\, 2}(s)}{(\phi^{-1})'(s)} \d s.
\end{equation}
We insert this into \eqref{e:Ddef} and apply Jensen's inequality to obtain
\begin{equation}
\begin{aligned}
\dense{\alpha}(\phi)
  &
    =
  \frac{4\pi(\pi-\alpha)}{\sigma^2}
  \int_{\substack{0\leq\rho<1\\
      0\leq \theta<1}} 
  \exp\left\{-\frac{2(\pi^2-\alpha^2)}{\sigma^2} \int_{\T} \frac{\phi_{z,0}'^{\, 2}(s)}{(\phi^{-1})'(s)} ds\right\}\frac{4 \rho \d\rho \, \d\theta \,}{(1-\rho^2)^2},\\
&\leq 
\frac{4\pi(\pi-\alpha)}{\sigma^2}
\int_{\substack{0\leq\rho<1\\
0\leq \theta<1}}
\int_{\T}\frac{1-\rho^2}{1+\rho^2}\, \phi_{z,0}'^{\, 2}(s)\,
\exp\left\{-\frac{2(\pi^2-\alpha^2)}{\sigma^2} \cdot\frac{1+\rho^2}{1-\rho^2} \cdot \frac{1}{(\phi^{-1})'(s)}\right\} 
  \d s \,\frac{4 \rho \d\rho \, \d\theta }{(1-\rho^2)^2}.
\end{aligned}
\end{equation}
Then, by Tonelli's theorem,
\begin{equation}
\begin{aligned}
  \dense{\alpha}(\phi)
  &\leq
\frac{4\pi(\pi-\alpha)}{\sigma^2}
\int_{0\leq \rho<1}\int_{\T}
\exp\left\{-\frac{2(\pi^2-\alpha^2)}{\sigma^2} \cdot\frac{1+\rho^2}{1-\rho^2} \cdot \frac{1}{(\phi^{-1})'(s)}\right\} 
\d s \, \frac{4 \rho \d\rho}{(1-\rho^2)^2} \\
&= 
\frac{4\pi(\pi-\alpha)}{\sigma^2}
\int_{\T}  \exp\left\{-\frac{2(\pi^2-\alpha^2)}{\sigma^2 (\phi^{-1})'(s)}\right\}\frac{\sigma^2 (\phi^{-1})'(s)}{2(\pi^2-\alpha^2)}\d s \\
&\leq 
\frac{4\pi(\pi-\alpha)}{\sigma^2}
\int_{\T} \frac{\sigma^2 (\phi^{-1})'(s)}{2(\pi^2-\alpha^2)}\d s \\
&= \frac{2\pi}{\pi+\alpha}.
\end{aligned}
\end{equation}
To find the limit of $\dense{\alpha}(\phi)$, recall that
$\left\{\frac{1-\rho^2}{1+\rho^2}\phi_{z,0}'^{\, 2}\right\}_{\rho\nearrow 1}$ is an approximate identity.
More precisely, 
\begin{equation}
\int_{\T} \frac{\phi_{z,0}'^{\, 2}(s)}{(\phi^{-1})'(s)}\d s 
= \frac{1+\rho^2}{1-\rho^2}\cdot \left(\frac{1}{(\phi^{-1})'(\theta)}+o(1)\right)
\qquad\text{ as }\rho\nearrow 1, 
\end{equation}
uniformly in $\theta \in \T$, by Lemma~\ref{LemmaCalculationsDVSquareIntegral}. Therefore,
as $\alpha \nearrow \pi$,
\begin{equation}
\begin{aligned}
&\int_{0\leq\rho<1}  \exp\left\{-\frac{2(\pi^2-\alpha^2)}{\sigma^2} \int_{\T} \frac{\phi_{z,0}'^{\, 2}(s)}{(\phi^{-1})'(s)} \d s\right\}\frac{4 \rho \d\rho}{(1-\rho^2)^2} \\
&\qquad = \exp\left\{-\frac{2(\pi^2-\alpha^2)}{\sigma^2 \big[(\phi^{-1})'(\theta)+o(1)\big]}\right\}
\frac{\sigma^2\big[(\phi^{-1})'(\theta)+o(1)\big]}{2(\pi^2-\alpha^2)}+O(1)\\
&\qquad= 
\frac{\sigma^2(\phi^{-1})'(\theta)}{2(\pi^2-\alpha^2)}\big(1+o(1)\big).
\end{aligned}
\end{equation}
Integration in $\theta$ finishes the proof.
\end{proof}

\begin{proof}[Proof of Proposition~\ref{prp_schwarzian_uniqueness}]
  By Theorem~\ref{thm:uniqueness} for uniqueness of regular unquotiented orbital measures,
  it suffices to prove that the lift defined by \eqref{e:measurefactor} is regular.
  This follow from the assumption that the quotient measure is finite and Lemma~\ref{lemmaCalculationsDV}.
  Indeed, the definition of regular (see Definition~\ref{def:regular}) is
  that there is $\alpha^2\in \R$ such that
  \begin{equation}
    \int \exp\left\{\frac{2(\alpha^2-\pi^2)}{\sigma^2}\int_{\T}\phi'^{\, 2}(\tau)\d\tau\right\} \d\FMeasSL{\sigma^2}(\phi) <\infty.
  \end{equation}
  By Proposition~\ref{propMeasureFactor} and  Lemma~\ref{lemmaCalculationsDV}, for any $\alpha<\pi$, this equals
  \begin{equation}
    \frac{\sigma^2}{4\pi (\pi-\alpha)} \int D^\alpha(\phi) \d\FMeas{\sigma^2}(\CnjCl{\phi})
    \leq
    \frac{\sigma^2}{2(\pi^2-\alpha^2)} \int \d\FMeas{\sigma^2}(\CnjCl{\phi})
  \end{equation}
  which is finite by since $\FMeas{\sigma^2}$ is assumed to be finite, and we used that $D^\alpha(\phi)$ only depends on $[\phi]$.
\end{proof}

\begin{proof}[Proof of Proposition \ref{prpIntegralRegularization}]
It follows from the definition of $\FMeasSL{\sigma^2}$ that 
\begin{equation}
\exp\left\{-\frac{2(\pi^2-\alpha^2)}{\sigma^2}\int_{\T} \phi'^{\, 2}(t)\d t \right\} 
\d\FMeasSL{\sigma^2}(\phi)
= 
\d \measOrb{\sigma^2}{\alpha}(\phi_0) \otimes \d \Theta,
\end{equation}
where
\begin{equation}
  \phi(\cdot) = \phi_0(\cdot)+\Theta, \text{ for } \Theta \in [0,1).
\end{equation}
Therefore,
\begin{equation}
\int_{\Diff^1(\T)} F(\CnjCl{\phi}) \d\measOrb{\sigma^2}{\alpha}(\phi) 
=
\int_{\Diff^1(\T)} F(\CnjCl{\phi}) 
\exp\left\{-\frac{2(\pi^2-\alpha^2)}{\sigma^2}\int_{\T} \phi'^{\, 2}(t)\d t \right\}
\d\FMeasSL{\sigma^2}(\phi).
\end{equation}
Using the definition of $D^\alpha$ given in \eqref{e:Ddef} and the factorisation of $\FMeasSL{\sigma^2}$ as stated in \eqref{e:measurefactor},
\begin{multline}\label{eqQuotientAlphaReg}
\frac{4\pi(\pi-\alpha)}{\sigma^2}
\int_{\Diff^1(\T)} F(\CnjCl{\phi}) 
\exp\left\{-\frac{2(\pi^2-\alpha^2)}{\sigma^2}\int_{\T} \phi'^{\, 2}(t)\d t \right\}
\d\FMeasSL{\sigma^2}(\phi)\\
=
\int_{\Diff^1{(\T)}/\SL(2, \R)} \dense{\alpha}(\phi) F(\phi) \d \FMeas{\sigma^2}(\phi).
\end{multline}

If $\int F\big(\CnjCl{\phi}\big) \d\FMeas{\sigma^2}\big(\CnjCl{\phi}\big)$ is finite, then using Lemma~\ref{lemmaCalculationsDV} and
the Dominated Convergence Theorem we obtain the desired result.
If, on the other hand, $\int F\big(\CnjCl{\phi}\big) \d\FMeas{\sigma^2}\big(\CnjCl{\phi}\big)$ is infinite, then we get the desired by Fatou's Lemma.
\end{proof}

\color{black}

\section{Construction of unquotiented orbital measures}
\label{sec:construction}

In this section we prove Theorem~\ref{thm:orbital-exist-bis}, i.e.\ the explicit realisation
of the unquotiented Virasoro $0$-orbital measure
(see Definition~\ref{def:orbital-measure}), as a version of what is known as the Malliavin--Shavgulidze measure.
For convenience we repeat the statement of Theorem~\ref{thm:orbital-exist-bis}.

\begin{thr} \label{thm:orbital-exist}
  Regular unquotiented Virasoro $\alpha$-orbital measures exist.
  In fact, the $0$-orbital measure can be realised as a function of a Brownian Bridge and a Lebesgue measure on $\T$ for the zero mode:
  \begin{equation}%
    \d\mu_{\sigma^2}(\phi) \coloneqq 
    \d \WS{\sigma^2}{0}{1}(\xi)
    \otimes \d\Theta, \qquad \text{with } \phi(t) = \Theta + \A_{\xi}(t)\enspace (\mathrm{mod}\, 1), \text{ for } \Theta\in [0, 1),
  \end{equation}
  where $\d\Theta$ is the Lebesgue measure on $[0,1)$ and with 
  \begin{equation} \label{defP-bis}
    \A(\xi)(t) \coloneqq \A_{\xi}(t) 
    \coloneqq \frac{\int_{0}^t e^{\xi(\tau)}\d\tau}{\int_{0}^1 e^{\xi(\tau)}\d\tau}.
  \end{equation}
\end{thr}

It is easy to see that the pinned $0$-orbital measure $\measPin{\sigma^{2}}{0}{t_{0}}$ (see Section~\ref{sec:pinned}) is obtained analogously by fixing $\Theta = -\A_{\xi}(t_{0})$.
In particular, $\d\measPin{\sigma^{2}}{0}{0}(\A_{\xi}) = \d\WS{\sigma^{2}}{0}{1}(\xi)$.
Importantly it satisfies the pinned version of the change of variable formula \eqref{eq:64} as in Proposition~\ref{prpWSMeasureChangePinned}.

To prove the theorem, we begin with the definition
of the unnormalised Brownian Bridge measure $\WS{\sigma^2}{a}{T}$ in Section~\ref{sec:BB}.
In Section~\ref{sec:MSmeas}, we then
show that $\mu_{\sigma^2}$ satisfies the change of variables formula for $0$-orbital measure
(see Proposition~\ref{lemDiffeoMeasureChange}).
Since it is obviously finite (thus regular), we can thus identify $\mu_{\sigma^2}$ with the unique regular $0$-orbital measure $\measOrb{\sigma^2}{0}$
according Theorem~\ref{thm:uniqueness}, up to normalisation.

To fix the normalisation, in Section~\ref{sec:BB}, we explain that the natural normalisation of the partition function of a Brownian Bridge $\WS{\sigma^2}{0}{1}$
with variance $\sigma^2$ is
\begin{equation}
  \WS{\sigma^2}{0}{1}(\Cfree[0,1]) = \frac{1}{\sqrt{2\pi \sigma^2}} %
  .
\end{equation}
Since the Lebesgue measure of $\T$ is $1$ and the parametrization \eqref{defP-bis} does not change the total mass,
we also normalise the $0$-orbital measure in the same way:
\begin{equation}
  \measOrb{\sigma^2}{0}(\Diff^1(\T)) = \frac{1}{\sqrt{2\pi \sigma^2}}.
\end{equation}

\subsection{Unnormalised Brownian Bridge measure}
\label{sec:BB}

The unnormalised version of the Brownian Bridge measure is defined in Definition~\ref{defn:BB} below.
It should be a finite measure on $\Cfree[0,T] = \left\{f\in C[0,T]\, | \, f(0)=0\right\}$ formally represented as
\begin{equation} \label{e:BB-formaldensity}
\d\WS{\sigma^2}{a}{T}(\xi) = \exp\left\{-\frac{1}{2\sigma^2}\int_{0}^{T}\xi'^{\, 2} (t)\d t\right\} \delta(\xi(0)) \delta(\xi(T)-a)\prod_{\tau \in (0,T)}\d\xi(\tau).
\end{equation}
For any $T_1, T_2>0$, $a\in \R$ and 
any positive continuous functional $F$ on $C[0, T_1+T_2]$, we then expect
\begin{equation}\label{eqConvWS}
  \int F(\xi)\d \WS{\sigma^2}{a}{T_1+T_2}(\xi) = 
  \int_{\R}\int \int 
  F(\xi_1 \sqcup \xi_2)
  \d \WS{\sigma^2}{b}{T_1}(\xi_1)\d \WS{\sigma^2}{a-b}{T_2}(\xi_2)\d b,
\end{equation}
where for $f\in \Cfree [0, T_1]$ and $g\in \Cfree [0, T_2]$, we denote by $f\sqcup g \in \Cfree[0,T_1+T_2]$ %
the function
\begin{equation}
(f\sqcup g)(t) =   
\begin{dcases} 
  f(t) & \text{if }  t\in[0, T_1], \\
   f(T_1)+g(t-T_1)   & \text{if } t\in(T_1, T_1+T_2].
  \end{dcases}
\end{equation}
The precise definition achieving these properties is as follows.

\begin{defn} \label{defn:BB}
  The unnormalised Brownian Bridge measure with variance $\sigma^2 > 0$ is a finite Borel measure $\d\WS{\sigma^2}{a}{T}$ on $\Cfree[0, T]$ such that
  \begin{equation}\label{eq:25}
    \sqrt{2\pi T\sigma^2}\, \exp\left\{\frac{a^2}{2 T\sigma^2}\right\}\d\WS{\sigma^2}{a}{T}(\xi)
  \end{equation}
  is the distribution of a Brownian Bridge $(\xi(t))_{t\in[0,T]}$ with variance $\sigma^2$ and $\xi(0) = 0$, $\xi(T) = a$.
\end{defn}

\color{black}

\noindent
The above normalisation of the normalised Brownian bridge is exactly the one that is needed to ensure that the composition property \eqref{eqConvWS} holds.
We note that, up to a constant, it coincides with the $\zeta$-regularised determinant of the Dirichlet Laplacian on $[0,T]$, namely $\det_{\zeta}(-\frac{1}{\sigma^{2}}\Delta^{[0,T]}) \propto T\sigma^{2}$,
see \eqref{e:det-rho} below.

\begin{prp} \label{prop:BBcomposition}
  The unnormalised Brownian Bridge measures $\d\WS{\sigma^2}{a}{T}$ satisfy the property \eqref{eqConvWS}.
\end{prp}

\begin{proof}
  Notice that $\d\WS{\sigma^2}{a}{T}(\xi) \otimes \d a$ is a probability measure and that the distribution of $\xi$ under this measure that of a Brownian motion restricted to $[0,T]$, see [RevuzYor, Exercise (3.16)].
  Therefore, the Markov property for Brownian motion implies that the distribution of $\xi_1\sqcup \xi_2$ under
  $\d\WS{\sigma^2}{b}{T_1}(\xi_1) \otimes \d b \otimes \d\WS{\sigma^2}{a-b}{T_2}(\xi_2) \otimes \d a$ is a Brownian motion restricted to $[0,T_1+T_2]$.
  On the other hand, $\xi$ under $\d\WS{\sigma^2}{a}{T_1+T_1}(\xi) \otimes \d a$ is also a Brownian motion on $[0,T_1+T_2]$.
  Since $(\xi_1\sqcup\xi_2)(T_1+T_2)=a$ under the first measure and $\xi(T_1+T_2)=a$ under the second measure, the distributions
  of $\xi_1\sqcup\xi_2$ and $\xi$ under $\d\WS{\sigma^2}{b}{T_1}(\xi_1) \otimes \d b \otimes \d\WS{\sigma^2}{a-b}{T_2}(\xi_2)$
  respectively $\d\WS{\sigma^2}{a}{T_1+T_1}(\xi)$ must be the same, i.e.\ \eqref{eqConvWS} holds.
\end{proof}

\subsection{Unnormalised Malliavin--Shavgulidze measure}
\label{sec:MSmeas}

We now define the $0$-orbital measure as a finite measure $\mu_{\sigma^2}= \measOrb{\sigma^2}{0}$ on $\Diff^1(\T)$.
This measure is very similar to what is known as the Malliavin--Shavgulidze measure,
see \cite[Section~11.5]{BogachevMalliavin}. 
It should formally correspond to
\begin{equation} \label{e:MeasureMuFormal}
\d\mu_{\sigma^2}(\phi) 
= \exp\left\{-\frac{1}{2\sigma^2}\int_{\T}\left(\frac{\phi''(\tau)}{\phi'(\tau)}\right)^2\d\tau\right\} \prod_{\tau \in \T}\frac{\d\phi(\tau)}{\phi'(\tau)}.
\end{equation}
To motivate the actual definition,
recall the formal density \eqref{e:BB-formaldensity} of the unnormalised Brownian Bridge measure $\WS{\sigma^2}{0}{1}$.
Thus under the measure $\mu_{\sigma^2}$ the process $(\log\phi'(\tau) - \log\phi'(0))_{\tau \in [0,1)}$
formally has the same density as $(\xi_{t})_{t\in [0,1)}$ under $\WS{\sigma^2}{0}{1}$,
and we define $\mu_{\sigma^2}$ by
\begin{equation}\label{defMeasureMu}
\d\mu_{\sigma^2}(\phi) \coloneqq 
\d \WS{\sigma^2}{0}{1}(\xi)
\otimes \d\Theta, \qquad \text{with } \phi(t) = \Theta + \A_{\xi}(t)\enspace (\mathrm{mod}\, 1), \text{ for } \Theta\in [0, 1),
\end{equation}
where $\d\Theta$ is the Lebesgue measure on $[0,1)$ and with the change of variables
\begin{equation} \label{defP}
\A(\xi)(t) \coloneqq \A_{\xi}(t) 
\coloneqq \frac{\int_{0}^t e^{\xi(\tau)}\d\tau}{\int_{0}^1 e^{\xi(\tau)}\d\tau},
\end{equation}
The variable $\Theta$ corresponds to the value of $\phi(0)$.
Note that the map $\xi \mapsto \A(\xi)$ is a bijection between $\Cfree[0, 1]$
and $\Diff^1 [0,1]$ with inverse map
\begin{equation}
\begin{aligned} 
\A^{-1}: \Diff^1[0,1] &\to \Cfree[0, 1]\\
\varphi &\mapsto \log \varphi'(\cdot) - \log \varphi'(0).
\end{aligned}
\end{equation}
With $\mu_{\sigma^2}$ defined as above, the following change of variables formula holds:

\begin{prp} \label{lemDiffeoMeasureChange}
Let $f \in \Diff^3[0, 1]$ be such that $f'(0) = f'(1)$ and $f''(0) = f''(1)$,
and denote by $f^*\mu_{\sigma^2} = f^{-1}_*\mu_{\sigma^2}$ the push-forward of $\mu_{\sigma^2}$ under post-composition with $f^{-1}$,
i.e.
\begin{equation}
f^*\mu_{\sigma^2}(A) = \mu_{\sigma^2}(f\circ A),
\end{equation} 
where $f\circ A \coloneqq \left\{f \circ \phi\, \big| \, \phi \in A\right\}$.
Then
\begin{equation}\label{eqLemmaDiffeoMeasureChange}
\frac{\d f^*\mu_{\sigma^2}(\phi)}{\d\mu_{\sigma^2}(\phi)} = 
\exp\left\{
\frac{1}{\sigma^2}\int_{\T} \Schw_f\big(\phi(t)\big)\,\phi'^{\, 2}(t) \d t \right\}.
\end{equation}
\end{prp}

The proposition is a consequence of the following change of variables formula for the unnormalised Brownian Bridge.
For $f\in \Diff^3[0,1]$ denote by $L_f$ the post-composition operator on $\Diff^1[0,1]$:
\begin{equation}\label{eq:39}
L_f(\phi) = f \circ \phi.
\end{equation}

\begin{prp}%
\label{prpBBMMeasureChange}
Let $f\in \Diff^3[0,1]$ and set $b = \log f'(1) - \log f'(0)$. 
Let $f^\sharp \WS{\sigma^2}{a}{1} = f_{\sharp}^{-1} \WS{\sigma^2}{a}{1}$ be the push-forward of $\WS{\sigma^2}{a}{1}$ under $\A^{-1} \circ L_{f^{-1}} \circ \A=\left(\A^{-1} \circ L_{f} \circ \A\right)^{-1}$.
Then for any $a\in \R$, $f^{\sharp} \WS{\sigma^2}{a}{1}$ is absolutely continuous with respect to $\WS{\sigma^2}{a-b}{1}$ and
\begin{equation}\label{eqBBMMeasureChange}
\frac{\d f^{\sharp} \WS{\sigma^2}{a}{1}(\xi)}{\d  \WS{\sigma^2}{a-b}{1}(\xi)} =  
\frac{1}{\sqrt{f'(0)f'(1)}}
\exp\left\{ 
\frac{1}{\sigma^2}\left[\frac{f''(0)}{f'(0)}\A_{\xi}'(0)-\frac{f''(1)}{f'(1)}\A_{\xi}'(1)\right]
+
\frac{1}{\sigma^2}\int_0^1 \Schw_f \big(\A_{\xi}(t)\big)\, \A_{\xi}'(t)^{2}\d t
\right\}.
\end{equation}
\end{prp}

We prove this statement in Section~\ref{sect_Malliavin_Calculations}.
A similar quasi-invariance statement (for the Wiener measure instead of the unnormalised Brownian Bridge)
can be found in \cite[Theorem~11.5.1]{BogachevMalliavin}, for example, but we were not able to locate a rigorous proof of formula \eqref{eqBBMMeasureChange} in the existing literature.

For now, we show how it implies Proposition~\ref{lemDiffeoMeasureChange}:

\begin{proof}[Proof of Proposition~\ref{lemDiffeoMeasureChange}]
Let 
\begin{align}
\phi(t) &=   \Theta + \A_{\xi}(t)  \enspace (\mathrm{mod}\, 1),\\
\big(f \circ \phi \big) (t) &=  \widetilde{\Theta} + \A_{\widetilde{\xi}}(t) \enspace (\mathrm{mod}\, 1) .
\end{align}
In other words,
\begin{equation}
\widetilde{\Theta} = f(\Theta),\qquad
\widetilde{\xi} = \Big(\A^{-1}\circ L_{f_{\Theta}} \circ \A \Big) (\xi), 
\end{equation}
where $f_{\Theta}(\tau) = f(\tau+\Theta) - \widetilde{\Theta}$.
From \eqref{defMeasureMu} we see that,
\begin{equation}
\d f^*\mu_{\sigma^2}(\phi) = \d f_{\Theta}^{\sharp}\WS{\sigma^2}{0}{1}(\xi) \times \d f(\Theta),
\end{equation}
with notation as in Proposition~\ref{prpBBMMeasureChange}.
Hence, Proposition~\ref{prpBBMMeasureChange} implies
\begin{equation}
\frac{\d f_{\Theta}^{\sharp}\WS{\sigma^2}{0}{1}(\xi) }{\d\WS{\sigma^2}{0}{1}(\xi) }
= \frac{1}{f'(\Theta)} \exp\left\{\frac{1}{\sigma^2}\int_{\T} \Schw_f\big(\phi(t)\big)\,\phi'^{\, 2}(t) \d t\right\},
\end{equation}
and using
\begin{equation}
\frac{\d f(\Theta)}{\d \Theta} = f'(\Theta)
\end{equation}
the proof is finished.
\end{proof}

\color{black}

\subsection{Change of variables: Proof of Proposition~\ref{prpBBMMeasureChange}}

\label{sect_Malliavin_Calculations}

Let $\Wiener_{\sigma^2}$ denote the Wiener measure with variance $\sigma^2$ on $\Cfree[0,1] = \left\{f\in C[0,1]\, | \, f(0)=0\right\}$,
and recall the definition of the unnormalised Brownian Bridge $\WS{\sigma^2}{a}{1}$ from $0$ to $a$ in time $1$ from Definition~\ref{defn:BB}.
Then by [RevuzYor, Exercise (3.16)], for any $ X\subset \Cfree[0,1]$,
\begin{equation}\label{eqWienerSlicingWS}
\Wiener_{\sigma^2}(X) 
= \int_{\R}\WS{\sigma^2}{a}{1}(X)\d a.
\end{equation}
Moreover, by Proposition~\ref{prop:disintegration} in Appendix~\ref{sec:disintegration},
  this disintegration identity characterises $\WS{\sigma^2}{a}{1}$ uniquely, assuming that
$\WS{\sigma^2}{a}{1}$ is supported on $\{f \in C[0,1]\,|\, f(0)=0, \, f(1)=a\}$ and that
$a\mapsto \WS{\sigma^2}{a}{1}$ is continuous in the sense of Definition~\ref{defMeasureCont} 
(with the open cover $\{U_a\}_{a\in A}$ being the whole space: $U_a = \Cfree[0,1]$).

\begin{lmm}\label{lmmExpMomentAprioriBound}
Let $\TechMeas$ be either $\WS{\sigma^2}{a}{T}$ for some $a\in\R$ and $T>0$, or $\Wiener_{\sigma^2}$ (in this case we put $T=1$).
Then there exists $\eps>0$ such that
\begin{equation}
\int \exp\left(\frac{\eps}{\int_0^T e^{\xi(\tau)}\d\tau} \right)\d\TechMeas(\xi)<\infty.
\end{equation}
\end{lmm}
\begin{proof}

If $\TechMeas = \WS{\sigma^2}{a}{T}$, then let $\widetilde{\xi}$ be a Brownian Bridge distributed according to the probability measure $\sqrt{2\pi T}\sigma\exp\left(\frac{a^2}{2T \sigma^2}\right) \d \WS{\sigma^2}{a}{T}(\widetilde{\xi})$, and according to $\Wiener_{\sigma^2}$ otherwise.
Then 
\begin{equation} 
\Prob\left[
\frac{1}{\int_0^{T} e^{\widetilde{\xi}(\tau)}\d \tau} > \lambda
\right] 
\leq 
\Prob\left[
\min_{t\in[0, e \lambda^{-1})]} \widetilde{\xi}(t)<-1
\right]
\leq C^{-1} e^{-C\lambda} ,
\end{equation}
for some $C>0$, independent of $\lambda>10/T$. 
\end{proof}

Recall the left-composition operator $L_{f}$ as in \eqref{eq:39}.
The following lemma is key for the calculations. It verifies formula (7) in \citep{BelokurovShavgulidzeExactSolutionSchwarz} and (2.7) in \citep{BelokurovShavgulidzeCorrelationFunctionsSchwarz}, which are key for the calculations.

\begin{lmm}\label{lemmaWienerMeasureChange}
Let $f\in \Diff^3[0,1]$, and let $f^{\sharp} \Wiener_{\sigma^2} = f^{-1}_{\sharp} \Wiener_{\sigma^2}$ be the push-forward of $\Wiener_{\sigma^2}$ under $\A^{-1} \circ L_{f^{-1}} \circ \A =\left(\A^{-1} \circ L_{f} \circ \A\right)^{-1}$. 
Then 
\begin{multline}\label{eqWienerMeasureChange}
\frac{\d f^{\sharp}\Wiener_{\sigma^2}(W)}{\d \Wiener_{\sigma^2}(W)} 
= \frac{1}{\sqrt{f'(0)f'(1)}}\\
\times\exp\left\{
\frac{1}{\sigma^2}\left[\frac{f''(0)}{f'(0)}\A_W'(0)-\frac{f''(1)}{f'(1)}\A_W'(1)\right]
+
\frac{1}{\sigma^2}\int_0^1 \Schw_f \Big(\A_W(t)\Big)\Big(\A_W'(t)\Big)^2\d t
\right\}.
\end{multline}
\end{lmm}

\begin{proof}
Let $B(t) = \sigma^{-1} W(t)$ be a standard Brownian motion.
The transformation $W \mapsto \A^{-1} \circ L_f \circ \A (W)$ corresponds to the transformation of $B(t)$ given by
\begin{equation}\label{eqLemmaWienerChangeTransformation}
B(t) \mapsto
B(t) + 
\sigma^{-1}\log f'\left(\frac{\int_0^t e^{\sigma B(\tau)}\d\tau}{\int_0^1 e^{\sigma B(\tau)}\d\tau}\right) 
- \sigma^{-1}\log f'(0).
\end{equation}
The problem reduces to the calculation of Radon-Nikodym derivative of the probability measure, corresponding to the standard Brownian motion under the transformation inverse to \eqref{eqLemmaWienerChangeTransformation}.
Denote
\begin{equation}
\varphi(t) = \frac{\int_0^t e^{\sigma B(\tau)}\d\tau}{\int_0^1 e^{\sigma B(\tau)}\d\tau},
\end{equation}
and
\begin{equation}\label{eqLemmaWienerChangeDefh}
h(x) = \log f'(x) -  \log f'(0).
\end{equation}
Using \citep[Theorem 4.1.2]{Nualart_Malliavin}, the Radon-Nikodym derivative is given by
\begin{equation}\label{eqLemmaWienerChangeFormulaRadonNik}
\det\nolimits_2\left(1+K\right)\exp\left(-\delta(u)-\frac{1}{2}\int_0^1 u^2(t)\d t \right),
\end{equation} 
where $\delta$ denotes Skorokhod integral, $\det\nolimits_2$ is a Hilbert-Carleman (or Carleman-Fredholm) determinant (see, e.g. \citep[Chapter X]{bookDeterminants}),
$u(t)=u[\varphi](t)$ is given by
\begin{equation}
u(t) =\frac{\d}{\d t} \sigma^{-1}h\big(\varphi(t)\big) = \sigma^{-1} h'\big(\varphi(t)\big)\varphi'(t),
\end{equation}
and $K = K[\varphi]$ is the Fr\'{e}chet derivative of the map
\begin{equation}
(B(t))_{t\in [0,1]}
\mapsto
(\sigma^{-1}h(\varphi(t)))_{t\in [0,1]}
=
\left( \sigma^{-1}\log f'\left(\frac{\int_0^t e^{\sigma B(\tau)}\d\tau}{\int_0^1 e^{\sigma B(\tau)}\d\tau}\right) 
- \sigma^{-1}\log f'(0) \right)_{t\in [0,1]}
\end{equation}
with respect to the Cameron-Martin space of Brownian motion, which we associate with the Sobolev space
$\tilde{H}^1 = \left\{g:\, \|g\|_{\tilde{H}^1} = \int_0^1 \big(g'(t)\big)^2\d t<\infty, \, g(0)=0 \right\}$.
Direct calculation shows that
\begin{equation}\label{eqLemmaWienerChangeDifferentialFormula}
(K g)(t) =  \int_0^1 k(t, s)g(s) \d s, \qquad \text{with }
k(t,s) = -h'\left(\varphi(t)\right) \varphi(t) \varphi'(s) + \one_{s<t}\, h'\left(\varphi(t)\right) \varphi'(s).
\end{equation}

\noindent
\textbf{Integral.}
First, we calculate Skorokhod integral
\begin{equation}
\delta (u)=
\sigma^{-1}\delta\Big( h'(\varphi(t))\varphi'(t)\Big)
=
\sigma^{-1} \delta\left(h'
\left(\frac{\int_0^t e^{\sigma B(\tau)}\d\tau}{\int_0^1 e^{\sigma B(\tau)}\d\tau}\right)
 \frac{e^{\sigma B(t)}}{\int_0^1 e^{\sigma B(\tau)}\d\tau}\right).
\end{equation}
The process is not adapted because of the term $\int_0^1 e^{\sigma B(\tau)}\d\tau$. 
We use \citep[Theorem 3.2.9]{Nualart_Malliavin} in order to reduce the Skorokhod integral to a It\^{o} integral. 
It follows from Lemma~\ref{lmmExpMomentAprioriBound} that the random variable
\begin{equation}
\frac{1}{\int_0^1 e^{\sigma B(\tau)}\d\tau}
\end{equation}
is Malliavin smooth and that
\begin{equation}
D_t \left(\frac{1}{\int_0^1 e^{\sigma B(\tau)}d\tau}\right) 
= \frac{\sigma \int_t^1 e^{\sigma B(\tau)}\d\tau}{\left(\int_0^1 e^{\sigma B(\tau)}\d\tau \right)^2}.
\end{equation}
Thus, since $u(t)$ is also Malliavin smooth, using \citep[Theorem 3.2.9]{Nualart_Malliavin} we get
\begin{multline}\label{eqLemmaWienerChangeSkorokhod}
\delta\left(h'
\left(\frac{\int_0^t e^{\sigma B(\tau)}\d\tau}{\int_0^1 e^{\sigma B(\tau)}\d\tau}\right) 
\frac{e^{\sigma B(t)}}{\int_0^1 e^{\sigma B(\tau)}\d\tau}\right) 
= 
\left.\int_0^1 h'\left(F\int_0^t e^{\sigma B(\tau)}\d\tau\right) F\, e^{\sigma B(t)}\,\d B(t)\right|_{F= \left(\int_0^1 e^{\sigma B(\tau)}\d\tau\right)^{-1}}\\
+ \int_0^1 \Big[ 
h''\big(\varphi(t)\big) \varphi(t)\varphi'(t)
+  h'\big(\varphi(t)\big)\varphi'(t)
\Big]
\cdot
\left(
\frac{\sigma\int_t^1 e^{\sigma B(\tau)}d\tau }{\int_0^1 e^{\sigma B(\tau)}d\tau}
\right)
\d t.
\end{multline}
We observe that 
the expression in square brackets is equal to
\begin{equation}
\frac{\d}{\d t}\Big[h'(\varphi(t))\varphi(t)\Big],
\end{equation}
and calculate the second term in the right-hand side of \eqref{eqLemmaWienerChangeSkorokhod} by integrating by parts
\begin{multline}
\int_0^1 \Big[ 
h''\big(\varphi(t)\big) \varphi(t)\varphi'(t)
+  h'\big(\varphi(t)\big)\varphi'(t)
\Big]
\cdot
\left(
\frac{\sigma\int_t^1 e^{\sigma B(\tau)}\d\tau }{\int_0^1 e^{\sigma B(\tau)}\d\tau}
\right)\d t
\\
=\sigma \int_0^1  h'(\varphi(t))\varphi(t)\frac{e^{\sigma B(t)}}{\int_0^1 e^{\sigma B(\tau)}\d\tau}\d t.
\end{multline}
Integrating by parts once again we obtain 
\begin{align}
\sigma \int_0^1  h'(\varphi(t))\varphi(t)\frac{e^{\sigma B(t)}}{\int_0^1 e^{\sigma B(\tau)}\d\tau}\d t
&= 
\sigma \int_0^1  h'(\varphi(t))\varphi'(t)\varphi(t)\d t\nonumber\\
&=
\sigma \, h(1) - \sigma \int_0^1  h(\varphi(t))\varphi'(t)\d t\nonumber\\
&=
\sigma \, h(1) - \sigma \int_0^1  h(s)\d s.
\end{align}

We calculate the first term in the right-hand side of \eqref{eqLemmaWienerChangeSkorokhod} by replacing $e^{\sigma B(t)}\d B(t) = \sigma^{-1} \d e^{\sigma B(t)} - \frac{\sigma }{2}e^{\sigma B(t)}\d t$, and using It\^{o} integration by parts
\begin{multline}
\left.\int_0^1 h'\left(F\int_0^t e^{\sigma B(\tau)}\d\tau\right) F\, e^{\sigma B(t)}\,\d B(t)\right|_{F= \left(\int_0^1 e^{\sigma B(\tau)}\d\tau\right)^{-1}}\\
= \frac{h'(1) e^{\sigma B(1)}}{\sigma \int_0^1 e^{\sigma B(\tau)}\d\tau} - \frac{h'(0)}{\sigma \int_0^1 e^{\sigma B(\tau)}\d\tau}
- \frac{\sigma}{2}\int_0^1 h'\big(\varphi(t)\big)\varphi'(t)\d t 
-\sigma^{-1} \int_0^1 h''\left(\varphi(t)\right)\varphi'^{\, 2}(t) \d t\\
=
\sigma^{-1} h'(1)\varphi'(1) - \sigma^{-1} h'(0)\varphi'(0)
-\frac{\sigma}{2}h(1) 
- \sigma^{-1}\int_0^1 h''\left(\varphi(t)\right)\varphi'^{\, 2}(t) \d t.
\end{multline}
Therefore,
\begin{equation} \label{eqLemmaWienerChangeSkorokhodResult}
\delta(u) 
= \sigma^{-2} \big(h'(1)\varphi'(1) - h'(0)\varphi'(0)\big) +\frac{1}{2}h(1) - \int_0^1  h(s)\d s - \sigma^{-2}\int_0^1 h''\left(\varphi(t)\right)\varphi'^{\, 2} (t) \d t.
\end{equation}

\noindent
\textbf{Determinant.}
To calculate the Hilbert-Carleman determinant $\det\nolimits_2 \left(1+ K\right)$, with $K$ given by \eqref{eqLemmaWienerChangeDifferentialFormula},
we use \citep[Chapter XIII, Corollary 1.2]{bookDeterminants}.
In the notations from \citep[Chapter XIII]{bookDeterminants} we have
\begin{equation}
k(t,s) = 
\begin{cases}
F_1(t)G_1(s),& 0\leq s<t \leq 1; \\
-F_2(t)G_2(s),& 0\leq t<s \leq 1,
\end{cases}
\end{equation}
with
\begin{equation}
F_1(t) = h'(\varphi(t))\left(1-\varphi(t)\right), \qquad F_2(t) =  h'(\varphi(t))\varphi(t), \qquad G_1(s) = G_2(s) = \varphi'(s).
\end{equation}
According to \citep[Chapter XIII, Corollary 1.2]{bookDeterminants}, 
\begin{equation}\label{eqLemmaWienerChangeDetMain}
\det\nolimits_2 (1+ K) = \det\left(N_1+N_2 U(1)\right)\cdot \exp \left\{\int_0^1 F_2(s)G_2(s)\d s\right\},
\end{equation}
where
\begin{equation}
N_1 = 
\begin{pmatrix}
1 & 0\\
0 & 0
\end{pmatrix}, \quad
N_2 = 
\begin{pmatrix}
0 & 0\\
0 & 1
\end{pmatrix},
\end{equation}
and $U(t)$ is a $2\times 2$ matrix which satisfies the differential equation 
\begin{equation} \label{eqLemmaWienerChangeDetMonodromy}
U(0) = 
\begin{pmatrix}
1 & 0\\
0 & 1
\end{pmatrix}, \quad
U'(t) = 
- \begin{pmatrix}
G_1(t)F_1(t) & G_1(t)F_2(t)\\
G_2(t)F_1(t) & G_2(t)F_2(t)
\end{pmatrix}U(t).
\end{equation}
We start by solving the differential equation \eqref{eqLemmaWienerChangeDetMonodromy}. Note that
\begin{equation} %
 \begin{pmatrix}
G_1(t)F_1(t) & G_1(t)F_2(t)\\
G_2(t)F_1(t) & G_2(t)F_2(t)
\end{pmatrix} 
= h'\big(\varphi(t)\big)\varphi'(t) \begin{pmatrix}
1-\varphi(t) & \varphi(t)\\
1-\varphi(t) & \varphi(t)
\end{pmatrix} 
.
\end{equation}
Therefore, by rewriting \eqref{eqLemmaWienerChangeDetMonodromy}, we see that $U(t) = V\big(\varphi(t)\big)$ with
\begin{equation}
V'(s) = - h'(s) 
\begin{pmatrix}
1-s & s\\
1-s & s
\end{pmatrix} 
V(s) = -h'(s)
\begin{pmatrix}
1\\
1
\end{pmatrix} 
\left(
\frac{1}{2}
\begin{pmatrix}
1 & 1
\end{pmatrix}
+
\frac{2s-1}{2}
\begin{pmatrix}
-1 & 1
\end{pmatrix}
\right)V(s).
\end{equation}
Observe that if 
\begin{equation}
v(t) 
= 
v_1(t)\begin{pmatrix}
1 \\
1
\end{pmatrix} 
+ 
v_2(t)
\begin{pmatrix}
-1 \\
1
\end{pmatrix}
\end{equation}
is a representation of $v(t)=V(t)v(0)$ in terms of the orthogonal basis, then
\begin{align}
v_2(t) &= v_2(0),\\
v_1(t) &=v_1(0) e^{-h(t)}-\int_{0}^t e^{h(s)-h(t)}h'(s)(2s-1)v_2(s)\d s\\
& = v_1(0)e^{-h(t)}-v_2(0)\left((2t-1) + e^{-h(t)}-2e^{-h(t)}\int_0^t e^{h(s)}\d s\right).
\end{align}
In particular, 
\begin{equation}
v_2(1) = v_2(0),
\qquad 
v_1(1)  = v_1(0)e^{-h(1)}+ v_2(0)\left(2e^{-h(1)}\int_0^1 e^{-h(s)}\d s -e^{-h(1)}-1 \right).
\end{equation}
Therefore,
\begin{align}
  V(1) \begin{pmatrix} 1 \\ 1 \end{pmatrix}
  &=
  e^{-h(1)}\begin{pmatrix} 1 \\ 1 \end{pmatrix}
  \\
  V(1) \begin{pmatrix} -1 \\ 1 \end{pmatrix}
  &=
  \begin{pmatrix} -1 \\ 1 \end{pmatrix} +
  \left(2e^{-h(1)}\int_0^1 e^{-h(s)}\d s -e^{-h(1)}-1 \right)
  \begin{pmatrix} 1 \\ 1 \end{pmatrix}
\end{align}
which is equivalent to
\begin{equation}
U(1) = V(1) 
=
\begin{pmatrix}
1+e^{-h(1)}- e^{-h(1)}\int_0^1 e^{h(s)}\d s &
-1 + e^{-h(1)}\int_0^1 e^{h(s)}\d s\\
e^{-h(1)}- e^{-h(1)}\int_0^1 e^{h(s)}\d s&
e^{-h(1)}\int_0^1 e^{h(s)}\d s
\end{pmatrix}.
\end{equation}
Thus,
\begin{equation}\label{eqLemmaWienerChangeDetSmallDet}
 \det\Big(N_1+N_2 U(1)\Big) = e^{-h(1)}\int_0^1 e^{h(s)}\d s.
\end{equation}
Also, 
\begin{equation} \label{eqLemmaWienerChangeDetTrace}
\int_0^1 F_2(s)G_2(s)\d s = \int_0^1 h'\big(\varphi(s)\big)\varphi(s)\varphi'(s)ds 
= \int_0^1 h'(t)t\,\d t = h(1) - \int_0^1 h(t)\d t. 
\end{equation}
Combining \eqref{eqLemmaWienerChangeDetMain}, 
\eqref{eqLemmaWienerChangeDetSmallDet}, and
\eqref{eqLemmaWienerChangeDetTrace} we 
obtain
\begin{equation} \label{eqLemmaWienerChangeDetResult}
  \det\nolimits_2 \left(1+ K\right) = \exp\left(-\int_{0}^1h(t)\d t\right)\cdot\int_0^1 e^{h(s)}\d s
  .
\end{equation}

\noindent
\textbf{Final result.}
Bringing together \eqref{eqLemmaWienerChangeDetResult}, \eqref{eqLemmaWienerChangeSkorokhodResult}, \eqref{eqLemmaWienerChangeDefh},
and the formula for the Radon-Nikodym derivative \eqref{eqLemmaWienerChangeFormulaRadonNik} we get the desired result.
\end{proof}

Proposition~\ref{prpBBMMeasureChange} is a straightforward consequence of the previous lemma.

\begin{proof}[Proof of Proposition~\ref{prpBBMMeasureChange}]
From Lemma~\ref{lemmaWienerMeasureChange} and \eqref{eqWienerSlicingWS}  we have that 
\begin{equation}\label{eqProofPropBBMMeasureChange0}
\d f^{\sharp}\Wiener_{\sigma^2}(W)
=
\int_{\R} 
\d \widetilde{\mathcal{B}}_{\sigma^2}^{\, a, 1}(W)
\d a
\end{equation}
where
\begin{multline}\label{eqProofPropBBMMeasureChange1}
\d \widetilde{\mathcal{B}}_{\sigma^2}^{\, a, 1}(W)
=
\frac{1}{\sqrt{f'(0)f'(1)}}\\
\times\exp\left\{
\frac{1}{\sigma^2}\left[\frac{f''(0)}{f'(0)}\A_W'(0)-\frac{f''(1)}{f'(1)}\A_W'(1)\right]
+
\frac{1}{\sigma^2}\int_0^1 \Schw_f \Big(\A_W(t)\Big)\Big(\A_W'(t)\Big)^2\d t
\right\}
\d\WS{\sigma^2}{a}{1}(W).
\end{multline}
This is a disintegration of $\d f^{\sharp}\Wiener_{\sigma^2}(W)$ into a family of measures, parametrised by $a\in \R$.
Since measures $\WS{\sigma^2}{a}{1}$ are continuous in $a\in\R$ in the sense of Definition~\ref{defMeasureCont} from Appendix~\ref{sec:disintegration} and the density in the right-hand side of \eqref{eqProofPropBBMMeasureChange1} is continuous in $W$, 
by Proposition~\ref{prop_measure_density_contin} we obtain that measures $\widetilde{\mathcal{B}}_{\sigma^2}^{\, a, 1}$ are continuous in $a$ as well.

On the other hand, from \eqref{eqWienerSlicingWS}
\begin{equation}
\d f^{\sharp} \Wiener_{\sigma^2}(W) 
= \int_{\R} \d f^{\sharp} \WS{\sigma^2}{a}{1}(W)\d a
\end{equation}
is also a disintegration of $f^{\sharp} \Wiener_{\sigma^2}$ into a continuous family of measures.

Observe that for any $a\in \R$ the measure $f^{\sharp} \WS{\sigma^2}{a}{1}$ is supported on $\{h\in \Cfree[0, 1]: h(1) = a-b\}$, 
and $\widetilde{\mathcal{B}}_{\sigma^2}^{\, a, 1}$ is supported on  $\{h\in \Cfree[0, 1]: h(1) = a\}$.
Therefore, by Proposition~\ref{prop:disintegration} this disintegrations coincide and $f^{\sharp} \WS{\sigma^2}{a}{1} = \widetilde{\mathcal{B}}_{\sigma^2}^{\, a-b, 1}$, which finishes the proof.
\end{proof}

\section{Schwarzian measure with varying metric}
\label{sec:metric}

In this final section, we explain how Theorem~\ref{thrMeasureSLInv} can be adapted to a varying metric (Theorem~\ref{thr:2}).

\subsection{Definition via reparametrisation}

The measure $\FMeasSL{\rho}$ can be defined from $\FMeas{\sigma^2}$ by reparametrisation and adjustment of the normalisation.
Define the reparametrisation $h$ of $[0,1]$ by
\begin{equation} \label{e:h-rho}
  t \mapsto h(t) = \int_0^t \frac{\rho(\tau)}{\sigma^2} \, \d\tau, \qquad \text{where }\sigma^2  = \int_0^1 \rho(\tau)\, \d\tau,
\end{equation}
and we recall that $\rho$ is the positive square root of $\rho^2$. Then if $\FMeasSL{\rho}$ is defined by
\begin{equation} 
  \int F(\phi) \, \d\FMeasSL{\rho}(\phi)
  = 
  C(\rho)
  \int F(\phi \circ h) \, \d\FMeasSL{\sigma^2}(\phi),
\end{equation}
the change of variables formula  \eqref{eqDiffeoMeasureChange-metric} follows from \eqref{eqDiffeoMeasureChange}.
Indeed, writing $\d\FMeasSL{\rho}(\phi) = C(\rho)\d\FMeasSL{\sigma^2}(\phi\circ y)$ where $y$ is the inverse map to $h$,
and $y'(\tau)^2 \, \d\tau = y'(\tau) \, \d y(\tau) = \d y(\tau)/h'(y(\tau)) =\sigma^2  \d y(\tau)/\rho(y(\tau))$,
so that changing variables from $\tau$ to $y(\tau)$ in the last equality below, we see that \eqref{eqDiffeoMeasureChange} gives
\begin{equation}
  \begin{aligned}
    \frac{\d \psi^{\ast}\!\FMeasSL{\rho}(\phi)}{\d\FMeasSL{\rho}(\phi)}
    &= \frac{\d \FMeasSL{\rho}(\psi\circ \phi)}{\d\FMeasSL{\rho}(\phi)}\\
    &= 
  \exp\left\{ \frac{1}{\sigma^2}
      \int_{\T} \Big[\Schw(\tan(\pi\psi-\tfrac{\pi}{2}),\phi\circ y(\tau))-2\pi^2\Big]\, \phi'(y(\tau))^2y'(\tau)^2 \d \tau \right\}
    \\
    &= \exp\left\{
      \int_{\T} \Big[\Schw(\tan(\pi\psi-\tfrac{\pi}{2}),\phi(\tau))-2\pi^2\Big]\, \phi'(\tau)^2 \frac{\d \tau}{\rho(\tau)} \right\}.
  \end{aligned}
\end{equation}
To fix the normalisation $C(\rho)$, recall the chain rule for the Schwarzian derivative: %
\begin{equation}
  \Schw(f \circ h, \tau) = \Schw(f,h(\tau)) (h'(\tau))^2 + \Schw(h,\tau),
\end{equation}
which implies that
\begin{equation}
  \int_{\T} \Schw(\phi\circ h,\tau) \frac{\d\tau}{\rho(\tau)}  
  =
  \int_{\T} \Schw(h,\tau)\, \frac{\d\tau}{\rho(\tau)}
  +
  \frac{1}{\sigma^2}\int_{\T} \Schw(\phi, \tau) \, \d\tau
  .
\end{equation}
Identifying the left-hand side with the desired action of $\FMeasSL{\rho}$
and the second term on the right-hand side with the action of $\FMeasSL{\sigma^2}$ we fix the constant by defining
\begin{equation} \label{eq:Schw-h-rho}
  C(\rho) = \exp\left\{ \int_{\T}\Schw(h,\tau) \frac{\d\tau}{\rho(\tau)}\right\}
  .
\end{equation}
The term in the exponential can alternatively be written as
\begin{equation} \label{e:Schwh-3}
  \int_{\T} \Schw(h,\tau) \frac{\d\tau}{\rho(\tau)}  =  \frac{1}{2} \int_{\T} \frac{\rho'(\tau)^2}{\rho(\tau)^3} \, \d\tau,
\end{equation}
as can be seen by integrating by parts:
\begin{equation} \label{e:Schwh-3-pf}
  \begin{aligned}
  \int_{\T} \Schw(h,\tau) \frac{\d\tau}{h'(\tau)}
  &= \int_{\T} \left(\left(\frac{h''(\tau)}{h'(\tau)}\right)'-\frac12 \left(\frac{h''(\tau)}{h'(\tau)}\right)^2\right) \frac{\d\tau}{h'(\tau)}
    \\
  &= \int_{\T} \left(-\left(\frac{h''(\tau)}{h'(\tau)}\right)\left(\frac{1}{h'(\tau)}\right)'-\frac12 \left(\frac{h''(\tau)^2}{h'(\tau)^3}\right)\right) \d\tau
    \\
  &= \frac12 \int_{\T} \frac{h''(\tau)^2}{h'(\tau)^3} \d\tau.
  \end{aligned}
\end{equation}
This normalisation agrees with the $\zeta$-regularisation of the partition function of the Brownian Bridge with metric $\rho^2$
and then defining the measure using the same procedure as in Section~\ref{sec:construction}. The details are outlined in the next subsection.

\subsection{Normalisation of the measure}

We outline how following the same procedure as in Section~\ref{sec:construction} with $\zeta$-regularisation for the partition
function of the Brownian Bridege leads to the normalisation of the  measure $\FMeasSL{\rho}$ defined by \eqref{eq:Schw-h-rho}.

\paragraph{Step 1. Unnormalised Brownian Bridge.}

In terms of the reparametrisation \eqref{e:h-rho}, define the unnormalised Brownian Bridge measure $\WS{\rho}{0}{1}$ by
\begin{equation} \label{e:BB-repar}
  \int F(\xi)
  \, \d\WS{\rho}{0}{1}(\xi)
  =
  \int F(\xi \circ h)
  \, \d\WS{\sigma^2}{0}{1}(\xi)
  .
\end{equation}
It is standard that if $\xi$ is distributed according to the normalised version of the measure $\WS{\sigma^2}{0}{1}$
then the time-changed process $\xi \circ h$ has quadratic variation $(\int_0^t \rho(\tau)\, \d\tau)_{t}$
and that this corresponds to the following heuristic change of variables in the action of the Brownian Bridge:
\begin{equation}
  \begin{aligned}
  \frac{1}{\sigma^2} \int_0^T \xi'(\tau)^2\, \d\tau
  &=\frac{1}{\sigma^2} \int_0^T \xi'(h(\tau))^2 h'(\tau)\, d\tau
  \\
  &=\frac{1}{\sigma^2} \int_0^T (\xi\circ h)'(\tau)^2 \frac{d\tau}{h'(\tau)}
    = \int_0^T (\xi\circ h)'(\tau)^2 \frac{d\tau}{\rho(\tau)}.
  \end{aligned}
\end{equation}
The normalisation of the Brownian Bridge measure implied by \eqref{e:BB-repar} %
also coincides with square root of the $\zeta$-function normalised determinant (up to an overall constant).
Indeed, if
$\Delta_\rho = \rho^{-1} \frac{\partial}{\partial \tau} (\rho^{-1} \frac{\partial}{\partial \tau})$
is the Laplace--Beltrami operator on $[0,1]$ with metric $\rho^2$ and Dirichlet boundary condition, then
\begin{equation} \label{e:det-rho}
  {\det}_{\zeta}\left(-\Delta_\rho\right) = C \int_0^1 \rho(t)  \,\d t = C\sigma^2,
\end{equation}
where ${\det}_{\zeta}$ is the $\zeta$-regularised determinant and $C$ is a constant independent of $\rho$.
The determinant is defined by (see for example \citep{OsgoodPhillipsSarnak})
\begin{equation}
  {\det}_{\zeta}\left(-\Delta_\rho\right) = e^{-\zeta'(0)},
\end{equation}
where $\zeta$ is the spectral $\zeta$-function, i.e.\ the analytic continuation
of $\zeta(s) = \sum_n \lambda_n^{-s}$ where $\lambda_n$ are the eigenvalues of $-\Delta_\rho$.
To see the equality \eqref{e:det-rho} one can adapt the argument leading to \citep[Equation~(1.13)]{OsgoodPhillipsSarnak} to $d=1$.

\paragraph{Step 2. Unnormalised Malliavin--Shavgulidze measure.}

Define the measure $\mu_\rho$ %
analogously to \eqref{defMeasureMu} except with the following additional prefactor motivated by the Schwarzian action:
\begin{equation}\label{defMeasureMu-metric}
\d\mu_{\rho}(\phi) \coloneqq 
\exp\left\{\int_{\T} \Big(\frac{\phi''(\tau)}{\phi'(\tau)}\Big)' \, \frac{\d\tau}{\rho(\tau)}\right\}
\d \WS{\rho}{0}{1}(\xi)
\otimes \d\Theta,
\end{equation}
again with
$\phi(t) = \Theta + \A_{\xi}(t)\enspace (\mathrm{mod}\, 1)$ and $\A_\xi(t)$
given by \eqref{defP},
and where the term in the exponential is interpreted as the It\^o integral
\begin{equation}
  \int_{\T} \Big(\frac{\phi''(\tau)}{\phi'(\tau)}\Big)' \, \frac{\d\tau}{\rho(\tau)}
  = \int_0^1 \xi''(\tau) \, \frac{\d\tau}{\rho(\tau)}
  =
  \int_0^1 \frac{\rho'(\tau)}{\rho(\tau)^2} \, \d\xi(\tau)
  .
\end{equation}
Thus $\mu_\rho$ has formal density
\begin{equation}  \label{e:MeasureMuFormal-rho}
  \d\mu_{\rho}(\phi)
  =
  \exp\left\{\int_{\T} \Schw(\phi,\tau)
    \frac{\d\tau}{\rho(\tau)}\right\} \prod_{\tau \in \T}\frac{\d\phi(\tau)}{\phi'(\tau)}.  
\end{equation}
It can be checked from Girsanov's theorem that
  \begin{equation} \label{e:mu-rho-sigma}
    \int F(\phi)
    \, \d\mu_\rho(\phi)
    =
    \exp\left\{\frac{1}{2} \int_{\T} \frac{\rho'(\tau)^2}{\rho(\tau)^3} \, \d\tau\right\}
    \int F(\phi \circ h)
    \, \d\mu_{\sigma^2}(\phi).
  \end{equation}

\paragraph{Step 3. Schwarzian measure.}
In view of \eqref{eq:1-h} and \eqref{defMeasureMu-metric},
the unquotiented Schwarzian measure should be defined by
\begin{equation} \label{defMeasureMeas}
\d\FMeasSL{\rho}(\phi) = \exp\left\{ 2\pi^2\int_{\T} \phi'^{\, 2}(\tau)\frac{\d\tau}{\rho(\tau)}\right\}\d \mu_{\rho}(\phi).
\end{equation}
It can be checked that it agrees with the previous definition including the normalisation.
The constant $C(\rho)$ comes from \eqref{e:mu-rho-sigma}.

\begin{rmrk}
The quadratic variation of $(\log \phi'(\tau))_{\tau \in [0,1)}$ does not depend on the representative of $\phi \in \Diff^1(\T)/\SL(2,\R)$
because the action by $\SL(2,\R)$ is by post-composition with a smooth function.
It follows from the construction of the measure $\FMeas{\rho}$ that, almost surely under the normalised version of  $\FMeas{\rho}$,
this quadratic variation is given by $\tau \mapsto \int_0^\tau \rho(t)\d{t}$.
\end{rmrk}

\subsection{Calculation of formal correlation functions}
\label{sec:appendix-formal-corr}

The \emph{truncated} correlation functions are formally defined as the functional derivatives
\begin{equation}
  \label{eq:31}
  \langle \Schw(\tau_{1});\cdots;\Schw(\tau_{k})\rangle_{\sigma^{2}}
  \coloneqq
  \frac{\delta}{\delta(1/\rho)(\tau_{1})}\cdots \frac{\delta}{\delta(1/\rho)(\tau_{k})}\Big\vert_{\rho =\sigma^{2}} \log \PartFSL(\rho).
\end{equation}
In the following we make sense of this expression in a distributional sense by considering the $k$-th variation of $\log \PartFSL(\rho)$:
For $h_{1},\ldots,h_{k} \in C^{\infty}(\T)$ define $\rho_{\epsilon_{1}, \ldots, \epsilon_{k}}$ via $1/\rho_{\epsilon_{1}, \ldots, \epsilon_{k}} = 1/\rho + \sum_{i=1}^{k} \epsilon_{i}h_{i}$.
Then
\begin{equation}
  \label{eq:36}
  \begin{aligned}
    [D^{k}_{1/\rho}\log \PartFSL(\rho)](h_{1},\ldots,h_{k})
    &\coloneqq
    \frac{\partial^{k}}{\partial \epsilon_1 \cdots \partial \epsilon_k}\big\vert_{\epsilon = 0} \log(\PartFSL(\rho_{\epsilon_{1},\ldots,\epsilon_{k}}))\\
    &= \int \prod_{i=1}^{k}\d\tau_{i}\, h(\tau_{1})\cdots h(\tau_{k})\, \langle\Schw(\tau_{1});\cdots;\Schw(\tau_{k})\rangle_{\rho},
  \end{aligned}
\end{equation}
where the last line is understood as defining \eqref{eq:31} as a $k$-variate distribution.

\begin{prp}\label{prop:corr-fcts-functional}
  For $\sigma^{2} > 0$, $k\geq 1$ and $h_{1},\cdots,h_{k} \in C^{\infty}(\T)$ we have
  \begin{equation}
    \label{eq:40}
    \begin{aligned}
      &\hspace*{-2em}
        [D^{k}_{1/\rho}\log \PartFSL(\rho)]\Big\vert_{\rho=\sigma^{2}}(h_{1},\ldots,h_{k})\\
      &=
        (-1)^{k} k!\, \sigma^{2(k-1)} \sum_{1\leq i < j \leq k} \int \d\tau h'_{i}h'_{j} \;{\textstyle\prod_{l \neq i,j}} h_{l}\\
      &\hspace*{2em}+
        (-1)^{k}\sigma^{2k} \sum_{\pi \in \mathrm{Part}[k]} \sigma^{2 |\pi|} [\log \PartFSL]^{(|\pi|)}(\sigma^{2}) \prod_{B\in \pi} \left( |B|! \int \d\tau \prod_{b\in B}h_{b}(\tau) \right),
    \end{aligned}
  \end{equation}
  where the sum in the last line is over all partitions $\pi = \{B_{1}, \ldots, B_{|\pi|}\}$ of $\{1,\ldots,k\}$.
\end{prp}

\begin{crl}
  For non-coinciding $\tau_{1}, \cdots, \tau_{k} \in \T$ we have
  \begin{equation}
    \label{eq:41}
    \langle \Schw(\tau_{1});\cdots;\Schw(\tau_{k})\rangle_{\sigma^{2}}
    =
    (-1)^{k}\sigma^{4k} [\log \PartFSL]^{(k)}(\sigma^{2})
    = 2\pi^{2} k!\, \sigma^{2(k-1)} + \tfrac32 (k-1)!\, \sigma^{2k}
    .
  \end{equation}
\end{crl}

Thus the Schwarzian correlators are constant away from coinciding points and their values (up to factors of $\sigma^{2}$)
are given by the cumulants of the Boltzmann-weighted spectral density $e^{-\sigma^{2} E} \nu(E)\d{E}$, see Remark~\ref{rk:spectraldensity}.
The untruncated Schwarzian correlators for non-coinciding points are therefore equal to the moments of the spectral measure $\rho(E)$ (again up to factors of $\sigma^2$).
We note that this relationship has been predicted by Stanford and Witten \cite[Appendix~C]{StanfordWittenFermionicLocalization}.

More generally, by \eqref{eq:36} and \eqref{eq:40} we can express the truncated Schwarzian correlators completely in terms of multivariate (derivatives of) $\delta$-functions in the variables $\tau_{i}$.
While the general expression is somewhat messy, we can easily derive \eqref{eq:22} and \eqref{eq:23}:

\begin{equation}
  \label{eq:42}
  \langle \Schw(\tau) \rangle_{\sigma^{2}}
  = - \sigma^{4} [\log \PartFSL]'(\sigma^{2})
  = 2\pi + \tfrac32\sigma^{2}
\end{equation}
and
\begin{equation}
  \label{eq:43}
  \begin{aligned}
    \langle \Schw(0)\Schw(\tau) \rangle_{\sigma^{2}}
    &=\langle \Schw(0);\Schw(\tau) \rangle_{\sigma^{2}}
      + \langle \Schw(0) \rangle_{\sigma^{2}} \langle \Schw(\tau) \rangle_{\sigma^{2}}\\
    &= [2\pi^{2} \sigma^{2} + \tfrac32 \sigma^{4} + (2\pi + \tfrac32\sigma^{2})^{2}]
    - 2\sigma^{2} [(2\pi + \tfrac32\sigma^{2})] \delta(\tau) - 2 \sigma^{2} \delta''(\tau).
  \end{aligned}
\end{equation}

\begin{proof}[Proof of Proposition~\ref{prop:corr-fcts-functional}]
  We treat the first and second summand in
  \begin{equation}
    \label{eq:47}
    \log \PartFSL(\rho) = \frac12 \int \frac{\rho'^{\, 2}}{\rho^{3}} + \log \PartFSL(\sigma^{2}_{\rho})
  \end{equation}
  separately.
  For the second term, we simply apply the multivariate version Fa\`{a} di Bruno's formula:
  \begin{equation}
    \label{eq:45}
    \begin{aligned}
      [D^{k}_{1/\rho}\log \PartFSL(\sigma^{2}_{\rho})]\Big\vert_{\rho=\sigma^{2}}(h_{1},\ldots,h_{k})
      =
      \sum_{\pi \in \mathrm{Part}[k]} [\log \PartFSL]^{(|\pi|)}(\sigma^{2}) \prod_{B\in \pi}\left( [D^{|B|}_{1/\rho} \sigma^{2}_{\rho}]\Big\vert_{\rho=\sigma^{2}}(\{h_{b}\}_{b\in C})\right).
    \end{aligned}
  \end{equation}
  Recall that $\sigma^{2}_{\rho} = \int\rho$ and check that
  \begin{equation}
    \label{eq:46}
    [D^{|B|}_{1/\rho} \sigma^{2}_{\rho}]\Big\vert_{\rho=\sigma^{2}}(\{h_{b}\}_{b\in C})
    = (-1)^{|B|} |B|!\, \sigma^{2(|B| + 1)} \int \prod_{b\in B} h_{b}.
  \end{equation}
  This yields the second summand on the right-hand side of \eqref{eq:40}.
  For the first summand write
  \begin{equation}
    \label{eq:48}
    F(\tfrac1{\rho})
    \coloneqq
    \frac12 \int \frac{\rho'^{\, 2}}{\rho^{3}}
    =
    \frac12 \int \left(\log\frac1{\rho}\right)' \left(\frac1{\rho}\right)' \d\tau.
  \end{equation}
  For smooth functions $f,h \in C^{\infty}(\T)$ and $f > 0$ we have for $\epsilon > 0$ sufficiently small
  \begin{equation}
    \label{eq:49}
    \begin{aligned}
      \Big(\log(f+\epsilon h)\Big)' (f+\epsilon h)'
      &= (f' + \epsilon h') \left(\log{f} + \sum_{k\geq 1} \frac{(-1)^{k-1}}{k} \epsilon^{k} \left(\frac{h}{f}\right)^{k}\right)'\\
      &= (f' + \epsilon h') \left(\frac{f'}{f} + \sum_{k\geq 1} (-1)^{k-1} \epsilon^{k} \left(\frac{h}{f}\right)^{k-1} \left(\frac{h}{f}\right)' \right),
    \end{aligned}
  \end{equation}
  with the series converging uniformly on $\T$.
  Hence, we have for sufficiently small $\epsilon > 0$ that
  \begin{equation}
    \label{eq:50}
    F(\tfrac{1}{\sigma^{2}} + \epsilon h)
    = \frac12 \sum_{k\geq 2} (-\epsilon)^{k} \sigma^{2(k-1)}  \int h'^{\,2} h^{k-2}.
  \end{equation}
  In other words
  \begin{equation}
    \label{eq:51}
    [D^{k}_{1/\rho} F(1/\rho)]\Big\vert_{\rho=\sigma^{2}}(h,\ldots,h)
    = (-1)^{k} k!\, \sigma^{2(k-1)} \,\frac12 \int h'^{\,2} h^{k-2}.
  \end{equation}
  The general derivative $[D^{k}_{1/\rho} F(1/\rho)]\big\vert_{\rho=\sigma^{2}}(h_{1},\ldots,h_{k})$ follows by polarisation of this identity and yields the first summand on the right-hand side \eqref{eq:40}.
\end{proof}

\appendix

\section{Disintegration of measures}
\label{sec:disintegration}

In this appendix we prove that for a given measure its disintegration into a continuous family of measures is unique. 
Although disintegration of measures is of course a classical topic, we were not able to locate a suitable result in the existing literature.

\begin{defn}\label{defMeasureCont}
Suppose $Y$ is a topological space.
We say that a family of Borel measures $\{\nu^y\}_{y\in Y}$ on a Polish space $X$ is continuous in $y$ if there exists a family $\left\{U_a\right\}_{a\in A}$ of open subsets of $X$ such that 
\begin{enumerate}
\item $\cup_{a\in A} U_a = X$;
\item For any $a\in A$ there exists $M>0$ such that for any $y\in Y$ we have $\nu_{y}(U_a) <M$.
\item For any $a\in A$ and any bounded continuous $F: X \to [0, \infty)$, such that $F(x)=0$ for any $x\notin U_a$, we have that the map $y\mapsto \int F\d\nu^y$ is continuous in $y$.
\end{enumerate}
\end{defn}

\begin{rmrk}
Since $X$ is Polish, we can assume that $A=\N$ by selecting a countable subcover of $\left\{U_a\right\}_{a\in A}$.
Note also that the definition above immediately implies that the measures $\{\nu^y\}_{y\in Y}$ are locally finite, so they automatically are Radon measures (since any locally finite measure on a Polish space is Radon).
We also remark that for probability measures one can often take $U_a = X$.
\end{rmrk}

The continuity property defined in Definition~\ref{defMeasureCont} is preserved under many different operations on measures.
Below we show that we can add continuous densities to the measures.
\begin{prp}\label{prop_measure_density_contin}
Let $Y$ be a topological space.
Suppose a family of Borel measures $\{\nu^y\}_{y\in Y}$ on a Polish space $X$ is continuous in $y$.
If $F:X\to[0, \infty)$ is a continuous function, then the family of Borel measures $\{F \cdot \nu^y\}_{y\in Y}$ is also continuous in $y$.
\end{prp}

\begin{proof}
For $n\in \N$ and $a\in A$ denote $U_{a, n} = \{x\in U_a: \, F(x) < n\}$.
Then it is easy to check that the family of open sets $\left\{U_{a, n}\right\}_{a\in A, n\in \N}$ satisfies the properties stated in the Definition~\ref{defMeasureCont} applied to the family of measures $\{F \cdot \nu^y\}_{y\in Y}$.
\end{proof}

We prove the uniqueness of measure disintegration in the class of continuous measures.
\begin{prp} \label{prop:disintegration}
  Let $X$ and $Y$ be Polish spaces and $\pi: X\to Y$ be a continuous projection.
  Let also $\nu$ and $\lambda$ be Radon measures on $X$ and $Y$ respectively.
  
  Suppose a family $\{\nu^y\}_{y\in Y}$ of Borel measures on $X$ satisfies the following.
  \begin{enumerate}
  \item For any $y\in Y$ the measure is supported on $\pi^{-1}(y)$;
  \item We have $\d\nu(\cdot) = \int \d\nu^y(\cdot) \, \d \lambda(y)$;
  \item The family of measures $\{\nu^y\}_{y\in Y}$ is continues in $y$ in the sense of Definition~\ref{defMeasureCont}. 
  \end{enumerate}
  Then the decomposition $\{\nu^y\}$ is unique.
\end{prp}

\begin{proof}
Let $\{U_n\}_{n\in\N}$ be a countable subcover of $\left\{U_a\right\}_{a\in A}$ from the Definition~\ref{defMeasureCont}. 
Also let $H_n:X\to [0, \infty)$ be a continuous function such that for all $x \in U_n$ we have $0<H_n(x)\leq 1$ and for any $x\notin U_n$ we have $H_n(x)=0$ (for example we can take $H_n(x) = \min\{ d_X(x, X\backslash U_n), 1\}$, where $d_X$ is a metric which turns the Polish space $X$ into a complete metric space).

Let $\delta_\eps$ be a continuous approximation of the Dirac delta function at $y_0$ as $\eps\to 0$ with respect to $\d\lambda$, i.e. for any continuous $g:Y\to \R$,
\begin{equation}
\lim_{\eps\to 0} \int_Y \delta_\eps(y) g(y) \d\lambda(y) = g(y_0).
\end{equation}
For example, we can take 
\begin{equation}
\delta_\eps(y) 
= 
\frac{\widetilde{\delta}_\eps(y)}{\int_Y \widetilde{\delta}_\eps(z)\d \lambda(z)},
\end{equation}
where 
$\widetilde{\delta}_\eps(y) = \max\left\{\eps - d_Y(y, y_0), 0\right\}$
with $d_Y$ being a metric which turns the Polish space $Y$ into a complete metric space.

Then consider a function $x\in X \mapsto \delta_{\eps}(\pi(x)) H_n(x) F(x)$ which is continuous on $X$ and vanishes outside $U_n$,
\begin{equation}
  \label{eq:60}
  \begin{aligned}
    \int \delta_{\epsilon}(\pi(x)) H_n(x) F(x) \d\nu(x)
    &= \int \d \lambda(y) \int \d\nu^{y}(x) \delta_{\epsilon}(\pi(x)) H_n(x) F(x)\\
    &= \int \d \lambda( y) \, \delta_{\epsilon}(y) \int \d\nu^{y}(x) H_n(x) F(x)
      \to  \int \d\nu^{y_{0}} H_n(x) F(x),
  \end{aligned}
\end{equation}
where in the second line we used the domain assumption $\pi(x) = y$ under $\d\nu^{y}$ and the fact that $y\mapsto \int H_n F\d\nu^y$ is continuous in $y$.

Suppose $\{\tilde{\nu}^y\}_{y\in Y}$ is another decomposition as in the statement of the proposition. 
Fix $y_0\in Y$.
Then for any continuous non-negative $F$ and any $n\in \N$ we have
\begin{equation}
\int \d\nu^{y_{0}} H_n(x) F(x) 
=
\int \d\tilde{\nu}^{y_{0}} H_n(x) F(x)
\end{equation}
Using that Radon measures are determined by their integrals with respect to nonnegative continuous functions
we obtain that $ H_n\d{\nu}^{y_{0}} = H_n \d\tilde{\nu}^{y_{0}} $, and thus $\d{\nu}^{y_{0}}|_{U_n} =  \d\tilde{\nu}^{y_{0}}|_{U_n}$.
Since $\cup_{n\in \N}U_n = X$, we deduce that $\d{\nu}^{y_{0}} =  \d\tilde{\nu}^{y_{0}}$.
\end{proof}

\section{Quotients of measures: Proof of Proposition~\ref{propMeasureFactor}}\label{sec:quotient-measure}

Suppose $(X,d)$ is a complete separable metric space (not necessarily locally compact).
Suppose a locally compact Hausdorff group $G$ is acting properly and continuously on $X$ from the right.
Then $X/G$ is a Polish space%
\footnote{Separability is clear.
  Using continuity of the projection $\pi\colon X\to X/G$ and paracompactness of $X$, one may check that $X/G$ is paracompact.
  Since $G$ acts properly, it also is Hausdorff.
  As a consequence, it is normal (Theorem of Jean Dieudonné).
  Moreover, since $\pi$ is an open map, $X/G$ is second countable and Urysohn's metrisation theorem implies that it is metrisable.
  Finally, complete metrisability follows from \citep{michael1986note}.}.
Let $\nu$ denote a left-invariant Haar measure on $G$ and $\Delta_{G}$ denote the modular function, such that $\nu(\cdot g) = \Delta_{G}(g) \nu(\cdot)$.
Write $\pi\colon X \to X/G$ for the canonical projection.
Write $C_{b}(X)$ for the space of continuous bounded functions equipped with the compact-open topology.
Denote by $C_{b}^{G-\mathrm{inv}}(X)$ the subspace of $G$-invariant functions equipped with the subspace topology.
Observe that $\pi^{\ast}\colon C_{b}(X/G) \to C_{b}^{G-\mathrm{inv}}(X), h\mapsto h\circ\pi$ is a bijection.

We say that a set $A \subseteq X$ is \emph{$G$-(pre)compact} if $G_{A} \coloneqq \{g \in G\colon A \cap Ag \neq \emptyset\}$ is (pre)compact in $G$.
  Furthermore, we say that a set $A$ is \emph{$G$-tempered} if it is $G$-precompact and moreover has a $G$-precompact open neighbourhood $U \supseteq A$, such that $UG \supseteq \mathrm{Cl}(A G)$.
Write $C_{b}^{G-\mathrm{temp}}(X)$ for the space of bounded continuous functions whose support is $G$-tempered (note that this space is not necessarily linear).
For $f \in C_{b}^{G-\mathrm{temp}}(X)$ write
\begin{equation}
  \label{eq:3}
  f^{\flat}(x) = \int_{G}f(xg)\d\nu(g).
\end{equation}
Note that $f^{\flat}$ is $G$-invariant (i.e.\ $f^{\flat}(\cdot g) = f^{\flat}$) and satisfies $[f(\cdot g)]^{\flat} = \Delta_{G}(g)^{-1} f^{\flat}$, for $g \in G$.
The main goal of this appendix is to prove the following

\begin{prp}\label{prop:existence-quotient-measure}
  Suppose $\mu$ is a Radon measure on $X$ such that $\mu(\cdot\, g) = \Delta_{G}(g)\, \mu$ for $g\in G$.
  Then there exists a unique Radon measure $\lambda$ on $X/G$ such that
  \begin{equation}
    \label{eq:2}
    \int_{X/G}f^{\flat}(x)\d\lambda(xG) = \int_{X}f(x)\d\mu(x),
  \end{equation}
  for any $f \in C_{b}^{G-\mathrm{temp}}(X)$. This extends to all non-negative $f\in C(X)$.
\end{prp}

The modular function makes an appearance in our condition for $G$-covariance of $\mu$, since we are working with a left-Haar measure $\nu$, but a right-action on $X$.
Requiring that $\mu(\cdot g) = \Delta_{G}(g)\, \mu$, tells us that the measure transforms like a left-Haar measure under the right action.
Also, note that in our application, the group $G=\SL(2,\R)$ is unimodular, hence $\Delta_{G} \equiv 1$, and the requirement simply reduces to $G$-invariance of $\mu$.

\medskip

We follow \cite[Section~2]{bourbaki_integration_2004}, with appropriate modifications to allow for $X$ not locally compact. %

\begin{lmm}
  For $f \in C_{b}^{G-\mathrm{temp}}(X)$, the function $f^{\flat}$ is well-defined as an element of $C_{b}^{G-\mathrm{inv}}(X)$.
\end{lmm}

\begin{proof}
  Obviously, we have $|f(xg)| \leq \|f\|_{\infty} \mathbf{1}_{\mathrm{spt}(g\mapsto f(xg))}$.
  Moreover, note that $\mathrm{spt}(g\mapsto f(xg))$ is contained in a $G$-translate of $G_{\mathrm{spt}{f}}$, hence $f^{\flat}(x) \leq \|f\|_{\infty} \nu(\mathrm{spt}(g\mapsto f(xg))) \leq \|f\|_{\infty} \nu(G_{\mathrm{spt}{f}})$.
  Consequently, $f^{\flat}$ is well-defined and uniformly bounded.
  $G$-invariance is clear from invariance of the Haar-measure $\nu$.
  Regarding continuity, suppose $x_{n} \to x$ in $X$.
  Write $\mathcal{X} \coloneqq \{x\} \cup \{x_{n}\}_{n\in \N}$.
  Let $U \supseteq \mathrm{spt}(f)$ denote a $G$-precompact open neighbourhood of $\mathrm{spt}(f)$ such that $UG \supseteq \mathrm{Cl}(\mathrm{spt}(f)G)$.
  We may assume $x \subseteq UG$.
  In fact, otherwise $f^{\flat}(x_{n}) = 0 = f^{\flat}(x)$ for sufficiently large $n$ and continuity is clear.
  Moreover, choosing $g_{0} \in G$ such that $x \in U g_{0}$, we can assume without loss of generality that $\mathcal{X} \subseteq U g_{0}$.
  Now consider
  \begin{equation}
    \label{eq:13}
    |f^{\flat}(x) - f^{\flat}(x_{n})|
    \leq
    \int_{G} |f(xg) - f(x_{n}g)| \d{\nu}(g).
  \end{equation}
  Since $f$ is bounded, it suffices to show that we can restrict the integral to a compact subset of $G$ (independently from $n$).
  Note for $g \in G$ we have $|f(xg) - f(x_{n}g)| \neq 0$ only if $U \cap \mathcal{X}g \neq \emptyset$.
  Since $\mathcal{X}g \subseteq Ug_{0}g$, the set of such $g$ is precompact.
  Consequently, the right-hand side of \eqref{eq:13} goes to zero as $n \to \infty$ by dominated convergence, proving continuity of $f^{\flat}$.
\end{proof}

\begin{lmm}\label{lem:g-precompact-neighbourhood}
  For any $x \in X$, there exists an open neighbourhood $U \ni x$ which is $G$-tempered.
\end{lmm}
\begin{proof}
  First we show that there exist arbitrarily small $G$-compact neighbourhoods:
  Consider a decreasing sequence of neighbourhoods $\{U_{n}\}_{n\in\N}$, such that $\bigcap_{n}U_{n} = \{x\}$.
  Assume that none of the $U_{n}$ is $G$-precompact, hence there exists a sequence $x_{n} \in U_{n}$ and $g_{n} \in G$, such that $x_{n}g_{n} \in U_{n}$, but with $\{g_{n}\}$ non-precompact.
  However, by definition of the $U_{n}$, we have $x_{n} \to x$ and $x_{n}g_{n} \to x$, and since $G$ acts properly on $X$ this implies that $g_{n}$ is precompact, leading to a contradiction.

  Now suppose $\tilde{U}$ is a $G$-compact neighbourhood of $x$.
  Then $\tilde{U}G$ is a neighbourhood of $xG$ in $X/G$.
  Since $X/G$ is metrisable, we may choose a sufficiently small ball $\mathcal{B}$ around $[x]$ in $X/G$, such that $\mathrm{Cl}(\mathcal{B}) \subseteq \tilde{U}G$.
  Set $U \coloneqq \tilde{U} \cap \pi^{-1}(\mathcal{B})$.
  This is a $G$-tempered neighbourhood of $x$.
\end{proof}

\begin{lmm}\label{lem:bump-function}
    For any $x_{0} \in X$ and $\epsilon > 0$ sufficiently small, there exists $0< \epsilon' < \epsilon$ and a non-negative function $v \in C^{G-\mathrm{temp}}_{b}(X)$ with $\mathrm{spt}(v) \subseteq B_{\epsilon}(x_{0})$, such that
  \begin{equation}
    \label{eq:14}
    0 < \inf_{x' \in B_{\epsilon'}(x_{0})} v^{\flat}(x') < \sup_{x' \in B_{\epsilon'}(x_{0})} v^{\flat}(x') < \infty.
  \end{equation}
\end{lmm}
\begin{proof}
  It suffices to show the following:
  for any $x_{0} \in X$ and $\epsilon > 0$ sufficiently small, there exists $0 < \epsilon' < \epsilon$, such that
  \begin{equation}
    \label{eq:12}
    0 < \inf_{x'\in B_{\epsilon'}(x_{0})} (\mathbf{1}_{B_{\epsilon}(x_{0})}^{\flat})(x') < \sup_{x'\in B_{\epsilon'}(x_{0})} (\mathbf{1}_{B_{\epsilon}(x_{0})}^{\flat})(x') < \infty,
  \end{equation}
  where we note that $(\mathbf{1}_{B_{\epsilon}(x_{0})}^{\flat})(x') = \nu(\{g\in G\colon x'g \cap B_{\epsilon}(x_{0}) \neq \emptyset\})$.
  Indeed if above holds then we can choose $v = \max(0, 1-\tfrac{4}{3\epsilon}\mathrm{dist}(x, B_{\epsilon/2}(x_{0})))$.
  Since $\tfrac13 \mathbf{1}_{B_{\epsilon/2}} \leq v\leq \mathbf{1}_{B_{\epsilon}}$, for small enough $\epsilon > 0$ the claim in \eqref{eq:14} follows.

  Now we address the claim \eqref{eq:12}.
  By continuity of the group action $X\times G \to X, (x,g) \mapsto xg$, the preimage of $B_{\epsilon}(x_{0})$ under this map is an open neighbourhood of $(x_{0}, \mathrm{id})$.
  In particular, it contains a neighbourhood of the form $B_{\epsilon'}(x_{0}) \times U$, with $U \subseteq G$ a neighbourhood of $\mathrm{id}$.
  In particular, for $x' \in B_{\epsilon'}(x_{0})$ we have $U \subseteq \{g\in G\colon x'g \cap B_{\epsilon}(x_{0}) \neq \emptyset\}$ and hence $(\mathbf{1}_{B_{\epsilon}(x_{0})}^{\flat})(x') \geq \nu(U) > 0$.
  This proves the lower bound in \eqref{eq:12}.
  For the upper bound note that we may assume $\epsilon > 0$ sufficiently small, such that $B_{\epsilon}(x_{0})$ is $G$-tempered.
  As a consequence $\{g\in G\colon x'g \cap B_{\epsilon}(x_{0}) \neq \emptyset\} \subseteq G_{B_{\epsilon}(x_{0})}$ for $x' \in B_{\epsilon}(x_{0})$, so $(\mathbf{1}_{B_{\epsilon}(x_{0})}^{\flat})(x') \leq \nu(G_{B_{\epsilon}(x_{0})}) < \infty$, which proves the upper bound in \eqref{eq:12}.
\end{proof}

\begin{lmm}\label{lem:partition-of-unity}
  Consider above setting and consider a Radon measure $\mu$ on $X$.
  There exists a countable family of non-negative functions $u_{i} \in C_{b}^{G-\mathrm{temp}}(X)$, $i \in \N$, with bounded support, such that $\mu(u_{i}) < \infty$ and such that $\{u_{i}^{\flat}\}$ is a partition of unity.

  Moreover, writing $X_{i} \coloneqq \mathrm{spt}(u_{i})\cdot G$, the map $f\mapsto f^{\flat}$ is surjective from $C^{G-\mathrm{temp}}(X_{i})$ onto $C^{G-\mathrm{inv}}(X_{i})$.
\end{lmm}

\begin{proof}
  Recall that by assumption $X/G$ is Polish, hence Hausdorff and paracompact.
  This implies that any open cover of $X/G$ admits a subordinate partition of unity.
  We choose $\{x_{i}\}_{i\in\N} \subseteq X$ and $\epsilon_{i} > 0$, with $\{B_{\epsilon_{i}}(x_{i})\}_{i\in\N}$ covering the whole space, such that $B_{\epsilon_{i}}(x_{i})$ is $G$-tempered and $\mu(B_{\epsilon_{i}}(x_{i})) < \infty$ (which can be done by Lemma~\ref{lem:g-precompact-neighbourhood} and since $\mu$ is Radon).
  Furthermore, choose $\epsilon_{i}$ sufficiently small, such that the conclusion of Lemma~\ref{lem:bump-function} holds true:
  There are $0< \epsilon_{i}' < \epsilon_{i}$ and non-negative functions $v_{i} \in C^{G-\mathrm{temp}}_{b}(X)$ supported in $B_{\epsilon_{i}}(x_{i})$, such that $v_{i}^{\flat}$ is uniformly bounded away from zero and infinity on $B_{\epsilon_{i}'}(x_{i})$.
  
  The saturated balls $\{B_{\epsilon_{i}'}(x_{i})G\}$ form an open cover of the quotient space $X/G$, hence admit a subordinate partition of unity $\{\bar{u}_{i}\}_{i\in\N}$, with $\bar{u}_{i} \in C_{b}(X/G)$ and $\mathrm{spt}(\bar{u}_{i}) \subseteq B_{\epsilon_{i}'}(x_{i})G$.
  The lifts $\pi^{\ast}\bar{u}_{i} \in C^{G-\mathrm{inv}}_{b}(X)$ along $\pi\colon X\to X/G$ form a partition of unity on $X$.
  Define the functions
  \begin{equation}
    \label{eq:15}
    u_{i} =
    \begin{cases}
      \tfrac{v_{i}}{v_{i}^{\flat}}\, \pi^{\ast}\bar{u}_{i} \text{ on } B_{\epsilon_{i}'}(x_{i})G\\
      0 \text{ otherwise}.
    \end{cases}
  \end{equation}
  Notice that $u_{i} \geq 0$ is a well-defined bounded continuous function, is supported in $B_{\epsilon_{i}}(x_{i})$, and satisfies $u_{i}^{\flat} = \pi^{\ast}\bar{u}_{i}$.
  This concludes the construction of the $u_{i}$.
  Moreover, by the same construction as in \eqref{eq:15}, we see the surjectivity of $f\mapsto f^{\flat}$ from $C^{G-\mathrm{temp}}(X_{i})$ onto $C^{G-\mathrm{inv}}(X_{i})$.
\end{proof}

\begin{proof}[Proof of Proposition~\ref{prop:existence-quotient-measure}]
  In the following we write $\mu(f) = \int_{X}f \d{\mu}$.
  We first note that the condition $\mu(\cdot g) = \Delta_{G}(g)\, \mu$ implies that $\mu(f_{1} f_{2}^{\flat}) = \mu(f_{1}^{\flat} f_{2})$ for any $f_{1}, f_{2} \in C^{G-\mathrm{temp}}_{b}(X)$.
  Indeed,
  \begin{equation}
    \label{eq:16}
    \begin{aligned}
      \mu(f_{1} f_{2}^{\flat})
      &= \int_{X}\d{\mu}(x) f_{1}(x) \int_{G}\d{\nu}(g)f_2(xg)\\
      &= \int_{G}\d{\nu}(g) \int_{X}\d{\mu}(xg)\Delta_{G}(g)^{-1} f_{1}(x) f_{2}(xg)\\
      &= \int_{G}\d{\nu}(g) \Delta_{G}(g)^{-1} \int_{X}\d{\mu}(x) f_{1}(xg^{-1}) f_{2}(x)\\
      &= \int_{X}\d{\mu}(x) \Big[ \int_{G}\d{\nu}(g) \Delta_{G}(g)^{-1} f_{1}(xg^{-1})\Big] f_{2}(x)\\
      &= \int_{X}\d{\mu}(x) \Big[ \int_{G}\d{\nu}(g^{-1}) f_{1}(xg^{-1})\Big] f_{2}(x).\\
      &= \mu(f_{1}^{\flat} f_{2}),
    \end{aligned}
  \end{equation}
  where in second to last line we used that $\d{\nu}(g) \Delta_{G}(g)^{-1} = \d{\nu}(g^{-1})$.

  Consider $\{u_{i}\}$ as in Lemma~\ref{lem:partition-of-unity} and write $X_{i} \coloneqq \mathrm{spt}(u_{i})\cdot G$ for the saturation of the support.
  For any $i\in\N$, define a linear functional $I_{i}\colon C_{b}^{G-\mathrm{inv}}(X_i) \to \R$ by
  \begin{equation}
    \label{eq:10}
    I_{i}(h) = \mu(u_{i}h) \text{ for } \text{ for } h \in C_{b}^{G-\mathrm{inv}}(X_i).
  \end{equation}
  By \eqref{eq:16} we have $I_{i}(f^{\flat}) = \mu(u_{i}^{\flat} f)$ for any $f \in C^{G-\mathrm{temp}}_{b}(X_{i})$.
  Note that this characterises $I_{i}$ due to the surjectivity of $f\mapsto f^{\flat}$ from $C_{b}^{G-\mathrm{temp}}(X_{i})$ onto $C_{b}^{G-\mathrm{inv}}(X_{i})$.
  Each $I_{i}$ is monotone, since $\mu u_{i}$ is a positive measure, so $I_{i}$ is continuous with respect to the compact-open topology on $C^{G-\mathrm{inv}}_{b}(X_i)$.
  Considering the push-forward along $\pi\colon X_{i}\to X_{i}/G$, we get $\tilde{I}_{i} \coloneqq \pi_{\ast} I_{i} \in C_{b}(X_{i}/G)^{\ast}$.
  Continuity of $\tilde{I}_{i}$ follows from continuity of $\pi$ and definition of the compact-open topologies.
  
  Since $\mu u_{i}$ is a finite Radon measure, for any $\epsilon > 0$ there exists a compact set $K_{\epsilon} \subseteq X_{i}$ such that $\mu(u_i h) \leq \epsilon \|h\|_{\infty}$ for $h\in C_{b}^{G-\mathrm{inv}}(X_{i})$ with $h\vert_{K_{\epsilon}G} \equiv 0$.
  Consequently, application of a variant of the Riesz–Markov–Kakutani representation theorem \cite[Theorem~7.10.6]{BogachevMeasure} to $\tilde{I}_{i}$ and the fact that $C_{b}(X_{i}/G) \cong C_{b}^{G-\mathrm{inv}}(X_{i})$ imply that there exists a unique finite Radon measure $\lambda_{i}$ on $X_{i}/G$, such that
  \begin{equation}
    \label{eq:11}
    \int_{X_{i}/G}f^{\flat}(x)\d\lambda_{i}(xG) = \int_{X_{i}}f(x)u_{i}^{\flat}(x)\d\mu(x) \text{ for } f\in C^{G-\mathrm{temp}}_{b}(X_{i}).
  \end{equation}
  In the following, also write $\lambda_{i}$ for the push-forward along the inclusion $X_{i} \hookrightarrow X$.
  Since $\sum_{i}u_{i}^{\flat} = 1$, we define $\lambda = \sum_{i}\lambda_{i}$ and note that it satisfies \eqref{eq:2} (using monotone convergence and positivity of $f \mapsto f^{\flat}$).
  This defines a locally finite Borel measure and since $X/G$ is strongly Radon (as a completely separable metric space), $\lambda$ is a Radon measure.
  Regarding uniqueness, note that for any $\lambda$ satisfying \eqref{eq:2}, we have $\lambda \cdot u_{i}^{\flat} = \lambda_{i}$.
  In particular, any other candidate for $\lambda$ constructed using a different partition of unity will therefore agree, which implies that the so constructed measure is unique.
\medskip

  The extension from $f\in C_b^{G-\mathrm{temp}}(X)$ to non-negative $f\in C(X)$ follows from the existence of a partition of unity on $X$, consisting of functions with $G$-tempered supports, and monotone convergence.
\end{proof}

\section*{Acknowledgements}
We would like to thank James Norris for many helpful discussions and careful proofreading.

\bibliography{Schwarzian}

\end{document}